\documentclass[11pt]{amsart}
\usepackage{amsfonts}
\usepackage{amsmath}
\usepackage{amsthm}
\usepackage{url}
\usepackage[all]{xy}
\usepackage{latexsym}
\usepackage{amssymb}
\usepackage[cp850]{inputenc}
\usepackage{amsfonts}
\usepackage{graphicx}
\usepackage{dsfont}
\usepackage{float}
\usepackage{stackrel}
\usepackage{labelfig}

\usepackage{tkz-euclide}
\usepackage[font=small,labelfont=bf]{caption}
\usepackage{tikz}
 \usetikzlibrary{cd}
\usetikzlibrary{decorations.markings}

\newtheorem{sat}{Theorem}[section]		
\newtheorem{lem}[sat]{Lemma}
\newtheorem{kor}[sat]{Corollary}			
\newtheorem{prop}[sat]{Proposition}

\newtheorem*{defi*}{Definition}			
\newtheorem*{bei*}{Example}
\newtheorem*{sat*}{Theorem}				
\newtheorem*{kor*}{Corollary}
\newtheorem*{rmk*}{Remark}				
	
\newtheorem*{quest*}{Question}	
\newtheorem*{fact}{Fact}	
\newtheorem{claim}{Claim}	
\newtheorem*{claim*}{Claim}

\let\ssection=\section
\renewcommand{\section}{\setcounter{equation}{0}\ssection}

\newtheorem*{namedtheorem}{\theoremname}
\newcommand{\theoremname}{testing}
\newenvironment{named}[1]{\renewcommand{\theoremname}{#1}\begin{namedtheorem}}{\end{namedtheorem}}

\theoremstyle{remark}
\newtheorem*{bem}{Remark}

\newtheorem*{namedtheoremr}{\theoremnamer}
\newcommand{\theoremnamer}{testing}

			\newcommand{\BF}{\mathbb F}
			\newcommand{\BH}{\mathbb H}
\newcommand{\BI}{\mathbb I}			
			
			\newcommand{\BN}{\mathbb N}
			
			\newcommand{\BR}{\mathbb R}
\newcommand{\BS}{\mathbb S}			\newcommand{\BT}{\mathbb T}

			\newcommand{\BZ}{\mathbb Z}

\newcommand{\BFA}{\mathbf A} 		\newcommand{\BFB}{\mathbf B} 
\newcommand{\BFC}{\mathbf C} 		 
 		 
\newcommand{\BFG}{\mathbf G}

 		\newcommand{\BFP}{\mathbf P}

 		\newcommand{\BFX}{\mathbf X}

\newcommand{\CG}{\mathcal G}			
			\newcommand{\CJ}{\mathcal J}
			\newcommand{\CL}{\mathcal L}
		
\newcommand{\CO}{\mathcal O}			
			
			\newcommand{\CT}{\mathcal T}

			\newcommand{\CZ}{\mathcal Z}

\newcommand{\actson}{\curvearrowright}
\newcommand{\D}{\partial}

\newcommand{\bs}{\backslash}

\DeclareMathOperator{\SL}{SL}		
\DeclareMathOperator{\PSL}{PSL}		
\DeclareMathOperator{\Id}{Id}		
\DeclareMathOperator{\vol}{vol}		

\DeclareMathOperator{\diam}{diam}

\newcommand{\comment}[1]{}

\DeclareMathOperator{\SO}{SO}

\DeclareMathOperator{\Aut}{Aut}

\DeclareMathOperator{\bconst}{\mathbf{const}}

\DeclareMathOperator{\fold}{Fold}

\DeclareMathOperator{\Surf}{Surf}

\DeclareMathOperator{\ver}{\mathbf{vert}}
\DeclareMathOperator{\edg}{\mathbf{edge}}
\DeclareMathOperator{\hal}{\mathbf{half}}
\DeclareMathOperator{\neigh}{\mathbf{neigh}}
\DeclareMathOperator{\wire}{\mathbf{wire}}

\DeclareMathOperator{\crit}{crit}
\DeclareMathOperator{\spine}{spine}
\DeclareMathOperator{\imm}{imm}

\newcommand{\fsubd}{\mathrel{{\scriptstyle\searrow}\kern-1ex^d\kern0.5ex}}
\newcommand{\bsubd}{\mathrel{{\scriptstyle\swarrow}\kern-1.6ex^d\kern0.8ex}}

\renewcommand{\epsilon}{\varepsilon}
\renewcommand{\le}{\leqslant}
\renewcommand{\ge}{\geqslant}
\renewcommand{\emptyset}{\varnothing}

\begin{document}

\title[]{Counting geodesics of given commutator length}
  \author{Viveka Erlandsson}
  \address{School of Mathematics, University of Bristol \\ Bristol BS8 1UG, UK {\rm and}  \newline ${ }$ \hspace{0.2cm} Department of Mathematics and Statistics, UiT The Arctic University of  \newline ${ }$ \hspace{0.2cm} Norway}
  \email{v.erlandsson@bristol.ac.uk}
\thanks{The first author gratefully acknowledges support from EPSRC grant EP/T015926/1 and UiT Aurora Center for Mathematical Structures in Computations (MASCOT). No data was used in this research.}
\author{Juan Souto}
\address{UNIV RENNES, CNRS, IRMAR - UMR 6625, F-35000 RENNES, FRANCE}
\email{jsoutoc@gmail.com}

\begin{abstract}
Let $\Sigma$ be a closed hyperbolic surface. We study, for fixed $g$, the asymptotics of the number of those periodic geodesics in $\Sigma$ having at most length $L$ and which can be written as the product of $g$ commutators. The basic idea is to reduce these results to being able to count critical realizations of trivalent graphs in $\Sigma$. In the appendix we use the same strategy to give a proof of Huber's geometric prime number theorem.
\end{abstract}
\maketitle

\section{Introduction}
In this paper we will be interested in studying the growth of the number of certain classes of geodesics in a closed, connected and oriented hyperbolic surface $\Sigma=\Gamma\bs\BH^2$, which we assume to be fixed throughout the paper. 

In \cite{Huber} Huber proved that the cardinality of the set $\BFC(L)$ of non-trivial oriented periodic geodesics $\gamma\subset\Sigma$ of length $\ell_\Sigma(\gamma)\le L$ behaves like 
\begin{equation}\label{eq huber}
\vert\BFC(L)\vert\sim \frac{e^L}L\text{ as }L\to\infty.
\end{equation}
Here $\sim$ means that the ratio between both quantities tends to $1$ as $L\to\infty$. Huber's result has been generalized in all possible directions, allowing for example for surfaces of finite area \cite{Hejhal} and for higher dimensional negatively curved manifolds or replacing periodic geodesic by orbits of Anosov flows \cite{Margulis-thesis}. Variations of Huber's result have also been obtained when one imposes additional restrictions on the geodesics one is counting. Here one should mention Mirzakhani's work on counting simple geodesics \cite{Maryam1} or more generally geodesics of given type \cite{Maryam2,book}, but what is closer to what we will care about in this paper are earlier results of Katsuda-Sunanda \cite{Katsuda-Sunada} and Phillips-Sarnak \cite{Phillips-Sarnak} giving the asymptotic behavior of the number of geodesics with length at most $L$ and satisfying some homological condition. For example,  Katsuda-Sunada \cite{Katsuda-Sunada} proved that 
\begin{equation}\label{eq homology huber}
\vert\{\gamma\in\BFC(L)\text{ homologically trivial}\}\vert\sim (g_\Sigma-1)^{g_\Sigma}\cdot\frac{e^L}{L^{g_\Sigma+1}}
\end{equation}
where $g_\Sigma$ is the genus of $\Sigma$. Phillips-Sarnak \cite{Phillips-Sarnak} get this same result with an estimate on the error term. 

Here we will be counting certain kinds of homologically trivial curves. Every homologically trivial curve $\gamma$ can be represented, up to free homotopy, as the product of commutators in $\pi_1(\Sigma)$. The smallest number of commutators needed is the {\em commutator length} of $\gamma$. This quantity agrees with the genus of the smallest connected oriented surface $S$ with connected boundary $\D S$ for which there is a (continuous) map $S\to\Sigma$ sending $\D S$ to $\gamma$. This explains why we think of the commutator length $cl(\gamma)$ of $\gamma$ as the {\em genus of $\gamma$}.

In this paper we study, for fixed but otherwise arbitrary $g\ge 1$, the asymptotic behavior of the cardinality of the set 
\begin{equation}\label{eq set we want to count}
\BFB_{g}(L)=\left\{\gamma\subset\Sigma\   \middle\vert \begin{array}{l} \text{closed geodesic with }
\ell_\Sigma(\gamma)\le L\text{ and}\\ \text{commutator length }cl(\gamma)=g
\end{array}  \right\}
\end{equation}
This is our main result:

\begin{sat}\label{thm counting curves}
Let $\Sigma$ be a closed, connected, and oriented hyperbolic surface and for $g\ge 1$ and $L>0$ let $\BFB_g(L)$ be as in \eqref{eq set we want to count}. We have 
$$\vert\BFB_g(L)\vert\sim \frac {2}{12^g\cdot g!\cdot(3g-2)!\cdot\vol(T^1\Sigma)^{2g-1}}\cdot L^{6g-4}\cdot e^{\frac L2}$$
as $L\to\infty$.
\end{sat}

Here, as we will throughout the paper, we have endowed the unit tangent bundle $T^1\Sigma$ with the Liouville measure, normalized in such a way that $\vol(T^1\Sigma)=2\pi\cdot\vol(\Sigma)=-4\pi^2\chi(\Sigma)$. In particular we have that, as in \eqref{eq homology huber}, the quantities in Theorem \ref{thm counting curves} depend on the topology of the underlying surface $\Sigma$ but there is no dependence on its geometry.
\medskip

Note that by definition the curves in $\BFB_g(L)$ bound a surface of genus $g$, but that this surface is just the image of a continuous, or if you want smooth, map. We prove that only a small but definite proportion of the elements in $\BFB_g(L)$ arise as the boundary of an immersed surface of genus $g$ (and connected boundary):

\begin{sat}\label{thm counting curves immersion}
Let $\Sigma$ be a closed, connected, and oriented hyperbolic surface and for $g\ge 1$ and $L>0$ let $\BFB_g(L)$ be as in \eqref{eq set we want to count}. We have 
$$\vert\{\gamma\in\BFB_g(L)\text{ bounds immersed surface of genus }g\}\vert\sim \frac 1{2^{4g-2}}\vert\BFB_g(L)\vert$$
as $L\to\infty$.
\end{sat}

Theorem \ref{thm counting curves immersion} should perhaps be compared with the Immersion Theorem in \cite[Theorem 4.79]{scl}. The Immersion Theorem asserts that every homologically trivial geodesic $\gamma$ in $\Sigma$ virtually bounds an immersed surface, meaning that there is an immersion into $\Sigma$ of a compact surface $S$ in such a way that the boundary of $S$ is mapped onto $\gamma$ with positive degree. In some sense Theorem \ref{thm counting curves immersion} seems to suggest that, most of the time, the genus of $S$ does not agree with the genus of $\gamma$.

In the course of the proof of Theorem \ref{thm counting curves} we will need a different counting result that we believe has its own interest. Suppose namely that $X$ is a compact trivalent graph. Under a {\rm critical realization} of $X$ in $\Sigma$ we understand a (continuous) map
$$\phi:X\to\Sigma$$
sending each edge to a non-degenerate geodesic segment in such a way that any two (germs of) edges incident to the same vertex are sent to geodesic segments meeting at angle $\frac{2\pi}3$ (see Lemma \ref{lem component critical realization} for an explanation of the etymology of this terminology). Although it may not be completely evident to the reader at this point, the set 
\begin{equation}\label{eq set of critical realizations}
\BFG^X(L)=\left\{\begin{array}{c}\phi:X\to\Sigma\text{ critical realization}\\ \text{with length }\ell_\Sigma(\phi)\le L\end{array}\right\}
\end{equation}
is finite, where the {\em length} of a critical realization is defined to be the sum
$$\ell_\Sigma(\phi)=\sum_{e\in\edg(X)}\ell_\Sigma(\phi(e))$$
of the lengths of the geodesic segments $\phi(e)$ when $e$ ranges over the set of edges of $X$. The following is the key to Theorem \ref{thm counting curves}:

\begin{sat}\label{thm counting minimising graphs}
Let $\Sigma$ be a closed, connected, and oriented hyperbolic surface. For every connected trivalent graph $X$ we have
$$\vert\BFG^X(L)\vert\sim 
\left(\frac 23\right)^{3\chi(X)}\cdot\frac{\vol(T^1\Sigma)^{\chi(X)}}{(-3\chi(X)-1)!}\cdot L^{-3\chi(X)-1}\cdot e^{L}$$
as $L\to\infty$. Here $\chi(X)$ is the Euler-characteristic of the graph $X$.
\end{sat}

Let us sketch the proof of Theorem \ref{thm counting minimising graphs}. Recall first that there are basically two approaches (that we know of) to establish Huber's theorem \eqref{eq huber}: either one approaches it \`a la Huber \cite{Huber}, that is from the point of spectral analysis or, as Margulis did later \cite{Margulis-thesis}, exploiting the ergodic properties of the geodesic flow. Indeed, already the predecessor to Huber's theorem, namely Delsarte's lattice point counting theorem \cite{Delsarte} can be approached from these two different points of view. In Section \ref{sec delsarte} below we will sketch the argument to derive Delsarte's theorem from the fact that the geodesic flow is mixing, discussing also the count of those geodesic arcs in $\Sigma$ going from $x$ to $y$ and whose initial and terminal speed are in predetermined sectors of the respective unit tangent spaces---see Theorem \ref{thm Delsarte's theorem for sectors} for the precise statement. This is indeed all the dynamics we will need in the proof of the theorems above. To be clear: we do not make use of any of the slightly rarified refinements of the mixing property of the geodesic flow. Specifically, we do not need exponential mixing or such.

The basic idea of the proof of Theorem \ref{thm counting minimising graphs} is to note that, for a fixed graph $X$, the set of all its realizations in $\Sigma$, that is maps $X\to\Sigma$ mapping each edge geodesically, is naturally a manifold. Generically, each connected component contains a unique critical realization. To count how many such critical realizations there are with total length less than $L$ we consider for small $\epsilon$ the set $\CG_{\epsilon-\crit}(L)$ of geodesic realizations of length at most $L$ and where the angles, instead of being $\frac{2\pi}3$ on the nose, are in the interval of size $\epsilon$ around that number. Delsarte's theorem allows us to compute the volume $\vol(\CG_{\epsilon-\crit}(L))$ of that set of geodesic realizations. A little bit of hyperbolic geometry then shows that most of the connected components of $\CG_{\epsilon-\crit}(L)$ have basically the same volume. We thus know, with a small error, how many connected components we have, and thus how many critical realizations. This concludes the sketch of the proof of Theorem \ref{thm counting minimising graphs}.

\begin{bem}
The strategy used to prove Theorem \ref{thm counting minimising graphs} can be used to give a pretty easy proof of Huber's theorem \eqref{eq huber}. We work this out in the appendix, and although there is no logical need of doing so, we encourage the reader to have a look at it before working out the details of the proof of Theorem \ref{thm counting minimising graphs}.
\end{bem}

Let us also sketch the proof of Theorem \ref{thm counting curves}. A {\em fat graph} is basically a graph which comes equipped with a regular neighborhood $\neigh(X)$ homeomorphic to an oriented surface. Such a fat graph {\em has genus $g$} if this regular neighborhood is homeomorphic to a compact surface of genus $g$ with connected boundary. The point of considering fat graphs is that whenever we have a realization $\phi:X\to\Sigma$ of the graph underlying a genus $g$ fat graph then we get a curve, namely $\phi(\D\neigh(X))$ which has at most genus $g$. The basic idea of the proof of Theorem \ref{thm counting curves} is to consider the set 
$$\BFX_{g}=\left\{(X,\phi)\middle\vert\begin{array}{c}X\text{ is a fat graph of genus }g\text{ and}\\
\phi:X\to\Sigma\text{ is a critical realization}\\ \text{of the underlying graph}\end{array}\middle\}\right/_{\text{equiv}}$$
of (equivalence classes) of realizations (see Section \ref{sec main} for details) and prove that the map
$$\Lambda:\BFX_{g}\to\BFC,\ \ (X,\phi)\mapsto\text{ geodesic homotopic to }\phi(\D X)$$
is basically bijective onto $\BFB_g$ and that generically, the geodesic $\Lambda(X,\phi)$ has length almost exactly equal to $2\cdot\ell(\phi)-C$ for some explicit constant $C$. Once we are here we get the statement of Theorem \ref{thm counting curves} from Theorem \ref{thm counting minimising graphs} together with a result by Bacher and Vdovina \cite{Bacher-Vdovina}. Other than proving Theorem \ref{thm counting minimising graphs}, the bulk of the work is to establish the properties of the map $\Lambda$. The key step is to bound the number of curves with at most length $L$ and which arise in two essentially different ways as the boundary of genus $g$ surfaces:

\begin{sat}\label{bounding multi-fillings}
For any $g$ there are at most $\bconst\cdot L^{6g-5}\cdot e^{\frac L2}$ genus $g$ closed geodesics $\gamma$ in $\Sigma$ with length $\ell(\gamma)\le L$ and with the property that there are two non-homotopic fillings $\beta_1:S_1\to\Sigma$ and $\beta_2:S_2\to\Sigma$ of genus $\le g$.
\end{sat}

Here, a genus $g$ filling of $\gamma$ is a continuous map $\beta: S\to\Sigma$ from a genus $g$ surface $S$ with connected boundary such that $\beta(\partial S) = \gamma$. 

Theorem \ref{bounding multi-fillings} may look kind of weak because we are only bounding by $\frac\bconst L$ the proportion of those elements in $\BFB_g(L)$ that we are double counting when we count surfaces of genus $g$ instead of counting curves. It should however be noted that this is the order of the error term in \eqref{eq homology huber} (see \cite{Phillips-Sarnak}), and this is indeed the order of error term that we expect in the results we prove here. For what it is worth, it is not hard to show that the set of those $\gamma$ in Theorem \ref{bounding multi-fillings} is at least of the order of $\bconst\cdot L^{6g-10}\cdot e^{\frac L2}$. Indeed, if $\omega\in\pi_1(\Sigma)$ arises in two ways as a commutator, for example if we choose $\omega=[aabab,ba^{-1}a^{-1}]=[aaba,bba^{-1}]$, and if $\eta$ is a randomly chosen product of $g-1$ generators then $\omega\eta$ arises in two different ways as a product of $g$ commutators and hence admits two non-homotopic fillings, and there are $\bconst\cdot L^{6g-10}\cdot e^{\frac L2}$ many choices for $\eta$.

\subsection*{Section-by-section summary}
In Section \ref{sec realizations} we discuss realizations of graphs and the topology and geometry of spaces of realizations.

In Section \ref{sec delsarte} we discuss Delsarte's classical lattice point counting result and a couple of minimal generalizations thereof. 

At this point we will have all tools needed to prove Theorem \ref{thm counting minimising graphs}. This is done in Section \ref{sec counting critical}.

In Section \ref{sec fillings} and Section \ref{sec fillings2} we work out the geometric aspects of the proof of Theorem \ref{thm counting curves}, the main result being Theorem \ref{bounding multi-fillings}. Although we do not use any results about pleated surfaces, some of the arguments in these two sections will come pretty natural to those readers used to working with such objects.

In Section \ref{sec main} we prove Theorem \ref{thm counting curves} combining the results of the previous two sections with Theorem \ref{thm counting minimising graphs}. Theorem \ref{thm counting curves immersion} is proved in Section \ref{sec immersed}.

Finally, Section \ref{sec comments} is dedicated to discussing what parts of what we do hold if we soften the assumption that $\Sigma$ is a closed hyperbolic surface.

To conclude, we present in Appendix \ref{sec huber} a proof of Huber's theorem using the same idea as in the proof of Theorem \ref{thm counting minimising graphs}.
\medskip

\begin{bem}
After conclusion of this paper we learned that the problem of calculating how many elements arise as commutators in some group has also been treated in other cases. More concretely, Park \cite{Park} uses Wicks forms to get asymptotics, as $L\to\infty$, for the number of elements in the free group $\BF_r$ of rank $r$ which arise as a commutator and have word-length $L$. In the same paper he also treats the case that the group is a free product of two non-trivial finite groups. It would be interesting to figure out if it is possible to apply the methods here to recover Park's beautiful theorems. 
\end{bem}

\subsection*{Acknowledgements}
We thank Vincent Delecroix, Lars Louder, Michael Magee, and Peter Sarnak for their very helpful comments. The second author is also very grateful for the patience shown by S\'ebastien Gou\"ezel, Vincent Guirardel and Fran\c cois Maucourant with all his questions---it is nice to have one's office next to theirs.
\bigskip

\newpage

\centerline{\bf Notation.}

\smallskip

Under a graph $X$ we understand a 1-dimensional CW-complex with finitely many cells. We denote by $\ver=\ver(X)$ the set of vertices, by $\edg=\edg(X)$ the set of edges of $X$, and by $\hal=\hal(X)$ the set of half-edges of a graph $X$---a half-edge is nothing other than the germ of an edge. Given a vertex $v\in\ver(X)$ we let $\hal_v=\hal_v(X)$ be the set of half-edges emanating out of $v$. Note that two elements of $\hal_v$ might well correspond to the same edge---that is, $X$ might have edges which are incident to the same vertex on both ends. The cardinality of $\hal_v$ is the {\em degree} of $X$ at $v$ and we say that $X$ is {\em trivalent} if its degree is 3 at every vertex. The reader can safely assume that the graphs they encounter are trivalent.

When it comes to surfaces, we will be working all the time with the same underlying surface, our fixed closed connected and oriented hyperbolic surface $\Sigma=\Gamma\bs\BH^2$. We identify $T^1\Sigma$ with $\Gamma\bs T^1\BH^2=\Gamma\bs\PSL_2\BR$, and endow $T^1\Sigma$ we the distance induced by a $\PSL_2\BR$ left-invariant Riemannian metric on $\PSL_2\BR$, say one that gives length $2\pi$ to each unit tangent space $T_{x_0}^1\Sigma$ and such that the projection $T^1\BH^2\to\BH^2$ is a Riemannian submersion. This means that the unit tangent bundle has volume $\vol(T^1\Sigma)=2\pi\cdot\vol(\Sigma)=4\pi^2\cdot\vert\chi(S)\vert$. Angles between tangent vectors based at the same point of $\Sigma$ will always be unoriented, meaning that they take values in $[0,\pi]$---note that this is consistent with the unit tangent spaces having length $2\pi$. 

Often we will denote the discrete sets we are counting by boldface capitals, such as $\BFC$ or $\BFB$. They will often come with a wealth of decorations such as for example $\BFG^X(L)$ or $\BFB_g(L)$. Often these sets arise as discrete subsets of larger spaces which will be denoted by calligraphic letters. Often the boldfaced and the calligraphic letters go together: $\BFG$ will be a subset of the space $\CG$.

A comment about constants. In this paper there are two kinds of constants: the ones whose value we are trying to actually compute, and those about which we just need to know that they exist. Evidently the first kind we have to track carefully. It would however be too painful, and for no clear gain, to do the same with all possible constants. And in general constants tend to breed more and more constants. This is why we we just write $\bconst$ for a constant whose actual value is irrelevant, allowing the precise value of $\bconst$ to change from line to line. We hope that this does not cause any confusion.

And now a comment about Euclidean vectors. All vector arising here, indicated with an arrow as in $\vec v=(v_1,\dots,v_k)$, are {\em positive} in the sense that the entries $v_i$ are positive. In other words, they live in $\BR_+^k$ or maybe in $\BN^k$. We will write 
$$\Vert\vec v\Vert=\vert v_1\vert+\dots+\vert v_k\vert=v_1+\dots+v_k$$
for the $L^1$-norm on vectors. This is the only norm we will encounter here.

\section{Realizations of graphs in surfaces}\label{sec realizations}
In this section we discuss certain spaces of maps of graphs into a surface. These spaces play a key r\^ole in this paper. Although long, the material here should be nice and pleasant to read---only the proof of Proposition \ref{lemma new} takes some amount of work.

\subsection*{Realizations}
We will be interested in connected graphs living inside our hyperbolic surface $\Sigma$, or more precisely in continuous maps
\begin{equation}\label{eq realization}
\phi:X\to\Sigma
\end{equation}
of graphs $X$ into $\Sigma$ which when restricted to each edge are geodesic. We will say that such a map \eqref{eq realization} is a {\em realization of $X$ in $\Sigma$.} We stress that realizations do not need to be injective, and that in fact the map $\phi$ could be constant on certain edges, or even on larger pieces of the graph. A {\em regular} realization is one whose restriction to every edge is non-constant. If $\phi:X\to\Sigma$ is a regular realization and if $\vec e\in\hal(X)$ is a half-edge incident to a vertex $v\in\ver(X)$ then we denote by $\phi(\vec e)\in T^1_{\phi(v)}\Sigma$ the unit tangent vector at $\phi(v)$ pointing in the direction of the image of $\vec e$.

We endow the set $\CG^X$ of all realizations of the (always connected) graph $X$ with the compact-open topology and note that unless $X$ itself is contractible, the space $\CG^X$ is not connected: the connected components of $\CG^X$ correspond to the different possible free homotopy classes of maps of $X$ into $\Sigma$. Indeed, pulling segments tight relative to their endpoints we get that any homotopy 
$$[0,1]\times X\to\Sigma,\ \ (t,x)\mapsto \varphi_t(x)$$
between two realizations is homotopic, relative to $\{0,1\}\times X$, to a homotopy, which we are still denoting by the same symbol, such that 
\begin{enumerate}
\item $t\mapsto \varphi_t(v)$ is geodesic for every vertex $v\in\ver(X)$, and
\item $\varphi_t:X\to\Sigma$ is a realization for all $t$.
\end{enumerate}
A homotopy $[0,1]\times X\to\Sigma$ satisfying (1) and (2) is said to be a {\em geodesic homotopy}.

Geodesic homotopies admit a different intrinsic description. Indeed, uniqueness of geodesic representatives in each homotopy class of arcs implies that each realization $\phi\in\CG^X$ has a neighborhood which is parametrized by the image of the vertices. This implies that the map
\begin{equation}\label{eq cover}
\Pi:\CG^X\to\Sigma^{\ver(X)},\ \ \phi\mapsto(\phi(v))_{v\in\ver(X)}
\end{equation}
is a cover. Pulling back the product of hyperbolic metrics we think of it as a manifold locally modeled on the product $(\BH^2)^{\ver(X)}=\BH^2\times\dots\times\BH^2$ of $\ver(X)$ worth of copies of the hyperbolic plane. Geodesic homotopies are, from this point of view, nothing other than geodesics in $\CG^X$.

Since all of this will be quite important, we record it here as a proposition:

\begin{prop}\label{prop cover}
Let $X$ be a graph. The map \eqref{eq cover} is a cover and geodesic homotopies are geodesics with respect to the pull-back metric.\qed
\end{prop}

\subsection*{Length function}
On the space $\CG^X$ of realizations of the graph $X$ in $\Sigma$ we have the {\em length function} 
$$\ell_\Sigma:\CG^X\to\BR_{\ge 0},\ \ \ell_\Sigma(\phi)=\sum_{e\in\edg(X)}\ell_\Sigma(\phi(e)).$$
First note that Arzela-Ascoli implies that $\ell_\Sigma$ is a proper function. It follows thus that the restriction of $\ell_\Sigma$ to any and every connected component of $\CG^X$ has a minimum. Now, convexity of the distance function $d_{\BH^2}(\cdot,\cdot)$ implies that $\ell_\Sigma$ is convex. More precisely, if $(\varphi_t)$ is a geodesic homotopy between two realizations in $\Sigma$ then the function $t\mapsto \ell_\Sigma(\varphi_t)$ is convex. Indeed, it is strictly convex unless the image of $[0,1]\times X$ is contained in a geodesic in $\Sigma$, or rather if the image of some (and hence any) lift to the universal cover is contained in a geodesic in $\BH^2$. 

Note now also that the length function is smooth when restricted to the set of regular realizations---its derivative is given by the first variation formula
$$\frac d{dt}\ell_\Sigma(\phi_t)=\sum_{v\in\ver(X)}\sum_{\vec e\in\hal_v(X)} \left\langle-\phi(\vec e),\frac d{dt}\phi_t(v)\right\rangle$$
where $\phi(\vec e)\in T_v^1\Sigma$ is the unit tangent vector based at $v$ and pointing as the image of the half-edge $\vec e$. It follows that a regular realization $\phi$ is a critical point for the length function $\ell_\Sigma:\CG^X\to\BR_{\ge 0}$ if and only if for every $v\in\ver(X)$ we have $\sum_{\vec e\in\hal_v}\phi(\vec e)=0$. Note that this implies, in the for us relevant case that $X$ is trivalent, that the (unsigned) angle $\angle(\phi'(\vec e_1),\phi'(\vec e_2))$ between the images of any two half edges incident to the same vertex is equal to $\frac{2\pi}3$. 

\begin{defi*}
A regular realization $\phi:X\to\Sigma$ of a trivalent graph $X$ into $\Sigma$ is {\em critical} if we have
$$\angle(\phi(\vec e_1),\phi(\vec e_2))=\frac{2\pi}3$$ 
for every vertex $v\in\ver(X)$ and for any two distinct $\vec e_1,\vec e_2\in\hal_v(X)$.
\end{defi*}

We collect in the next lemma a few of the properties of critical realizations:

\begin{lem}\label{lem component critical realization}
Let $X$ be a trivalent graph. A regular realization $\phi\in\CG^X$ is a critical point for the length function if and only if $\phi$ is a critical realization. Moreover, if $\phi\in\CG^X$ is a critical realization and $\CG^\phi$ is the connected component of $\CG^X$ containing $\phi$, then the following holds:
\begin{enumerate}
\item $\phi$ is the unique critical realization in $\CG^\phi$.
\item $\phi$ is the global minimum of the length function on $\CG^\phi$. 
\item Besides $\phi$, there are no other local minima in $\CG^\phi$ of the length function.
\item The connected component $\CG^\phi$ is isometric to the product $\BH^2\times\dots\times\BH^2$ of $\ver(X)$ many copies of the hyperbolic plane.
\end{enumerate}
\end{lem}

Note that lack of global smoothness of the length function means that we cannot directly derive (2) and (3) from (1). This is why they appear as independent statements.

\begin{proof}
Statement (1) was actually discussed in the paragraph preceding the definition of critical realization. Let us focus in the subsequent ones. Recall that if 
$$[0,1]\to\CG^\phi,\ t\mapsto\phi_t$$ 
is a (non-constant) geodesic homotopy with $\phi_0=\phi$, then the length function
\begin{equation}\label{eq I want pie}
t\mapsto\ell_\Sigma(\phi_t(X))
\end{equation}
is convex. Well, in our situation it is strictly convex: since $\phi$ is critical we get that its image, or rather the image of its lifts to the universal cover, are not contained in a geodesic because if they were, then the angles between any two half-edges starting at the same vertex could only take the values $0$ or $\pi$. Now, strict convexity and the fact that $\phi=\phi_0$ is a critical point of the length function implies that $t=0$ is the minimum and only critical point of the function \eqref{eq I want pie}. From here we get directly (2) and (3). To prove (4) note that if $\CG^\phi$ were not simply connected, then there would be a non-trivial geodesic homotopy starting and ending at $\phi$, contradicting the strict convexity of the length function. Note now that the restricting the cover \eqref{eq cover} to $\CG^\phi$ we get a locally isometric cover $\CG^\phi\to\Sigma^{\ver(X)}$. Since the domain of this cover is connected and simply connected, we get that it is nothing other than the universal cover of $\Sigma^{\ver(X)}$. This proves (4) and concludes the proof of the lemma.
\end{proof}

\begin{bem}
Although we will not need it here, let us comment briefly on the topology of the connected components of $\CG^X$. The fundamental group of the connected component $\CG^\phi$ containing a realization $\phi\in\CG^X$ is isomorphic to the centralizer $\CZ_{\pi_1(\Sigma)}(\phi_*(\pi_1(X)))$ in $\pi_1(\Sigma)$ of the image of $\pi_1(X)$ under $\phi_*:\pi_1(X)\to\pi_1(\Sigma)$. It follows that $\CG^\phi$ is isometric to the quotient under the diagonal action of $\CZ_{\pi_1(\Sigma)}(\phi_*(\pi_1(X)))$ of $\BH^2\times\dots\times\BH^2$ of $\ver(X)$-copies of the hyperbolic plane. In particular we have:
\begin{itemize}
\item If $\phi$ is homotopically trivial then $\CG^\phi\simeq\pi_1(\Sigma)\bs(\BH^2\times\dots\times\BH^2)$.
\item If $\phi_*(\pi_1(X))\neq\Id_{\pi_1(\Sigma)}$ is abelian then $\CG^\phi\simeq\BZ\bs(\BH^2\times\dots\times\BH^2)$.
\item $\phi_*(\pi_1(X))$ is non-abelian then $\CG^\phi\simeq\BH^2\times\dots\times\BH^2.$
\end{itemize}
In the appendix we will give a concrete description of the components of $\CG^X$ for the case that $X$ is a loop, that is the graph with a single vertex and a single edge.
\end{bem}

Lemma \ref{lem component critical realization}, with all its beauty, does not say anything about the existence of critical realizations. Our next goal is to prove that any realization that from far away kind of looks like a critical realization is actually homotopic to a critical realization.

\subsection*{Quasi-critical realizations}\label{sec almost critical}

Recall that a regular realization $\phi:X\to\Sigma$ of a trivalent graph $X$ is {\em critical} if the angles between the images of any two half-edges incident to the same vertex are equal to $\frac{2\pi}3$. We will say that a regular realization is {\em quasi-critical} if those angles are bounded from below by $\frac{1}{2}\pi$ and that the realization is {\em $\ell_0$-long} if $\ell(\phi(e))>\ell_0$ for all edges $e$ of $X$. Recall that we measure all angles in $[0,\pi]$.

Our next goal here is to prove that every sufficiently long quasi-critical realization is homotopic to a critical realization. This will follow easily from the following technical result:

\begin{prop}\label{lemma new}
For any trivalent graph $X$ and constant $C\geq0$ there exist constants $\ell_0, D>0$ such that given an $\ell_0$-long quasi-critical realization $\phi: X\to \Sigma$, a trivalent graph $Y$, and realization $\psi: Y\to\Sigma$ satisfying
\begin{enumerate}
\item $\ell(\psi)\leq\ell(\phi)+C$, and
\item there is a homotopy equivalence $\sigma: X\to Y$ with $\phi$ and $\sigma\circ\psi$ homotopic, 
\end{enumerate}
then there exists a homeomorphism $F: Y\to X$ mapping each edge with constant speed, such that $\sigma\circ F$ is homotopic to the identity and such that the geodesic homotopy $X\times [0,1]\to \sigma$,  $(x, t)\mapsto \phi_t(x)$ joining $\phi_0(\cdot) = \phi\circ F(\cdot)$ to $\phi_1(\cdot) = \psi(\cdot)$ has tracks bounded by $D$. 
\end{prop}

Since the proof of Proposition \ref{lemma new} is pretty long, let us first demonstrate that it might be useful: 

\begin{kor}\label{lem critical near kind of critical}
Let $X$ be a trivalent graph. There are positive constants $\ell_0$ and $D$ such that every component of $\CG^X$ which contains an $\ell_0$-long quasi-critical realization $\phi: X\to \Sigma$, also contains a critical realization $\psi$, which moreover is unique and homotopic to $\phi$ by a homotopy whose tracks have length bounded by $D$.
\end{kor}

\begin{proof} 
Let $\ell_0$ and $D$ be given by Lemma \ref{lemma new} for $C=0$ and if needed increase $\ell_0$ so that it is larger than $2D$. Let $\phi: X\to \Sigma$ be an $\ell_0$-long quasi-critical realization. Let $\psi: X\to \Sigma$ be the minimizer for the length function in the component $\CG^{\phi}$---it exists because the length function is proper. We claim that $\psi$ is critical. In the light of Lemma \ref{lem component critical realization} it suffices to prove that $\psi$ is regular.

Being a minimizer we have that $\ell_\Sigma(\psi)\le\ell_\Sigma(\phi)$. We thus get from Proposition \ref{lemma new} that there exists a homeomorphism $F: X\to X$ homotopic to the identity, mapping edges with constant speed, and such that the geodesic homotopy from $\phi\circ F$ to $\psi$ has tracks bounded by $D$. Now, since $X$ is trivalent we get that the homeomorphism $F$, being homotopic to the identity and mapping edges with constant speed, is actually equal to the identity. What we thus have is a homotopy from $\phi$ and $\psi$ with tracks bounded by $D$. Now, since each edge of $\phi(X)$ has at least length $\ell_0>2D$ and since the tracks of the homotopy are bounded by $D$, we get that $\psi$ is regular and hence critical by Lemma \ref{lem component critical realization}, as we needed to prove. Lemma \ref{lem component critical realization} also yields that $\psi$ is unique.
\end{proof}

Let us next prove the proposition:

\begin{proof}[Proof of Proposition \ref{lemma new}]
Let $\phi: X\to\Sigma$ be a quasi-critical realization and assume it is $\ell_0$-long for an $\ell_0$ large enough to satisfy some conditions we will give in the course of the proof. Let $H: X\times[0,1]\to \Sigma$ be the homotopy between $\phi$ and $\psi\circ\sigma$. 

To be able to consistently choose lifts of $\phi$, $\psi$, and $\sigma$ to the universal covers $\widetilde X, \widetilde Y$ and $\BH^2$ of $X,Y$ and $\Sigma$, let us start by picking base points. Fixing $x_0\in X$, consider the base points $\sigma(x_0)\in Y$ and $\phi(x_0)\in\Sigma$, and pick lifts $\widetilde{x}_0\in\widetilde X$, $\widetilde{\sigma(x_0)}\in\widetilde Y$ and $\widetilde{\phi(x_0)}\in\mathbb{H}^2$ of each one of those endpoints. Having chosen those base points we have uniquely determined lifts $\widetilde{\phi}: \widetilde{X}\to\mathbb{H}^2$ and $\widetilde{\sigma}: \widetilde{X}\to\widetilde{Y}$ of $\phi$ and $\sigma$ satisfying $\widetilde{\phi}(\widetilde x_0) = \widetilde{\phi(x_0)}$ and $\widetilde{\sigma}(\widetilde{x}_0) = \widetilde{\sigma(x_0)}$.  We can also lift to $\BH^2$, starting at $\widetilde\phi(\widetilde{x_0})$, the path in $\Sigma$ given by $t\mapsto H(x_0, t)$. The endpoint of this path is a lift of $\psi\circ\sigma(x_0)$ and we take the lift $\widetilde{\psi}: \widetilde{Y}\to\mathbb{H}^2$ which maps $\widetilde{\sigma}(\widetilde x_0)$ to this point. All those lifts are related by the following equivariance property:
\begin{equation}\label{eq equivariance}
\widetilde{\phi}(g(x)) = \phi_*(g)(\widetilde{\phi}(x)) \text{ and } (\widetilde{\psi}\circ\widetilde{\sigma})(g(x)) = \phi_*(g)\left((\widetilde{\psi}\circ\widetilde{\sigma})(x)\right)
\end{equation} 
for all $x\in \widetilde{X}$ and for all $g\in \pi_1(X, x_0)$ where $\phi_*: \pi_1(X, x_0)\to \pi_1(\Sigma, \phi(x_0))$ is the homomorphism induced by $\phi$ and the chosen base points. 

Note that, since $\phi$ is almost critical and $\ell_0$-long we get, as long as $\ell_0$ is large enough, that the lift $\widetilde\phi:\widetilde X\to\BH^2$ is an injective quasi-isometric embedding. For ease of notation, denote by $T_X=\widetilde{\phi}(\widetilde{X})\subset\mathbb{H}^2$ the image of $\widetilde\phi$ and let us rephrase what we just said: if $\ell_0$ is sufficiently large then $T_X$ is a quasiconvex tree, with quasiconvexity constants only depending on a lower bound for $\ell_0$. We denote by $\hat\pi:\BH^2\to T_X$ a nearest point retraction which is equivariant under $\phi_*(\pi_1(X, x_0))$. We would like to define 
$$\pi: \widetilde{Y}\to T_X$$ 
as $\hat\pi\circ\widetilde\psi$, but since $\hat\pi$ is definitively not continuous, we have to be slightly careful. We define $\pi$ as follows: first set $\pi(v) = \hat\pi(\tilde{\psi}(v))$ for all vertices $v\in \ver(\widetilde Y)$ of $\widetilde Y$, and then extend (at constant speed) over the edges of $Y$.  Note that \eqref{eq equivariance}, together with the equivariance of $\pi$, implies that $\pi\circ\widetilde{\sigma}: \tilde{X}\to T_X$ satisfies 
$$(\pi\circ\widetilde{\sigma})(g(x)) = \phi_*(g)(\pi\circ\widetilde{\sigma}(x))$$
for all $x\in \widetilde{X}$ and $g\in \pi_1(X, x_0)$.

\begin{claim}\label{claim no idea what name}
As long as $\ell_0$ is large enough, we have that $\pi$ maps each vertex of $\widetilde Y$ within $\bconst$ of one and only one vertex of $T_X$.
\end{claim} 

Starting with the proof of the claim, note that there is a constant $A\ge 0$ depending only on the quasiconvexity constants for $T_X$ with 
\begin{equation}\label{eq quasi-convex horrible 1}
d_{\BH^2}(x,y)\ge -A+d_{\mathbb{H}^2}(\hat\pi(x),\hat\pi(y))
\end{equation}
for all $x,y\in\BH^2$. In fact, there exists constants $K, A'>0$ which once again depend only on the quasiconvexity constant of $T_X$ such that whenever $d_{\mathbb{H}^2}(\hat\pi(x),\hat\pi(y))>K$ we have the better bound
\begin{equation}\label{eq quasi-convex horrible 2}
d_{\BH^2}(x,y)\ge -A'+d_{\BH^2}(x,\hat\pi(x))+d_{\BH^2}(\hat\pi(x),\hat\pi(y))+d_{\BH^2}(\hat\pi(y),y).
\end{equation}

For $a, b\in T_X$ let $d_{T_X}(a,b)$ denote the interior distance in $T_X$ between them. We will care about a modified version of this distance: let $B$ be the collection of balls in $\BH^2$ of radius 1 centered at each vertex of $T_X$ and define $d_{T_X\text{rel}B}(a, b)$ to be the length of the part of the path in $T_X$ between them that lies outside of $B$. For all $\ell_0$ large enough (depending only on hyperbolicity of $\BH^2$) we have from the choice of the radius\footnote{In fact, any radius greater than $\log\sqrt{2}\approx 0.3466$ would work for large $\ell_0$.} that 
\begin{equation}\label{eq relative}
d_{T\text{rel}B}(a, b) \leq d_{\BH^2}(a, b)
\end{equation}
for all $a, b \in T_X$.

Given two edges $e, e'$ of $\widetilde Y$ adjacent to the same vertex, we define 
$$\fold(e, e') = \ell_{T\text{rel}B}(\pi(e)\cap\pi(e'))$$ 
and let then 
$$\fold(\pi)=\max\fold(e,e')$$ 
where the maximum is taken over all such pairs of edges. The reason to introduce this quantity is that we have
\begin{equation}\label{eq another one}
\sum_{[v, v']\in\edg(Y)} d_{T_X\text{rel}B}(\pi(v), \pi(v'))
 \geq  \ell(\phi) + \fold(\pi) - A''
 \end{equation}
for some positive constant $A''$ depending only on that number of vertices (and the fact that we chose the balls $B$ to have radius $1$)---here we have identified each edge of $Y$ with one of its representatives in $\widetilde Y$.

Now, using \eqref{eq quasi-convex horrible 1}, \eqref{eq relative} and \eqref{eq another one} we get that 
\begin{align*}
\ell(\psi) 
&= \sum_{[v, v']\in\edg(Y)}d_{\BH^2}(\tilde\psi(v), \tilde\psi(v')) \\
& \geq -\bconst + \sum_{[v, v']\in\edg(Y)} d_{\BH^2}(\pi(v), \pi(v'))\\
& \geq -\bconst + \sum_{[v, v']\in\edg(Y)} d_{T_X\text{rel}B}(\pi(v), \pi(v'))\\
& \geq -\bconst + \ell(\phi) + \fold(\pi)
\end{align*} 
where each $\bconst$ is a positive constant (but not necessarily the same) depending on the quasiconvexity constant and combinatorics of $Y$.

Since $\ell(\psi)\leq\ell(\phi)+C$ it follows that $\fold(\pi) \leq\bconst$, where the constant depends on $C$ and on the lower bound for $\ell_0$ and on combinatorics of $Y$. In particular, for each $v\in \ver$ there is a vertex $v'\in T_X$ such that $d_{T_X\text{rel}B}(\pi(v), v') \leq \bconst$ and we can choose $\ell_0$ large such that $v'$ is unique. We have proved Claim \ref{claim no idea what name}.\qed
\medskip

Armed with Claim \ref{claim no idea what name} we can define a map $\widetilde{F}: \widetilde Y\to T_X$ by first mapping $v\in \ver(\widetilde Y)$, to the unique vertex in $T_X$ closest to $\pi(v)$, extending it so that it maps each $e\in\edg(\widetilde Y)$ with constant speed. Note that equivariance of $\pi$ implies that $\widetilde{F}$ satisfies $\widetilde F(\sigma_*(g)(y)) = g(\widetilde F(y))$ for all $y\in\widetilde Y$ and $g\in \pi_1(X, x_0)$, where $\sigma_*: \pi_1(X, x_0)\to \pi_1(Y, \sigma(x_0))$ is the isomorphism induced by $\sigma$ and the chosen base points. It follows that $\widetilde F$ descends to a map $F: Y\to X$. 

\begin{claim}\label{claim waiting for dpd}
$F:Y\to X$ is a homeomorphism with $F\circ\sigma$ homotopic to the identity.
\end{claim}

The fact that $F\circ\sigma$ is homotopic to the identity follows directly from the fact $\widetilde F\circ \widetilde\sigma: \widetilde X\to \widetilde X$ satisifies that $(\widetilde F\circ\sigma)(gx)=g((\widetilde F\circ\widetilde\sigma)(x))$ for $g\in\pi_1(X,x_0)$. What we really have to prove is that $F$ is a homeomorphism. To see that this is the case note that since the maps $\widetilde F,\pi:\widetilde Y\to\widetilde X$ send points within $\bconst$ of each other, and since in our way to proving Claim \ref{claim no idea what name} we proved that $\fold(\pi)<\bconst$, we get that $\fold(\widetilde F)<\bconst$. On the other hand, since $\widetilde F$ maps vertices to vertices and has constant speed on the edges we get that if $\fold(\widetilde F)\neq 0$ then $\fold(F)$ has to be at least as large as the shortest edge of $X$. This implies, using the fact that $\fold(\widetilde F)$ is uniformly bounded, that as long as $\ell_0$ is large enough then $\fold(\widetilde F)=0$. This implies in turn that $\widetilde F$ is locally injective. Local injectivity of $\widetilde F$, together with the fact that $F$, being a homotopy inverse to $\sigma$, is a homotopy equivalence, imply that $F$ is a homeomorphism. We have proved Claim \ref{claim waiting for dpd}.\qed
\medskip

What is left is to bound the tracks of the geodesic homotopy from $\phi\circ F$ to $\psi$---the argument is similar to the proof of the existence of $F$. First, note that we know that for every vertex $v\in \widetilde Y$, the nearest point projection $\hat\pi$ maps $\widetilde\psi(v)$ close to the vertex $\tilde\phi(\widetilde F(v))$ of $T_X$. It follows that if we choose $\ell_0$ large enough we can also assume that  $d_{\mathbb{H}^2}(\hat\pi(\tilde\psi(v)),\hat\pi(\tilde\psi((v')))) > K$ for any two distinct vertices $v,v'\in\ver(\widetilde Y)$. This means that \eqref{eq quasi-convex horrible 2} applies, and hence that we have 
\begin{align*}
\ell(\psi(Y))
&= \sum_{[v, v']\in\edg}d_{\BH^2}(\tilde\psi(v), \tilde\psi(v')) \\
& \geq -\bconst + \sum_{[v, v']\in\edg} d_{\mathbb{H}^2}(\hat\pi(\tilde\psi(v)),\hat\pi(\tilde\psi(v')))+ \\
& \hspace{1cm}+ 3\sum_{v\in\ver} d_{\BH^2}(\tilde\psi(v),\hat\pi(\tilde\psi(v)))\\
& \geq -\bconst + \ell(\phi(X)) +  3\sum_{v\in\ver} d_{\mathbb{H}^2}(\tilde\psi(v), \hat\pi(\tilde\psi(v)))
\end{align*}
where as before each $\bconst$ is a positive constant, the first inequality follows by \eqref{eq quasi-convex horrible 2}, the second by \eqref{eq relative} and \eqref{eq another one}. Hence, again by assumption (1) in the lemma, we must have that $d_{\mathbb{H}^2}(\tilde\psi(v), \hat\pi(\tilde\psi(v)))<\bconst$ for all $v\in\ver(\widetilde Y)$. Since $\widetilde{F}(v)$ and $\hat\pi(\tilde\psi(v))$ are near each other, we have that $d_{\mathbb{H}^2}(\tilde\psi(v), \widetilde F(v))<\bconst$ for all $v\in\ver(\widetilde Y)$. Convexity of the length function implies that the geodesic homotopy between $\tilde\psi$ and $\tilde \phi\circ\widetilde F$ has then tracks bounded by the same constant $\bconst$, as we had claimed.
\end{proof}

\subsection*{$\epsilon$-critical realizations}

Corollary \ref{lem critical near kind of critical} asserts that whenever we have a realization with very long edges and which looks vaguely critical, then it is not far from a critical realization. Our next goal is to give a more precise description of the situation in the case that our realization looks even more as if it were critical. To be precise, suppose that $X$ is a trivalent graph and $\epsilon$ positive and small. A regular realization $\phi:X\to\Sigma$ is {\em $\epsilon$-critical} if 
$$\angle(\phi(\vec e_1),\phi(\vec e_2))\in\left[\frac{2\pi}3-\epsilon,\frac{2\pi}3+\epsilon\right]$$
for every two half-edges $\vec e_1,\vec e_2\in\hal_v$ incident to any vertex $v\in\ver(X)$.

The goal of this section is to determine how the set $\CG^X_{\epsilon-\crit}$ of $\epsilon$-critical realizations of $X$ looks. To do that we need a bit of of notation and preparatory work. First, under a {\em tripod} we will understand a tuple $\tau=(p,\{v,v',v''\})$ where $p\in\BH^2$ is a point and where $\{v,v',v''\}\subset T^1_p\BH^2$ is an unordered collection consisting of three distinct unit vectors based at $p$. A tripod $\tau=(p,\{v,v',v''\})$ is {\em critical} if the three vectors $v,v',v,''$ have pairwise angle equal to $\frac{2\pi}3$. Any tripod $\tau$ determines 3 geodesic rays, that is 3 points in $\D_\infty\BH^2$, that is an ideal triangle $T_\tau\subset\BH^2$ (see left hand side of Figure \ref{fig:triangle}). Conversely, every ideal triangle $T\subset\BH^2$ determines a tripod $\tau_p^T=(p,\{v,v',v''\})$ for every $p\in\BH^2$: the vectors $v,v',v,''$ are the unit vectors based at $p$ and pointing to the (ideal) vertices of the ideal triangle $T$. Note that $T$ consists of exactly those points $p$ such that the vectors in the tripod $\tau_p^T$ have pairwise (always unoriented) angles adding up to $2\pi$. For a given $\epsilon$ we let
$$T(\epsilon)=\left\{p\in T\ \middle\vert\begin{array}{l}\text{the angles between the vectors}\\ \text{in }\tau_p^T\text{ lie in }[\frac{2\pi}3-\epsilon,\frac{2\pi}3+\epsilon]\end{array}\right\}$$
be the set of points such that the angles of the tripod $\tau_p^T$ are within $\epsilon$ of $\frac{2\pi}3$. The set $T(\epsilon)$ is, at least for $\epsilon\in(0,\frac\pi 6)$, a hexagon centered at the center of the triangle---the sides of the hexagon are not geodesic but rather sub-segments of curves of constant curvature, see Figure \ref{fig:triangle}. However, if we scale everything by $\epsilon^{-1}$ then $\epsilon^{-1}\cdot T(\epsilon)$ converges to the regular euclidean hexagon of side-length $\frac 23$. This means in particular that 
\begin{equation}\label{eq hexagon}
\diam(T(\epsilon))\sim\frac 43\epsilon\text{ and }\vol(T(\epsilon))\sim \frac 2{\sqrt 3}\epsilon^2\text{ as }\epsilon\to 0.
\end{equation}

\begin{figure}[h]
\centering
\includegraphics[width=0.7\textwidth]{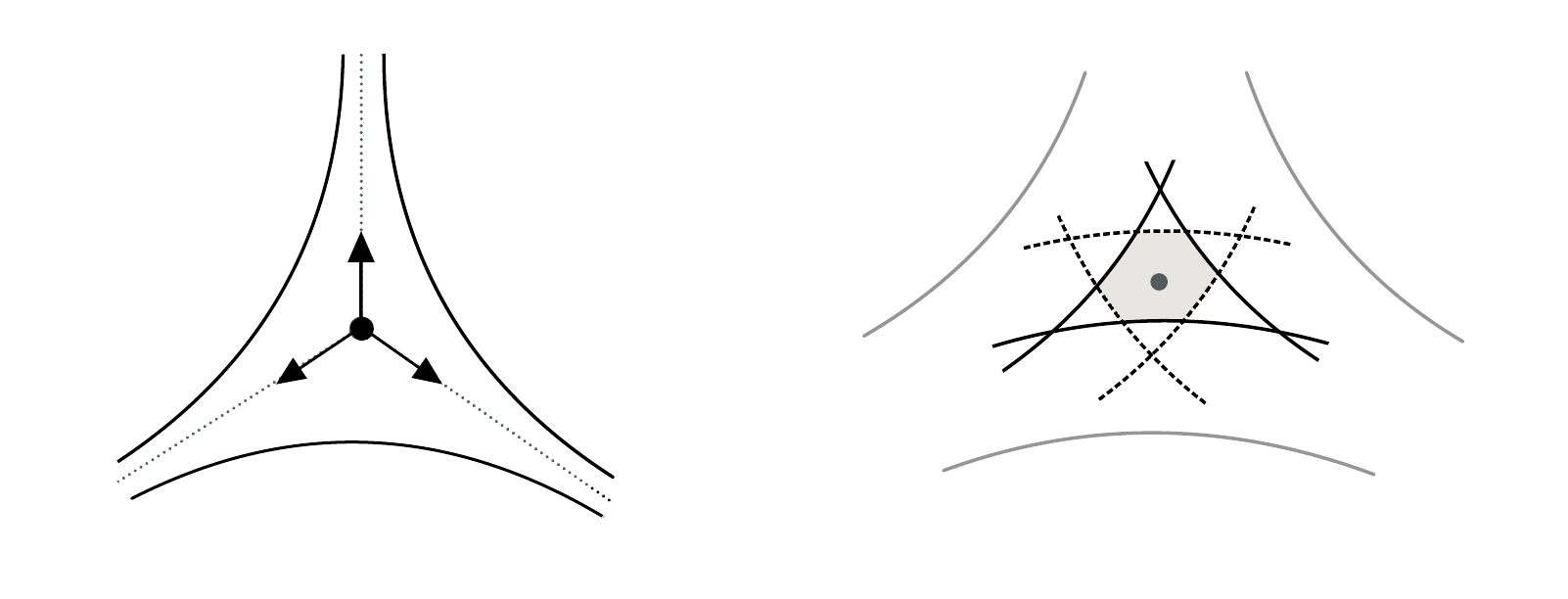}
\caption{Left: A critical tripod and (part of) the ideal triangle it determines. Right: The hexagon $T(\epsilon)$. Each of the lines making up the boundary of the hexagon corresponds to points $p$ having an angle between vectors in $\tau^T_P$ of measure $\frac{2\pi}{3}-\epsilon$ (dotted lines) and $\frac{2\pi}{3}+\epsilon$ (solid lines).} 
\label{fig:triangle}
\end{figure}

\begin{bem}
To get the shape of $T(\epsilon)$ one can use elementary synthetic hyperbolic geometry. But one can also resort to hyperbolic trigonometry. For example, evoking formula 2.2.2 (iv) in the back of Buser's book \cite{Buser} one gets 
$$T(\epsilon)=\left\{p\in T\,\middle\vert\, \cot\left(\frac{\pi}{3}-\frac{\epsilon}{2}\right)\leq \sinh(d(p, \partial T))\leq\cot\left(\frac{\pi}{3}+\frac{\epsilon}{2}\right)\right\}.$$
\end{bem}

\smallskip

Suppose now that $\phi:X\to\Sigma$ is a critical realization of a trivalent graph $X$ in our closed hyperbolic surface and recall that by Lemma \ref{lem component critical realization} we have that the connected component $\CG^\phi$ of geodesic realizations homotopic to $\phi$ is isometric to a product of hyperbolic planes $\BH^2\times\dots\times\BH^2$, one factor for each vertex of $X$. To each vertex $x$ of $X$ we can associate first the tripod $\tau_x^\phi=(\phi(x),\{\phi(\vec e)\}_{\vec e\in\hal_x(X)})$ consisting of the image under $\phi$ of the vertex $x$ and of the unit vectors $\phi(\vec e)$ tangent to the images of the half-edges incident to $x$, and then $T_{\tau^\phi_x}$ the ideal triangle associated to the critical tripod $\tau^\phi_x$. The assumption that $\phi$ is critical implies that the point $\phi(x)$ is the center of $T_{\tau^\phi_x}$ for every vertex $x\in\ver$. We let 
$$\CT^\phi=\prod_{x\in\ver(X)}T_{\tau_x^\phi}\subset\prod_{x\in\ver(X)}\BH^2=\CG^\phi$$
be the subset of $\CG^\phi$ consisting of geodesic realizations homotopic to $\phi$ via a homotopy that maps each vertex $x$ within $T_{\tau_x^\phi}$. Accordingly we set
$$\CT^\phi(\epsilon)=\prod_{x\in\ver(X)}T_{\tau_x^\phi}(\epsilon)\subset\CT^\phi$$
and we note that
$$\vol(\CT^\phi(\epsilon))\sim \left(\frac 2{\sqrt 3}\epsilon^2\right)^{\vert\ver(X)\vert}.$$

The reason why we care about all these sets is that, asymptotically, $\CT^\phi(\epsilon)$ agrees with the set $\CG^\phi_{\epsilon-\crit}$ of $\epsilon$-critical realizations $\psi$ homotopic to $\phi$. Indeed we get from Corollary \ref{lem critical near kind of critical} that the geodesic homotopy from $\psi$ to $\phi$ has lengths bounded by a uniform constant. This means that for all $e\in\edg(X)$ the geodesic segments $\phi(e)$ and $\psi(e)$ are very long but have endpoints relatively close to each other. This means that, when looked from (each) one of its vertices $x$, the segment $\psi(e)$ is very close to being asymptotic to $\phi(e)$. In particular, the angle between $\psi(e)$ and the geodesic ray starting at $\psi(x)$ and asymptotic to the ray starting in $\phi(x)$ in direction $\phi(e)$ is smaller than $\delta=\delta(\min_{e\in\edg(X)}\ell(\psi(e)))$ for some positive function with $\delta(t)\to 0$ as $t\to\infty$. With the rather clumsy notation we find ourselves working with we have that $\delta$ bounds from above the angles between corresponding edges of the tripods $\tau^\psi_x$ and $\tau^{T^\phi_x}_{\psi(x)}$. Since by assumption the angles of $\tau^\psi_x$ belong to $[\frac{2\pi}3-\epsilon,\frac{2\pi}3+\epsilon]$ we get that the angles of $\tau^{T^\phi_x}_{\psi(x)}$ belong to $[\frac{2\pi}3-\epsilon-2\delta,\frac{2\pi}3+\epsilon+2\delta]$, meaning that $\psi(x)\in T_{\tau_x^\phi}(\epsilon+2\delta)$.

To summarize, what we have proved is that for all $\epsilon_1>\epsilon$ there is $\ell_1$ such that if $\phi:X\to\Sigma$ is critical with $\ell(\phi(e))\ge\ell_1$ for all $e\in\edg$, then 
$$\CG^\phi_{\epsilon-\crit}\subset\CT^\phi(\epsilon_1).$$
The same argument proves that for all $\epsilon_2<\epsilon$ there is $\ell_2$ such that if $\phi:X\to\Sigma$ is critical with $\ell(\phi(e))\ge\ell_2$ for all $e\in\edg$, then 
$$\CT^\phi(\epsilon_2)\subset\CG^\phi_{\epsilon-\crit}$$
We record these facts:

\begin{lem}\label{lem pain in the ass}
There are functions $h:(0,\frac\pi 6)\to\BR_{>0}$ and $\delta:(0,\epsilon_0)\to\BR_{>0}$
with $\lim_{t\to 0}h(t)=0$ and $\lim_{t\to\infty}\delta(t)=0$ such that for every critical realizations $\phi:X\to\Sigma$ we have
$$\CT^\phi(\epsilon-r(\epsilon,\phi))\subset\CG^\phi_{\epsilon-\crit}\subset\CT^\phi(\epsilon+r(\epsilon,\phi))$$
where $r(\epsilon,\phi)=h(\epsilon)+\delta(\min_{e\in\edg(X)}\ell(\phi(e)))$.\qed
\end{lem}

To conclude this section we collect what we will actually need about the set of $\epsilon$-critical realizations in a single statement.

\begin{prop}\label{prop sum up section 5}
Let $X$ be a trivalent graph. There are $\ell>0$ and functions $\rho_0,\rho_1:\BR_{>0}\to\BR_{>0}$ with $\lim_{t\to 0}\rho_0(t)=0$ and $\lim_{t\to\infty}\rho_1(t)=0$ and such that the following holds for all $\epsilon\in[0,\frac\pi 6]$:

If $\phi\in\CG_{\epsilon-\crit}(X)$ is an $\ell$-long $\epsilon$-critical realization then
\begin{equation}\label{eq prop sum up section 5 1}
\left\vert\vol(\CG_{\epsilon-\crit}^\phi)-\left(\frac 2{\sqrt 3}\epsilon^2\right)^V\right\vert\le \rho_0(\epsilon)\cdot\rho_1\left(\min_{e\in\edg}\ell(\phi(e))\right)\cdot\epsilon^{2V},
\end{equation}
Moreover $\CG^\phi_{\epsilon-\crit}$ contains a unique critical realization $\psi$ and we have 
\begin{equation}\label{eq prop sum up section 5 2}
\max_{e\in\edg}\vert\ell_\Sigma(\phi(e))-\ell_\Sigma(\psi(e))\vert\le\rho_0(\epsilon)+\rho_1\left(\min_{e\in\edg}\ell(\phi(e))\right)\cdot\epsilon
\end{equation}
Here again $V$ is the number of vertices of $X$ and $\edg=\edg(X)$ is its set of edges.
\end{prop}
\begin{proof}
Suppose to begin with that $\ell$ is at least as large as the $\ell_0$ in Corollary \ref{lem critical near kind of critical}. We thus get that $\CG^\phi$ contains a unique critical realization $\psi$. Now we get from Lemma \ref{lem pain in the ass} that
\begin{equation}\label{eq fucking putin}
\CT^\phi(\epsilon-r(\epsilon,\phi))\subset\CG^\phi_{\epsilon-\crit}\subset\CT^\phi(\epsilon+r(\epsilon,\phi))
\end{equation}
where $r(\epsilon,\phi)=h(\epsilon)+\delta(\min_{e\in\edg}\ell(\phi(e)))$ for some functions $h:(0,\frac\pi 6)\to\BR_{>0}\text{ with }$ and $\delta:(0,\epsilon_0)\to\BR_{>0}$ with $\lim_{t\to\infty}h(t)=0$ and $\lim_{t\to\infty}\delta(t)=0$. The bounds \eqref{eq prop sum up section 5 1} and \eqref{eq prop sum up section 5 2} follow now from \eqref{eq fucking putin} and \eqref{eq hexagon}.
\end{proof}

\section{Variations of Delsarte's theorem}\label{sec delsarte}
In this section we establish the asymptotics of the number of realizations $\phi:X\to\Sigma$ satisfying some length condition, mapping each vertex to some pre-determined point in $\Sigma$, and so that the tuple of directions of images of half-edges belongs to some given open set of such tuples---see Theorem \ref{sat delsarte graph} for details. Other than some book-keeping, what is needed is a slight extension of Delsarte's lattice counting theorem \cite{Delsarte}, namely Theorem \ref{thm Delsarte's theorem for sectors} below. This result is known to experts, and much more sophisticated versions than what we need here can be found in the literature---see for example \cite[Th\'eor\`eme 4.1.1]{Roblin}. However, for the sake of completeness, we will explain how to derive Theorem \ref{thm Delsarte's theorem for sectors} from the fact that the geodesic flow is mixing. We refer to the very nice books \cite{Bekka-Meyer,Roblin} for the needed background.
\medskip

We start by recalling a few well-known facts about dynamics of geodesic flows on hyperbolic surfaces. First recall that we can identify $T^1\BH^2$ with $\PSL_2\BR$, and $T^1\Sigma=T^1(\Gamma\bs\BH^2)$ with $\Gamma\bs\PSL_2\BR$. More specifically, when using the identification
$$\PSL_2\BR\to T^1\BH^2\ \ g\mapsto \frac d{dt}g(e^ti)\vert_{t=0}$$
where the computation happens in the upper half-plane model and $i\in\BH^2$ is the imaginary unit, then the geodesic and horocyclic flows amount to right multiplication by the matrices
$$\rho_t=\left(\begin{array}{cc}e^{\frac t2} & 0 \\ 0 & e^{-\frac t2}\end{array}\right)\text{ and }h_s=\left(\begin{array}{cc}1 & s \\ 0 & 1\end{array}\right)$$
respectively.
Note that $K=\SO_2\subset\PSL_2\BR$, the stabilizer of $i$ under the action $\PSL_2\BR\actson\BH^2$, is a maximal compact subgroup---$K$ corresponds under the identification $\PSL_2\BR\simeq T^1\BH^2$ to the unit tangent space $T^1_i\BH^2$ of the base point $i\in\BH^2$. The KAN decomposition (basically the output of the Gramm-Schmidt process from linear algebra) asserts that every element in $g\in\PSL_2\BR$ can be written in a unique way as
\begin{equation}\label{eq kan coordinates}
g=k\rho_th_s
\end{equation}
for $k\in K$ and $t,s\in\BR$. In those coordinates, the Haar measure of $\PSL_2\BR$ is given by
$$\int f(g)\,\vol_{T^1\BH^2}(g)=\iiint f(k\rho_th_s)\cdot e^{-t}\, dkdtds$$
where $dk$ stands for integrating over the arc length in $K\simeq T^1_i\BH^2$, normalized to have total measure $2\pi$.

The basic fact we will need is that the geodesic flow $\rho_t$ on $T^1\Sigma$ is mixing \cite[III.2.3]{Bekka-Meyer}, or rather one of its direct consequences, namely the fact that for each $x_0\in\BH^2$ the projection to $\Sigma$ of the sphere $S(x_0,L)$ centered at $x_0$ and with radius $L$ gets equidistributed in $\Sigma$ when $L\to\infty$ \cite[III.3.3]{Bekka-Meyer}. To be more precise, note that we might assume without loss of generality that $x_0=i$, meaning that we are identifying $T_{x_0}^1\BH^2=K$. The fact that the geodesic flow is mixing implies then that for every non-degenerate interval $I\subset K$ the spherical arcs $I\rho_t$ equidistribute in $\Gamma\bs\PSL_2\BR$ in the sense that for any continuous function $f\in C^0(T^1\Sigma)=C^0(\Gamma\bs\PSL_2\BR)$ we have
\begin{equation}\label{eq equidistribution}
\lim_{L\to\infty}\frac 1{\ell(I)}\int_{I}f(\Gamma k\rho_L)\,dk=\frac 1{\vol_{T^1\Sigma}(T^1\Sigma)}\int_{T^1\Sigma} f(\Gamma g)\cdot dg
\end{equation}
where $\ell(I)$ is the arc length of $I$ and where the second integral is with respect to $\vol_{T^1\Sigma}$. Anyways, we care about equidistribution of the spheres for the following reason. If $f\in C^0(\Sigma)$ is a continuous function and if we let $\tilde f$ be the composition of $f$ with the cover $\BH^2\to\Sigma$ then \eqref{eq equidistribution} implies, with $I=K$, that
$$\lim_{L\to\infty}\frac 1{\vol_{\BH^2}(B(x_0,L))}\cdot\int_{B(x_0,L)}\tilde f(x)\, dx=\frac 1{\vol(\Sigma)}\int_\Sigma f(x)dx$$
Applying this to a non-negative function $f_{y_0,\epsilon}\in C^0(\Sigma)$ with total integral $1$ and supported by $B(y_0,\epsilon)$ we get that
$$\int_{B(x_0,L)}\tilde f_{y_0,\epsilon}(x)\, dx\sim\frac{\vol_{\BH^2}(B(x_0,L))}{\vol(\Sigma)}$$
Now, from the properties of $f_{y_0,\epsilon}$ we get that 
\begin{equation}\label{eq delsart bounds}
\int_{B(x,L-\epsilon)}\tilde f_{y_0,\epsilon}(\cdot)d\vol_{\BH^2}\le\vert\Gamma\cdot y_0\cap B(x_0,L)\vert\le\int_{B(x,L+\epsilon)}\tilde f_{y_0,\epsilon}(\cdot)d\vol_{\BH^2}
\end{equation}
When taken together, the last two displayed equations imply that for all $\epsilon$ one has
$$\frac{\pi\cdot e^{L-\epsilon}}{\vol(\Sigma)}\le\vert\Gamma\cdot y_0\cap B(x_0,L)\vert\le\frac{\pi\cdot e^{L+\epsilon}}{\vol(\Sigma)}$$
and hence that 
\begin{equation}\label{eq delsarte classic}
\vert\Gamma\cdot y_0\cap B(x_0,L)\vert\sim\frac{\pi\cdot e^L}{\vol(\Sigma)}\text{ as }L\to\infty.
\end{equation}
This is Delsarte's lattice point counting theorem \cite{Delsarte}---we refer to \cite[III.3.5]{Bekka-Meyer} for more details.

An observation: note that counting elements of $\Gamma\cdot y_0$ contained in the ball $B(x_0,L)$ is exactly the same thing as counting geodesic arcs in $\Sigma$ of length at most $L$ going from $x_0$ to $y_0$, or to be precise, from the projection of $x_0$ to the projection of $y_0$. In this way, if we are given a segment $I\subset K=T_{x_0}^1\Sigma$ and we use the equidistribution of the spherical segments $I\rho_t$ then when we run the argument above we get the cardinality of the set $\BFA_{I,y_0}(L)$ of geodesic arcs of length at most $L$, starting in $x_0$ with initial speed in $I$, and ending in $y_0$. As expected, the result is that
$$\vert\BFA_{I,y_0}(L)\vert\sim \frac{\ell(I)}{2\pi}\cdot\frac{\pi\cdot e^L}{\vol(\Sigma)}\ \ \ \text{ when }L\to\infty.$$

Suppose now that we dial it up a bit giving ourselves a second sector $J\subset T^1_{y_0}\Sigma$ and care, for some fixed $h$, about the cardinality of the set $\BFA_{I,J}(L,h)$ of geodesic arcs with length in $[L,L+h]$, joining $x_0$ to $y_0$, and with initial and terminal velocities in $I$ and $J$ respectively. We can obtain the asymptotics of $\vert\BFA_{I,J}(L,h)\vert$ following the same basic idea as in the proof of Delsarte's theorem above. We need however to replace the bump function $f_{y_0,\epsilon}$ by something else. 

Using the coordinates \eqref{eq kan coordinates} consider for $h$ and $\delta$ positive and small the set
$$\CJ(J,h,\delta)=\{J\rho_t h_s\text{ with }s\in[-\delta,\delta]\text{ and }t\in[0,h]\}\subset T^1\Sigma$$
and note that it has volume $\ell(J)\cdot 2\delta\cdot(1-e^{-h})$. Note also that the intersection of $\CJ(J,h,\delta)$ with the outer normal vector field of any horosphere consists of segments of the form $H_{v\rho_t}=\{v\rho_t h_s\text{ with }s\in[-\delta,\delta]\}$ where $v\in J$ and $t\in[0,h]$ and that each such segment has length $2\delta$. 
Finally observe each one of the sets $H_{v\rho_t}$ contains exactly one vector, namely $v\rho_t$, which lands in $J$ when we geodesic flow it for some time in $[-h,0]$. It follows that for all intervals $\BI\subset\BR$ and $w\in T^1\Sigma$ we have
$$\left\vert\left\vert\left\{r\in\BI\,\middle\vert\begin{array}{c}\exists s\in[-\delta,\delta],t\in[-h,0]\\ \text{ with }wh_s\rho_t\in J\end{array}\right\}\right\vert-\frac{\ell(\{s\in\BI\text{ with }wh_s\in\CJ\})}{2\delta}\right\vert\le 2$$
where $\CJ=\CJ(J,h,\delta)$ and where the number $2$ is there to take into account possible over counting near the ends of the horospherical segment. 

Using then the fact that when $L\to\infty$ the sphere $S(x_0,L+h)\subset\BH^2$ looks more and more like a horosphere, one gets the following analogue of \eqref{eq delsart bounds}: for any $\delta'<\delta<\delta''$ and $h'<h<h''$ and $J'\Subset J\Subset J''$ we have for all sufficiently large $L$
$$\int_{I\rho_{L+h}}\chi_{\CJ(J',h',\delta')}\le 2\delta\cdot\vert\BFA_{I,J}(L,h)\vert\le\int_{I\rho_{L+h}}\chi_{\CJ(J'',h'',\delta'')}$$
where $\chi_\CJ$ is the characteristic function of $\CJ$. From here we get that
\begin{align*}
\vert\BFA_{I,J}(L,h)\vert
&\sim \frac 1{2\delta}\int_{I\rho_{L+h}}\chi_{\CJ(J,h,\delta)}\sim\frac 1{2\delta}\cdot\frac{e^{L+h}}2\cdot\int_I\chi_{\CJ(J,h,\delta)}(k\rho_{L+h})\ dk\\
&\stackrel{\text{\eqref{eq equidistribution}}}\sim\frac 1{2\delta}\cdot\frac{e^{L+h}}2\cdot\ell(I)\cdot\frac{\vol_{T^1\Sigma}(\CJ(J,h,\delta))}{\vol_{T^1\Sigma}(T^1\Sigma)}\\
&\sim \frac{\ell(I)\cdot\ell(J)}{4\pi}\cdot\frac{e^{L+h}-e^L}{\vol(\Sigma)}
\end{align*}
We record this slightly generalized version of Delsarte's theorem for later reference:

\begin{sat}[Delarte's theorem for in-out sectors]\label{thm Delsarte's theorem for sectors}
Let $\Sigma$ be a closed hyperbolic surface, let $h$ be positive, and let $I\subset T_{x_0}\Sigma$ and $J\subset T_{y_0}\Sigma$ be non-degenerate segments. Let then $\BFA_{I,J}(L,h)$ be the set of geodesic arcs $\alpha:[0,r]\to\Sigma$ with length $r\in[L,L+h]$, with endpoints $\alpha(0)=x_0$ and $\alpha(r)=y_0$, and with initial and terminal speeds satisfying $\alpha'(0)\in I$ and $\alpha'(r)\in J$. Then we have
$$\vert\BFA_{I,J}(L,h)\vert\sim \frac{\ell(I)\cdot\ell(J)}{4\pi}\cdot\frac{e^{L+h}-e^L}{\vol(\Sigma)}$$
when $L\to\infty$.\qed
\end{sat}

As we mentioned earlier, Theorem \ref{thm Delsarte's theorem for sectors} is a very special case of the much more general \cite[Th\'eor\`eme 4.1.1]{Roblin}. We would not be surprised if there were also other references we are not aware of.

\begin{bem}
Note also that the asymptotic behavior in Theorem \ref{thm Delsarte's theorem for sectors} is uniform as long as $I$ and $J$ belong to a compact set of choices. For example, since the surface $\Sigma$ is compact, we get that for all $\delta>0$ the asymptotic behavior in Theorem \ref{thm Delsarte's theorem for sectors} is uniform as long as $I$ and $J$ are intervals of length at least $\delta$ contained respectively in $T_{x_I}^1\Sigma$ and $T_{x_J}^1\Sigma$ for some $x_I$ and $x_J$ in $\Sigma$.
\end{bem}

Now, let $X$ be a trivalent graph with vertex set $\ver(X)$ and let $\vec x=(x_v)_{v\in\ver(X)}\in\Sigma^{\ver(X)}$ be a $\ver(X)$-tuple of points in the surface. Let also  
$$U\subset\prod_{v\in\ver(X)}\left(\bigoplus_{\bar e\in\hal_v(X)} T^1_{x_v}\Sigma\right)\stackrel{\text{def}}=\BT_{\vec x}$$
be an open set where, as always, $\hal_v(X)$ is the set of all half-edges of $X$ starting at the vertex $v$. 

Given for each edge $e\in\edg(X)$ a positive number $L_e$ and writing $\vec L=(L_e)_{e\in\edg(X)}$ we are going to be interested in the set $\BFG^X_{U}(\vec L_e,h)$ of realizations $\phi:X\to\Sigma$ 
\begin{itemize}
\item[(a)] mapping each vertex $v\in\ver(X)$ to the point $x_v$, 
\item[(b)] with $(\phi(\vec e))_{\bar e\in\hal(X)}\in U$, and 
\item[(c)] with $\ell(e)\in[L_e,L_e+h]$ for all $e\in\edg(X)$. 
\end{itemize}
In this setting we have the following version of Delsarte's theorem relating the cardinality of $\BFG^X_{U}(\vec L_e,h)$ with the volume $\vol_{\vec x}(U)$ of $U$, where the volume is measured in the flat torus $\BT_{\vec x}$.

\begin{sat}[Delarte's theorem for graph realizations]\label{sat delsarte graph}
Let $\Sigma$ be a closed hyperbolic surface, let $X$ be a finite graph, fix $\vec x\in\Sigma^{\ver(X)}$ and an open set $U_{\vec x}\subset\BT_{\vec x}$. If $\vol_{\vec x}(\bar U_{\vec x}\setminus U_{\vec x})=0$, then for every $h>0$ we have
$$\vert\BFG^X_{U_{\vec x}}(\vec L,h)\vert\sim \frac{\vol_{\vec x}(U_{\vec x})}{(4\pi)^{E}}\cdot\frac{(e^h-1)^{E}\cdot e^{\Vert\vec L\Vert}}{\vol(\Sigma)^{E}}$$
when $\min_{e\in\edg(X)}L_e\to\infty$. Here $\bar U_{\vec x}$ is the closure of $U_{\vec x}$ in $\BT_{\vec x}$ and $\vol_{\vec x}(\cdot)$ stands for the volume therein. Also, $E=\vert\edg(X)\vert$ is the number of edges in $X$, and $\Vert\vec L\Vert=\sum_{e\in\edg(X)}L_e$.
\end{sat}

We stress that $\vol_{\vec x}$ is normalized in such a way that $\vol_{\vec x}(\BT_{\vec x})=(2\pi)^{2 E}$. We also stress that this is consistent with the fact that we use the interval $[0,\pi]$ to measure unoriented angles.

\begin{proof}
Let us say that a closed set of the form
$$\prod_{v\in\ver(X)}\left(\prod_{\bar e\in\hal_v(X)} I_{\bar e}\right)\subset\BT_{\vec x}=\prod_{v\in\ver(X)}\left(\bigoplus_{\bar e\in\hal_v(X)} T^1_{x_v}\Sigma\right)$$
where each $I_{\bar e}$ is a segment in $T^1_{x_v}\Sigma$ is {\em a cube}. We say that a closed subset of $\prod_{v\in\ver(X)}\left(\oplus_{\bar e\in\hal_v(X)} T^1_{x_v}\Sigma\right)$ it {\em cubical} if it can be given as the union of finitely many cubes with disjoint interiors. The assumption that the open set $U=U_{\vec x}$ in the statement is such that $\bar U\setminus U$ is a null-set implies that $U$ can be approximated from inside and outside by cubical sets $U'\subset U\subset U''$ with $\vol(U'')-\vol(U')$ as small as we want. It follows that it suffices to prove the theorem if $U$ is (the interior of) a cubical set. 

Note now that if $U=U_1\cup U_2$ is the disjoint union of two sets $U_1$ and $U_2$ and if the statement of the theorem holds true for $U_1$ and $U_2$ then it also holds true for $U$. Since every cubical set is made out of finitely many cubes with disjoint interior we deduce that it really suffices to prove the theorem for individual cubes
$$U=\prod_{v\in\ver(X)}\left(\prod_{\bar e\in\hal_v(X)} I_{\bar e}\right)$$ 
Note that, up to shuffling the factors we can see $U$ as
$$U=\prod_{e\in\edg(X)}\left(I_{e^+}\times I_{e^-}\right).$$ 
Here we are denoting by $e^+$ and $e^-$ the two half-edges of the edge $e\in\edg(X)$. Now, unpacking the notation one sees that a realization $\phi$ belongs to $\BFG_{U}(\vec L_e,h)$ if and only if for all $e\in\edg(X)$ the arc $\phi(e)$ belongs to $\BFA_{I_{e^+},I_{e^-}}(L_e,h)$. It follows thus from Delsarte's theorem for in-out sectors that 
\begin{align*}
\vert\BFG_{U}^X(\vec L,h)\vert
&=\prod_{e\in\edg(X)}\left\vert\BFA_{I_{e^+},I_{e^-}}(L_e,h)\right\vert\\
&\sim\prod_{e\in\edg(X)}\left(\frac{\ell(I_{e^+})\cdot\ell(I_{e^-})}{4\pi}\cdot \frac{(e^h-1)\cdot e^{L_e}}{\vol(\Sigma)}\right)\\
&=\frac{\prod_{e\in\edg(X)}(\ell(I_{e^+})\cdot\ell(I_{e^-}))}{(4\pi)^{E}}\cdot \frac{(e^h-1)^{E}\cdot e^{\sum_{e\in\edg(X)}L_e}}{\vol(\Sigma)^{E}}\\
&=\frac{\vol(U)}{(4\pi)^{E}}\cdot \frac{(e^h-1)^{E}\cdot e^{\Vert\vec L\Vert}}{\vol(\Sigma)^{E}}
\end{align*}
where, as in the statement we have set $E=\vert\edg(X)\vert$. We are done.
\end{proof}

Let us consider now the concrete case we will care mostly about. Suppose namely that $X$ is a trivalent graph and that we want all the angles to be between $\frac 23\pi-\epsilon$ and $\frac 23\pi+\epsilon$. In other words, we are interested in the set of $\epsilon$-critical realizations $\phi:X\to\Sigma$ which map the vertices to our prescribed tuple $\vec x\in\Sigma^{\ver(X)}$. Then we have to consider the set 
$$U_{\vec x,\epsilon-\crit}^X\subset\prod_{v\in\ver(X)}\left(\bigoplus_{\bar e\in\hal_v(X)} T^1_{x_v}\Sigma\right)$$
of those tuples $(v_{\vec e})_{\vec e\in\hal(X)}$ with $\angle(v_{\vec e_1},v_{\vec e_2})\in[\frac{2\pi}3-\epsilon,\frac{2\pi}3+\epsilon]$ for all distinct $\vec e_1,\vec e_2\in\hal(X)$ incident to the same vertex.

Well, let us compute the volume of $U_{\vec x,\epsilon-\crit}^X$. Noting that the conditions on $U_{\vec x,\epsilon-\crit}^X$ associated to different vertices are independent, we get that it suffices to think vertex-by-vertex and then multiply all the numbers obtained for all vertices. For each vertex $v$ label arbitrarily the half-edges incident to $v$ by $\vec e_1,\vec e_2$ and $\vec e_3$. We have no restriction for the position of $v_{\vec e_1}\in T_{x_v}\Sigma$. Once we have fixed $v_{\vec e_1}\in T_{x_v}^1(\Sigma)$ we get that $v_{\vec e_2}$ can belong to two segments, each one of length $2\epsilon$, in $T^1_{x_v}(\Sigma)$---recall that all our angles are unoriented. Then, once we have fixed $v_{\vec e_1}$ and $v_{\vec e_2}$ we have to choose $v_{\vec e_3}$ in an interval of length $2\epsilon-\vert\angle(v_{\vec e_1},v_{\vec e_2})-\frac{2\pi}3\vert$. This means that when we have chosen $v_{\vec e_1}$ and which segment $v_{\vec e_2}$ is in, then we have $3\epsilon^2$ worth of choices of $(v_{\vec e_2},v_{\vec e_3})$. This means that the set of possible choices in $T_{x_v}^1\Sigma\times T_{x_v}^1\Sigma\times T_{x_v}^1\Sigma$ has volume equal to $12\pi\epsilon^2$. Since, as we already mentioned earlier, the conditions at all vertices of $X$ are independent, we get that
$$\vol_{\vec x}(U_{\vec x,\epsilon-\crit}^X)=(12\pi\epsilon^2)^{\vert\ver(X)\vert}$$
From Theorem \ref{sat delsarte graph} we get thus that for all $h>0$ and all $\vec x\in\Sigma^{\ver(X)}$ we have 
$$\vert\BFG^X_{\vec x,\epsilon-\crit}(\vec L,h)\vert\sim \frac{(12\pi\epsilon^2)^{V}}{(4\pi)^{E}}\cdot\frac{(e^h-1)^{E}\cdot e^{\Vert\vec L\Vert}}{\vol(\Sigma)^{E}}\text{ as }\min_{e\in\edg(X)}L_e\to\infty,$$
where we are writing $V$ and $E$ for the number of vertices and edges of $X$, and $\BFG^X_{\vec x,\epsilon-\crit}$ instead of $\BFG^X_{U_{\vec x,\epsilon-\crit}}$. Taking into account that $X$ is trivalent and hence satisfies $V=-2\chi(X)$ and $E=-3\chi(X)$ we can clean this up to 
$$\vert\BFG^X_{\vec x,\epsilon-\crit}(\vec L,h)\vert\sim \epsilon^{4\vert\chi(X)\vert}\cdot\left(\frac{2}{3}\right)^{2\chi(X)}\cdot\pi^{\chi(X)}\cdot\frac{(e^h-1)^{-3\chi(X)}\cdot e^{\Vert\vec L\Vert}}{\vol(\Sigma)^{-3\chi(X)}}$$
as $\min_{e\in\edg(X)}L_e\to\infty$. Moreover, since the geometry of the set $U_{\vec x}^X(\epsilon-\crit)$ is independent of the point $\vec x$, we get that the speed of convergence to this asymptotic is independent of $\vec x$. Altogether have the following result:

\begin{kor}\label{kor fat graph}
Let $\Sigma$ be a closed hyperbolic surface and $X$ be a trivalent graph, fix $\epsilon>0$ and $h>0$, and for $\vec x\in\Sigma^{\ver(X)}$, let $\BFG^X_{\vec x,\epsilon-\crit}(\vec L,h)$ be the set of $\epsilon$-critical realizations $\phi:X\to \Sigma$ mapping the vertex $v$ to the point $x_v=\phi(v)$. Then we have
\begin{equation}\label{eq key delsarte}
\vert\BFG_{\vec x,\epsilon-\crit}^X(\vec L,h)\vert \sim \epsilon^{4\vert\chi(X)\vert}\cdot\left(\frac{2}{3}\right)^{2\chi(X)}\cdot\pi^{\chi(X)}\cdot\frac{(e^h-1)^{-3\chi(X)}\cdot e^{\Vert\vec L\Vert}}{\vol(\Sigma)^{-3\chi(X)}}
\end{equation}
as $\min_{e\in\edg(X)}L_e\to\infty$.\qed
\end{kor}

Note that the set $U_{\vec x,\epsilon-\crit}^X$ needed to establish Corollary \ref{kor fat graph} is basically the same for all choices of $\vec x$. For example, if we take another $\vec y\in\Sigma^{\ver(X)}$ and we identity $\BT_{\vec x}$ with $\BT_{\vec y}$ isometrically by parallel transport (along any collection of curves whatsoever) then $U_{\vec x,\epsilon-\crit}^X$ is sent to $U_{\vec y,\epsilon-\crit}^X$. It follows that we can approximate $U_{\vec x,\epsilon-\crit}^X$ and $U_{\vec y,\epsilon-\crit}^X$ at the same time by cubical sets consisting of the same number of cubes with the same side lengths. It thus follows from the comment following Theorem \ref{thm Delsarte's theorem for sectors} that the asymptotics in Corollary \ref{kor fat graph} is uniform in $\vec x$. We record this fact:

\begin{named}{Addendum to Corollary \ref{kor fat graph}}
 The asymptotics in \eqref{eq key delsarte} is uniform in $\vec x\in\Sigma^{\ver(X)}$.\qed
\end{named}

\section{Counting critical realizations}\label{sec counting critical}
In this section we prove Theorem \ref{thm counting minimising graphs} from the introduction. Before restating the theorem, recall that if $X$ is a trivalent graph then we denote by 
\begin{equation}\label{eq this is an actual name}
\BFG^X(L)=\left\{\begin{array}{c}\phi:X\to\Sigma\text{ critical realization}\\ \text{with length }\ell_\Sigma(\phi)\le L\end{array}\right\}
\end{equation}
the set of all critical realizations of length $\ell_\Sigma(\phi)$ at most $L$. 

\begin{named}{Theorem \ref{thm counting minimising graphs}}
Let $\Sigma$ be a closed, connected, and oriented hyperbolic surface. For every connected trivalent graph $X$ we have
$$\vert\BFG^X(L)\vert\sim 
\left(\frac 23\right)^{3\chi(X)}\cdot\frac{\vol(T^1\Sigma)^{\chi(X)}}{(-3\chi(X)-1)!}\cdot L^{-3\chi(X)-1}\cdot e^{L}$$
as $L\to\infty$.
\end{named}

Fixing for the remaining of this section the trivalent graph $X$ we will write $\ver=\ver(X)$ and $\edg=\edg(X)$ for the sets of vertices and edges and $V=\vert\ver\vert$ and $E=\vert\edg\vert$ for their cardinalities. Similarly we will denote by $\CG=\CG^X$ the manifold of realizations of $X$ in $\Sigma$, and by $\BFG(L)=\BFG^X(L)$ the set of critical realizations with length at most $L$.

The main step in the proof of Theorem \ref{thm counting minimising graphs} is to count critical realizations of $X$ such that the corresponding vector of lengths $(\ell(\phi(e)))_{e\in\edg}$ belongs to a box of size $h>0$. More concretely, we want to count how many elements there are in the set 
\begin{equation}\label{eq critical in a box}
\BFG(\vec L,h)=\left\{\begin{array}{c}\phi:X\to\Sigma\text{ critical realization with}\\ \ell(\phi(e))\in(L_e,L_e+h]\text{ for all }e\in\edg\end{array}\right\}
\end{equation}
where $\vec L=(L_e)_{e\in\edg}$ is a positive vector. We start by establishing some form of an upper bound for the number of homotopy classes of realizations when we bound the length of each individual edge---recall that by Lemma \ref{lem component critical realization} any two homotopic critical realizations are identical.

\begin{lem}\label{lem bounding the rest}
For all $\vec L\in\BR_+^{\edg(X)}$ there are at most $\bconst\cdot e^{\Vert\vec L\Vert}$ homotopy classes of realizations $\phi:X\to\Sigma$ with $\ell(\phi(e))\le L_e$ for all $e\in\edg(X)$.
\end{lem}

It is worth pointing out that Lemma \ref{lem bounding the rest} fails if $\Sigma$ is allowed to have cusps---see Section \ref{sec comments}.

\begin{proof}
Let us fix a point $x_0\in\Sigma$ and note that every point in $\Sigma$---think of the images under a realization of the vertices of $X$---can be moved to $x_0$ along a path of length at most $\diam(\Sigma)$. It follows that every realization $\phi:X\to\Sigma$ is homotopic to a new realization $\psi:X\to\Sigma$ mapping all vertices of $X$ to $x_0$ and with 
\begin{equation}\label{eq ya no hay locos}
\ell(\psi(e))\le\ell(\phi(e))+2\cdot\diam(\Sigma)\le L_e+2\cdot\diam(\Sigma).
\end{equation}
for every edge $e\in\edg(X)$. Note that the homotopy class of $\psi$ is determined by the homotopy classes of the loops $\psi(e)$ when $e$ ranges over the edges of $X$. Now \eqref{eq ya no hay locos} implies that, up to homotopy, we have at most $\bconst\cdot e^{L_e}$ choices for the geodesic segment $\psi(e)$. This implies that there are at most $\bconst\cdot e^{\Vert\vec L\Vert}$ choices for the homotopy class of $\psi$, and hence for the homotopy class of $\phi$. We are done.
\end{proof}

Although it is evidently pretty coarse, Lemma \ref{lem bounding the rest} will play a key role in the proof of Theorem \ref{thm counting minimising graphs}. However, the main tool in the proof of the theorem is the following:

\begin{prop}\label{prop critical in box}
For all $h>0$ we have
$$\vert\BFG(\vec L,h)\vert\sim\frac {2^{4\chi(X)}}{3^{3\chi(X)}}\cdot\pi^{\chi(X)}\cdot
\frac{(e^{h}-1)^{-3\chi(X)}\cdot e^{\Vert\vec L\Vert}}{\vol(\Sigma)^{-\chi(X)}}$$
as $\min_{e\in\edg}L_e\to\infty$. Here $\BFG(\vec L,h)$ is as in \eqref{eq critical in a box}.
\end{prop}

\begin{proof}
Denote by $\CG(\vec L,h)\subset\CG$ the set of all realizations $\phi:X\to\Sigma$ with $\ell_\Sigma(\phi)\in[L_e,L_e+h]$ for all $e\in\edg$ and then let
\begin{align*}
G_{\epsilon}(\vec L,h)&=\left\{ \CG^\phi\in\pi_0(\CG)\text{ with }\CG^{\phi}_{\epsilon-\crit}\subset \CG(\vec L,h)\right\}\\
\hat G_{\epsilon}(\vec L,h)&=\left\{ \CG^\phi\in\pi_0(\CG)\text{ with }\CG^{\phi}_{\epsilon-\crit}\cap\CG(\vec L,h)\neq\emptyset\right\}
\end{align*}
be the sets of connected components of $\CG$ whose set of $\epsilon$-critical realizations is fully contained in (resp. which meet) $\CG(\vec L,h)$. It follows from Corollary \ref{lem critical near kind of critical} that there is some $\ell_0$ such that as long as $\vec L$ satisfies that $\min L_e\ge\ell_0$, then each component listed in $\hat G_\epsilon(\vec L,h)$ contains exactly one critical realization of the graph $X$. Assuming from now on that we are in this situation we get that 
$$\vert G_{\epsilon}(\vec L,h)\vert\le\vert\BFG(\vec L, h)\vert\le\vert\hat G_\epsilon(\vec L,h)\vert.$$

Now, from \eqref{eq prop sum up section 5 1} in Proposition \ref{prop sum up section 5} we get that for all $\delta>0$ there is $\ell_1>\ell_0$ with
$$(1-\delta)\cdot\vol\left(\bigcup_{\CG^\phi\in G_\epsilon(\vec L,h)}\CG^\phi_{\epsilon-\crit}\right)<\left(\frac 2{\sqrt 3}\epsilon^2\right)^V\cdot\vert G_{\epsilon}(\vec L,h)\vert$$
$$(1+\delta)\cdot\vol\left(\bigcup_{\CG^\phi\in\hat G_\epsilon(\vec L,h)}\CG^\phi_{\epsilon-\crit}\right)>\left(\frac 2{\sqrt 3}\epsilon^2\right)^V\cdot \vert \hat G_{\epsilon}(\vec L,h)\vert$$
whenever $\epsilon$ is small enough and $\min L_e\ge\ell_1$. Altogether we get that for all $\epsilon$ positive and small we have
\begin{align*}
(1-\delta)\cdot\left(\frac 2{\sqrt 3}\epsilon^2\right)^{-V}\cdot\vol\left(\bigcup_{\CG^\phi\in G_\epsilon(\vec L,h)}\CG^\phi_{\epsilon-\crit}\right)&<\vert\BFG(\vec L, h)\vert\\
(1+\delta)\cdot\left(\frac 2{\sqrt 3}\epsilon^2\right)^{-V}\cdot\vol\left(\bigcup_{\CG^\phi\in\hat G_\epsilon(\vec L,h)}\CG^\phi_{\epsilon-\crit}\right)&>\vert\BFG(\vec L, h)\vert
\end{align*}
for all $\vec L$ with $\min L_e\ge\ell_1$.

We get now from \eqref{eq prop sum up section 5 2} in Proposition \ref{prop sum up section 5} that there is $\ell_2>\ell_1$ such that, as long as $\epsilon$ is under some threshold, we have that whenever $\phi,\psi\in\CG$ are homotopic $\epsilon$-critical realizations with $\ell(\phi(e)),\ell(\psi(e))\ge\ell_2$ for all $e\in\edg$, then we have that the lengths $\ell(\phi(e))$ and $\ell(\psi(e))$ differ by at most $2\epsilon$ for each edge $e\in\edg$. This implies that for any such $\vec L$ and $\epsilon$ we have
\begin{align*}
\CG_{\epsilon-\crit}(\vec L+[2\epsilon],h-4\epsilon)&\subset \bigcup_{\CG^\phi\in G_\epsilon(\vec L,h)}\CG^\phi_{\epsilon-\crit}\\
\CG_{\epsilon-\crit}(\vec L-[2\epsilon],h+4\epsilon)&\supset \bigcup_{\CG^\phi\in\hat G_\epsilon(\vec L,h)}\CG^\phi_{\epsilon-\crit}
\end{align*}
where $\vec L+[t]\in\BR^{\edg}$ is the vector with entries $(\vec L+[t])_e=\vec L_e+t$.

Let us summarize what we have obtained so far:
\begin{align*}
(1-\delta)\cdot\left(\frac 2{\sqrt 3}\epsilon^2\right)^{-V}\cdot\vol\left(\CG_{\epsilon-\crit}(\vec L+[2\epsilon],h-4\epsilon)\right)&<\vert\BFG(\vec L, h)\vert\\
(1+\delta)\cdot\left(\frac 2{\sqrt 3}\epsilon^2\right)^{-V}\cdot\vol\left(\CG_{\epsilon-\crit}(\vec L-[2\epsilon],h+4\epsilon)\right)&>\vert\BFG(\vec L, h)\vert.
\end{align*}
Our next goal is to compute the volumes on the left. Using the cover \eqref{eq cover} we can compute volumes $\vol(\CG_{\epsilon-\crit}(\vec L,h))$ by integrating over $\Sigma^{\ver}$ the cardinality of the intersection 
$$\BFG_{\vec x,\epsilon-\crit}^X(\vec L,h)=\Pi^{-1}(\vec x)\cap\CG_{\epsilon-\crit}(\vec L,h)$$
of the fiber $\Pi^{-1}(\vec x)$ with the set we care about. In light of Corollary \ref{kor fat graph} we get in this way that
\begin{align*}
\vol(\CG_{\epsilon-\crit}(\vec L,h))
&=\int_{\Sigma^{\ver}} \left\vert\BFG_{\vec x,\epsilon-\crit}^X(\vec L,h)\right\vert\, d\vec x\\
&\stackrel{\text{Cor. \ref{kor fat graph}}}\sim\int_{\Sigma^{\ver}} \epsilon^{4\vert\chi(X)\vert}\cdot\left(\frac{2}{3}\right)^{2\chi(X)}\cdot\pi^{\chi(X)}\cdot\frac{(e^h-1)^{-3\chi(X)}\cdot e^{\Vert\vec L\Vert}}{\vol(\Sigma)^{-3\chi(X)}}\, d\vec x\\
&=\epsilon^{4\vert\chi(X)\vert}\cdot\left(\frac{2}{3}\right)^{2\chi(X)}\cdot\pi^{\chi(X)}\cdot\frac{(e^h-1)^{-3\chi(X)}\cdot e^{\Vert\vec L\Vert}}{\vol(\Sigma)^{-3\chi(X)}}\vol(\Sigma)^{V}\\
&=\epsilon^{4\vert\chi(X)\vert}\cdot\left(\frac{2}{3}\right)^{2\chi(X)}\cdot\pi^{\chi(X)}\cdot\frac{(e^h-1)^{-3\chi(X)}\cdot e^{\Vert\vec L\Vert}}{\vol(\Sigma)^{-\chi(X)}}
\end{align*}
where we have used that $V=-2\chi(X)$ and where the asymptotics hold true when $\min_eL_e\to\infty$. This means that whenever $\min_eL_e$ is large enough we have
\begin{align*}
(1-\delta)\cdot\frac {2^{4\chi(X)}}{3^{3\chi(X)}}\cdot\pi^{\chi(X)}\cdot
\frac{(e^{h-4\epsilon}-1)^{-3\chi(X)}\cdot e^{\sum_{e\in\edg}(L_e+2\epsilon)}}{\vol(\Sigma)^{-\chi(X)}}&<\vert\BFG(\vec L, h)\vert\\
(1+\delta)\cdot\frac {2^{4\chi(X)}}{3^{3\chi(X)}}\cdot\pi^{\chi(X)}\cdot
\frac{(e^{h+4\epsilon}-1)^{-3\chi(X)}\cdot e^{\sum_{e\in\edg}(L_e-2\epsilon)}}{\vol(\Sigma)^{-\chi(X)}}&>\vert\BFG(\vec L, h)\vert
\end{align*}
Since this is true for all $\epsilon>0$ we get that
$$(1-\delta)\cdot\frac {2^{4\chi(X)}}{3^{3\chi(X)}}\cdot\pi^{\chi(X)}\cdot
\frac{(e^{h}-1)^{-3\chi(X)}\cdot e^{\Vert\vec L\Vert}}{\vol(\Sigma)^{-\chi(X)}}\le\vert\BFG(\vec L, h)\vert$$
$$(1+\delta)\cdot\frac {2^{4\chi(X)}}{3^{3\chi(X)}}\cdot\pi^{\chi(X)}\cdot
\frac{(e^{h}-1)^{-3\chi(X)}\cdot e^{\Vert\vec L\Vert}}{\vol(\Sigma)^{-\chi(X)}}\ge\vert\BFG(\vec L, h)\vert$$
and hence, since for all $\delta>0$ we can choose $L$ large enough so that the above bounds hold, we have
$$\frac {2^{4\chi(X)}}{3^{3\chi(X)}}\cdot\pi^{\chi(X)}\cdot
\frac{(e^{h}-1)^{-3\chi(X)}\cdot e^{\Vert\vec L\Vert}}{\vol(\Sigma)^{-\chi(X)}}\sim\vert\BFG(\vec L, h)\vert$$
as we wanted to prove.
\end{proof}

Armed with Lemma \ref{lem bounding the rest} and Proposition \ref{prop critical in box} we can now prove the theorem:

\begin{proof}[Proof of Theorem \ref{thm counting minimising graphs}]
Let $h>0$ be small, and for $\vec n\in \mathbb{N}^{\edg}$ consider, with the same notation as in \eqref{eq critical in a box}, the set $\BFG(h\cdot\vec n,h)$. Setting
$$\Delta(N)=\{\vec n\in\BN^{\edg}\text{ with }\Vert\vec n\Vert\le N\}$$
where $\Vert\vec n\Vert=\sum_e n_e$, note that
\begin{equation}\label{eq proof theorem bla}
\sum_{\vec n\in\Delta(N-E)}\vert\BFG(h\cdot\vec n,h)\vert\le\vert\BFG(h\cdot N)\vert\le\sum_{\vec n\in\Delta(N)}\vert\BFG(h\cdot\vec n,h)\vert
\end{equation}
where $\BFG(h\cdot N)=\BFG^X(h\cdot N)$ is as in \eqref{eq this is an actual name} and where, once again, $E=\vert\edg\vert$ is the number of edges of the graph $X$. Finally, write 
$$\kappa=\frac {2^{4\chi(X)}}{3^{3\chi(X)}}\cdot(\pi\cdot\vol(\Sigma))^{\chi(X)}$$
Proposition \ref{prop critical in box} now reads as
$$\vert\BFG(h\cdot\vec n,h)\vert\sim\kappa\cdot(e^{h}-1)^{-3\chi(X)}\cdot e^{h\cdot\Vert\vec n\Vert}$$
where the asymptotic holds for fixed $h$ when $\min\vec n_e\to\infty$. This means that for all $h$ and $\delta$ there is $n(h,\delta)$ with 
$$\vert\BFG(h\cdot\vec n,h)\vert> (\kappa-\delta)\cdot(e^{h}-1)^{-3\chi(X)}\cdot e^{h\cdot\Vert\vec n\Vert}$$
and
$$\vert\BFG(h\cdot\vec n,h)\vert< (\kappa+\delta)\cdot(e^{h}-1)^{-3\chi(X)}\cdot e^{h\cdot\Vert\vec n\Vert}$$
for all $\vec n$ with $\min\vec n_e\ge n(h,\delta)$. It follows thus from the left side of \eqref{eq proof theorem bla} that
\begin{align*}
\vert\BFG(h\cdot N)\vert
&\ge\sum_{\tiny\begin{array}{c}\vec n\in\Delta(N-E),\\ \min\vec n_e\ge n(h,\delta)\end{array}}\vert\BFG(h\cdot\vec n,h)\vert\\
&> (\kappa-\delta)(e^h-1)^{-3\chi(X)}\sum_{\tiny\begin{array}{c}\vec n\in\Delta(N-E),\\ \min\vec n_e\ge n(h,\delta)\end{array}}e^{h\cdot\Vert\vec n\Vert}\\
&= (\kappa-\delta)(e^h-1)^{-3\chi(X)}\sum_{K=0}^{N-E}P(K)\cdot e^{h\cdot K}
\end{align*}
where $P(K)$ is the number of those $\vec n\in\BN^{\edg}$ with $\Vert n\Vert=K$ and $\min\vec n_e\ge n(h,\delta)$. As  $K$ tends to $\infty$ we have $P(K)\sim\frac 1{(E-1)!}K^{E-1}$. Taking into account that $E=-3\chi(X)$ we get that for all $N$ large enough we have
\begin{align*}
\vert\BFG(h\cdot N)\vert
&\succeq\frac{\kappa-\delta}{(-3\chi(X)-1)!}(e^h-1)^{-3\chi(X)}\sum_{K=0}^{N-E}K^{-3\chi(X)-1}\cdot e^{h\cdot K}\\
&=\frac{\kappa-\delta}{(-3\chi(X)-1)!}\left(\frac{e^h-1}h\right)^{-3\chi(X)}\sum_{K=0}^{N-E}(hK)^{-3\chi(X)-1}\cdot e^{h\cdot K}\cdot h\\
&\succeq \frac{\kappa-\delta}{(-3\chi(X)-1)!}\left(\frac{e^h-1}h\right)^{-3\chi(X)}\int_0^{(N-E)h} x^{-3\chi(X)-1}e^xdx
\end{align*}
where the symbol $\succeq$ means that asymptotically the ratio between the left side and the right side is at least $1$. 
When $N\to\infty$ then the value of the integral is asymptotic to $((N-E)\cdot h)^{-3\chi(X)-1}\cdot e^{(N-E)h}$, and this means that for all $N$ large enough we have
$$\vert\BFG(h\cdot N)\vert\succeq \frac{\kappa-\delta}{(-3\chi(X)-1)!}\left(\frac{e^h-1}h\right)^{-3\chi(X)}(Nh-Eh)^{-3\chi(X)-1}\cdot e^{Nh-Eh}$$
This being true for all $\delta$ and all $h$, and replacing $Nh$ by $L$, we have
$$\vert\BFG(L)\vert\succeq \frac{\kappa}{(-3\chi(X)-1)!}L^{-3\chi(X)-1}\cdot e^{L}$$
as $L\to\infty$.
In other words, we have established the desired asymptotic lower bound.
\medskip

Starting with the upper bound we get, again for $h$ positive and small, from the right side in \eqref{eq proof theorem bla} that
$$\vert\BFG(h\cdot N)\vert\le\sum_{\tiny\begin{array}{c}\vec n\in\Delta(N),\\ \min\vec n_e\ge n(h,\delta)\end{array}}\vert\BFG(h\cdot\vec n,h)\vert+\sum_{\tiny\begin{array}{c}\vec n\in\Delta(N),\\ \min\vec n_e\le n(h,\delta)\end{array}}\vert\BFG(h\cdot\vec n,h)\vert.$$
The same calculation as above yields that 
\begin{equation}\label{eq waiting for DHL1}
\begin{split}
\sum_{\tiny\begin{array}{c}\vec n\in\Delta(N),\\ \min\vec n_e\ge n(h,\delta)\end{array}}&\vert\BFG(h\cdot\vec n,h)\vert\preceq \\ &\preceq\frac{\kappa+\delta}{(-3\chi(X)-1)!}\left(\frac{e^h-1}h\right)^{-3\chi(X)}(h(N+E))^{-3\chi(X)-1}\cdot e^{h(N + E)}
\end{split}
\end{equation}
as $N\to\infty$. On the other hand we get from Lemma \ref{lem bounding the rest} that there is $C>0$ with $\vert\BFG(h\cdot\vec n,h)\vert\le C\cdot e^{h\cdot\vert\vec n\vert}$ for all $\vec n$. This means thus that
$$\sum_{\tiny\begin{array}{c}\vec n\in\Delta(N),\\ \min\vec n_e\le n(h,\delta)\end{array}}\vert\BFG(h\cdot\vec n,h)\vert\le C\cdot\sum_{K=1}^NQ(K)\cdot e^{h\cdot K}$$
where $Q(K)$ is the number of those $\vec n\in\BN^{\edg}$ with $\min\vec n_e< n(h,\delta)$ and $\vert\vec n\vert=K$. When $K\to\infty$ the function $Q(K)$ is asymptotic to $C'\cdot K^{E-2}$ for some positive constant $C'$, meaning that we have 
$$\sum_{\tiny\begin{array}{c}\vec n\in\Delta(N),\\ \min\vec n_e\le n(h,\delta)\end{array}}\vert\BFG(h\cdot\vec n,h)\vert\le \frac{(1+\delta)\cdot C\cdot C'}{h^{E-1}}\cdot\sum_{K=1}^N h\cdot(hK)^{E-2}\cdot e^{h\cdot K}$$
for all $L$ large enough. A similar estimation as the one above yields thus that there is another positive constant $C''$ with 
\begin{equation}\label{eq waiting for DHL2}
\sum_{\tiny\begin{array}{c}\vec n\in\Delta(N),\\ \min\vec n_e\le n(h,\delta)\end{array}}\vert\BFG(h\cdot\vec n,h)\vert\le C''\cdot (N\cdot h)^{-3\chi(X)-2}\cdot e^{N\cdot H}
\end{equation}
The quantity in the right hand side of \eqref{eq waiting for DHL2} is negligible when compared to the right hand side of \eqref{eq waiting for DHL1}, and this means that we have
$$\vert\BFG(h\cdot N)\vert\preceq \frac{\kappa+2\delta}{(-3\chi(X)-1)!}\left(\frac{e^h-1}h\right)^{-3\chi(X)}(h(N+E))^{-3\chi(X)-1}\cdot e^{h(N+E)}$$
for all large $N$. Since this holds true for all $\delta$ and all $h$, replacing $hN$ by $L$ we deduce that 
$$\vert\BFG(L)\vert\preceq \frac{\kappa}{(-3\chi(X)-1)!}L^{-3\chi(X)-1}\cdot e^{L}.$$
Having now also established the upper asymptotic bound, we are done with the proof of the theorem.
\end{proof}

Before moving on to other matters, we include an observation that we will use later on. The basic strategy of the proof of Theorem \ref{thm counting minimising graphs} was to decompose the problem of counting all critical realizations of at most some given length into the problem of counting those whose edge lengths are in a given box and then adding over all boxes. For most boxes, Proposition \ref{prop critical in box} gives a pretty precise estimation for the number of critical realizations in the box, and from Lemma \ref{lem bounding the rest} we get an upper bound for all boxes. We used these two to get the desired upper bound in the theorem, deducing that we could ignore the boxes where Proposition \ref{prop critical in box} does not apply. A very similar argument implies that the set of critical realizations where some edge is shorter than $\ell_0$ is also negligible. We state this observation as a lemma:

\begin{lem}\label{lem most critical realizations are long}
For all $\ell$, all but a negligible set of critical realizations are $\ell$-long.\qed
\end{lem}

It is probably clear from the context what negligible means here, but to be precise, we mean that the set is negligible inside the set of all critical realizations in the sense of \eqref{eq defi negligible} below.

\section{Fillings}\label{sec fillings}

Let $S_g$ be a compact, connected and oriented surface of genus $g$ and with one boundary component. Below we will be interested in continuous maps
\begin{equation}\label{eq defi boundary}
\beta:S_g\to\Sigma
\end{equation}
which send $\D S_g$ to a closed geodesic $\gamma=\beta(S_g)$. We will refer to such a map as a {\em filling of genus $g$ of $\gamma$}, a {\em genus $g$ filling of $\gamma$}, or just simply as a {\em filling of $\gamma$} when the genus $g$ is either undetermined or understood from the context. The genus of a curve $\gamma\subset\Sigma$ is the infimum of all $g$'s for which there is a genus $g$ filling \eqref{eq defi boundary} with $\gamma=\beta(\D S_g)$. Note that the genus of a curve is infinite unless $\gamma$ is homologically trivial, that is unless $\gamma$ is represented by elements in the commutator subgroup of $\pi_1(\Sigma)$. Indeed, as an element of $\pi_1(S_g)$ the boundary $\D S_g$ is a product of $g$ commutators. It follows that if a curve $\gamma$ in $\Sigma$ has genus $g$ then it is, when considered as an element in $\pi_1(\Sigma)$, a product of $g$ commutators. Conversely, if $\gamma$ is a product of $g$ commutators then there is a map as in \eqref{eq defi boundary} with $S_g$ of genus $g$. In a nutshell, what we have is that the genus and the commutator length of $\gamma$ agree. We record this fact for ease of reference:

\begin{lem}
The genus of a curve agrees with its commutator length.\qed
\end{lem}

Continuing with the same notation and terminology, suppose that a curve $\gamma$ in $\Sigma$ has genus $g$. We then refer to any $\beta:S\to\Sigma$ as in \eqref{eq defi boundary} with $S$ of genus $g$ as a {\em minimal genus filling}. Minimal genus fillings have very nice topological properties. Indeed, suppose that $\beta:S_g\to\Sigma$ is a filling as in \eqref{eq defi boundary} and suppose that there is an essential simple curve $m\subset S_g$ with $\beta(m)$ homotopically trivial. Then, performing surgery on the surface $S_g$ and the map $\beta$ we get a smaller genus filling $\beta':S_{g'}\to\Sigma$ with $\beta'(\D S_{g'})=\beta(\D S_g)$ and $g'<g$. It follows that if $\beta:S_g\to\Sigma$ is a minimal genus filling for $\gamma=\beta(\D S_g)$ then $\beta$ is {\em geometrically incompressible} in the sense that there are no elements in $\ker(\beta_*:\pi_1(S_g)\to\pi_1(\Sigma))$ which are represented by simple curves. We record this fact:

\begin{lem}\label{lem minimal filling is incompressible}
Minimal genus fillings are geometrically incompressible.\qed
\end{lem}

From now on we will be working exclusively with minimal genus fillings and the reader can safely add the words ``minimal genus" every time they see the word ``filling". 

We will not be that much interested in individual fillings, but rather in homotopy classes of fillings. In particular, we will allow ourselves to select particularly nice fillings. More precisely, we will be working with {\em hyperbolic fillings}, by what we will understand a particular kind of pleated surface. We remind the reader that according to Thurston \cite{Thurston} (see also \cite{CEG}) a pleated surface is a map from a hyperbolic surface to a hyperbolic manifold with the property that every point in the domain is contained in a geodesic segment which gets mapped isometrically.

\begin{defi*}
A filling $\beta:S\to\Sigma$ is {\em hyperbolic} if $S$ is endowed with a hyperbolic metric with geodesic boundary and if the map $\beta$ is such that every $x\in S$ is contained in the interior of a geodesic arc $I$ such that $\beta$ maps $I$ isometrically to a geodesic arc in $\Sigma$.
\end{defi*}

An important observation is that if $\beta:S\to\Sigma$ is a hyperbolic filling, if $x\in\D S$, and if $I\subset S$ is a geodesic segment with $x$ in its interior, then $I\subset\D S$. It follows that hyperbolic fillings map the boundary isometrically.

\begin{lem}
The restriction of any hyperbolic filling $\beta:S\to\Sigma$ to $\partial S$ is geodesic. \qed
\end{lem}

We should not delay making sure that hyperbolic fillings exist:

\begin{prop}\label{prop existence hyperbolic filling}
Every minimal genus filling is homotopic to a hyperbolic filling. 
\end{prop}

To prove this proposition we will first show that if $\beta:S\to\Sigma$ is any filling, and if the surface $S$ admits a triangulation with certain properties, then $\beta$ is homotopic to a hyperbolic filling. For lack of a better name we will say that a triangulation $\CT$ of $S$ is {\em useful} if it satisfies the following three conditions:
\begin{itemize}
\item[(i)] $\CT$ has exactly $g+1$ vertices $v_0,\dots,v_g$, with $v_0\in\D S$ and the others in the interior of $S$.
\item[(ii)] There is a collection of edges $I_0,\dots,I_g$ of $\CT$ such that both endpoints of $I_i$ are attached to $v_i$ for all $i=0,\dots,g$. Moreover $I_0=\D S$.
\end{itemize}
These two conditions are evidently pretty soft. It is the condition we will state next what makes useful triangulations actually useful. We first need a bit of notation: if $I$ is any edge of $\CT$ other than $I_0,\dots,I_g$ then let $G_I$ be the connected component of $I\cup I_0\cup\dots\cup I_g$ containing $I$. 
\begin{itemize}
\item[(iii)] For any edge $I$ other than $I_0,\dots,I_g$ we have that the image under $\beta_*:\pi_1(G_I)\to\pi_1(\Sigma)$ is not abelian.
\end{itemize}
Note that $\pi_1(G_I)$ is a a free group of rank $2$. This means that its image $\beta_*(\pi_1(G_I)))$ has at most rank $2$. Since we are assuming that it is not abelian, we actually get that it is a rank $2$ free group. Free groups being Hopfian we deduce that $\beta_*:\pi_1(G_I)\to\pi_1(\Sigma)$ is an isomorphism onto its image. In other words, $\beta_*$ is injective on $\pi_1(G_I)$.

The following result makes clear why we care about such triangulations:

\begin{lem}\label{lem spinning is possible}
If the domain $S$ of a filling $\beta:S\to\Sigma$ admits a useful triangulation $\CT$, then $\beta$ is homotopic to a hyperbolic filling.
\end{lem}
\begin{proof}
As we just discussed, our conditions imply that $\beta_*$ is injective on $\pi_1(G_I)$ for all $I$. This implies in particular that each one of the exceptional edges $I_0,\dots,I_g$ of $\CT$ closes up to a simple closed curve $\gamma_0,\dots,\gamma_g$ in $S$ which is mapped to a homotopically essential curve. This in turn means that the images of $\gamma_0,\dots,\gamma_g$ are homotopic to non-trivial closed geodesics. Since all these $g+1$ curves are mutually disjoint we can then homotope $\beta$ so that $\beta(\gamma_i)$ is a closed geodesic $\hat\gamma_i$ for all $i$. 

Let now $I$ be one of the remaining edges of $\CT$, let $v_i$ and $v_j$ be its (possibly equal) endpoints and note that $G_I=\gamma_i\cup I\cup\gamma_j$. Since $\beta_*$ is injective on $\pi_1(G_I,v_i)$ we know that the elements $\beta_*(\gamma_i)$ and $\beta_*(I*\gamma_j*I^{-1})$ do not commute and hence have distinct fixed points in $\D_\infty\BH^2$. This seemingly weak property is all we need to run the standard construction of pleated surfaces by spinning the edges of the triangulation over the geodesics $\hat\gamma_i$---compare with the proof of Theorem I.5.3.6 in \cite{CEG}.
\end{proof}

We can now prove Proposition \ref{prop existence hyperbolic filling}.

\begin{proof}
Let $\beta:S\to\Sigma$ be a minimal genus filling of a geodesic $\gamma=\beta(\partial S)$, say of genus $g$. In light of Lemma \ref{lem spinning is possible}, to prove that $\beta$ is homotopic to a hyperbolic filling it suffices to show that $S$ admits a useful triangulation. Well, let us start by taking $g$ disjoint compact one-holed tori $T_1,\ldots,T_g\subset S$. We claim that the restriction of $\beta$ to each $T_i$ is $\pi_1$-injective. Indeed, since we know that $\beta$ is geometrically incompressible we deduce that $\D T_i$ is not in the kernel of the induced homomorphism at the fundamental group level. It follows that the image of $\beta_*(\pi_1(T_i))$ cannot be abelian. Now, this implies that $\beta_*(\pi_1(T_i))$ is free, and evidently of rank 2. Hence, again since free groups are Hopfian, we get that the restriction of $\beta_*$ to $\pi_1(T_i)$ is injective for all $i$.

Now, why do we care about that? Well, knowing that the restriction of $\beta_*$ to $\pi_1(T_i)$ is injective for any $i$ implies that when the images under $\beta$ of the non-boundary parallel simple closed curves in $T_i$ determine infinitely many conjugacy classes of maximal abelian subgroups of $\pi_1(\Sigma)$. We can thus choose for each $i=1,\dots,g$ a non-boundary parallel simple closed curve $\gamma_i\subset T_i$ such that if we also set $\gamma_0=\D S$ then we have that 
\begin{itemize}
\item[(*)] no two of the the maximal abelian subgroups of $\pi_1(\Sigma)$ containing $\beta_*(\gamma_0),\dots,\beta_*(\gamma_g)$ are conjugate to each other. 
\end{itemize}
Choosing now a vertex $v_i$ in each one of the curves $\gamma_i$ we get from (*) that if $I$ is any simple path joining $v_i$ to $v_j$ for $i\neq j$, then the image of $\beta_*(\pi_1(\gamma_i\cup I\cup \gamma_j))$ is not abelian. 

The upshot of all of this is that any triangulation $\CT$ with 
\begin{enumerate}
\item vertex set $v_0,\dots,v_g$, 
\item such that there are edges $I_0,\dots,I_g$ incident on both ends to $v_i$ and with image $\gamma_i$, and
\item such that all other edges connect distinct endpoints,
\end{enumerate}
is useful. To see that such a triangulation exists cut $S$ along $\gamma_1,\dots,\gamma_g$. When doing this we get a $2g+1$ holed sphere $\Delta$ and each vertex $v_1,\dots,v_g$ arises twice---denote the two copies of $v_i$ by $v_i$ and $v_i'$ as in Figure \ref{fig:triangulation}. Now any triangulation of $\Delta$ with vertex set $v_0,v_1,v_1',v_2,v_2',\dots,v_g,v_g'$ and which 
\begin{itemize}
\item contains a sequence of $2g-1$ edges yielding a path 
$$[v_1,v_2],[v_2,v_3],\dots,[v_g,v_1'],[v_1',v_2'],\dots,[v_{g-1}',v_g']$$ 
and
\item all other edges incident to $v_0$
\end{itemize}
yields a triangulation of $S$ satisfying (1)--(3) above. Having proved that there is a useful triangulation, we get from Lemma \ref{lem spinning is possible} that $\beta$ is homotopic to a hyperbolic filling, as we needed to prove.
\end{proof}

\begin{figure}[h]
\leavevmode \SetLabels
\L(.235*.48) $v_1$\\%
\L(.295*.48) $v_2$\\%
\L(.41*.48) $v_g$\\%
\L(.47*.48) $v'_1$\\%
\L(.53*.48) $v'_2$\\%
\L(.66*.48) $v'_g$\\%
\L(.455*1.05) $v_0$\\%
\endSetLabels
\begin{center}
\AffixLabels{\centerline{\includegraphics[width=0.6\textwidth]{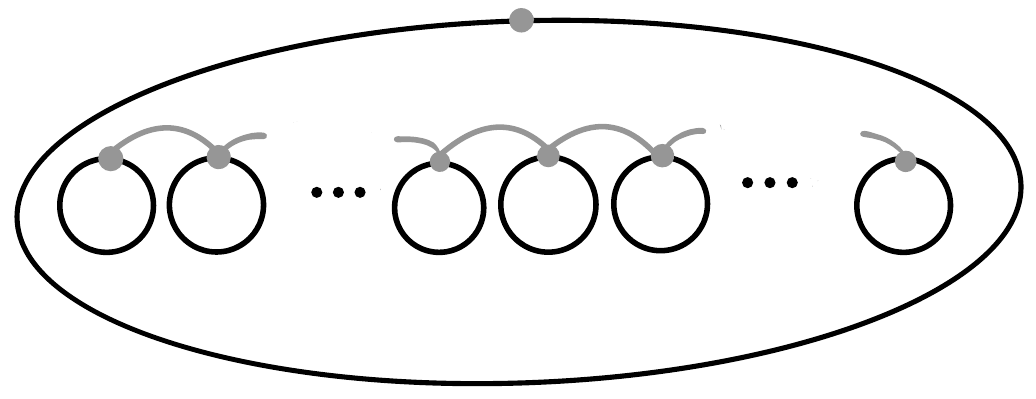}}\hspace{1cm}}
\vspace{-24pt}
\end{center}
\caption{Construction of the useful triangulation in Proposition \ref{prop existence hyperbolic filling}: The $2g+1$ holed sphere obtained from $S$ by cutting along $\gamma_1, \ldots, \gamma_g$. The vertices $v_1,\dots,v_g$ and their copies $v_1',\dots,v_g'$ are joined by the sequence of edges $[v_1,v_2],[v_2,v_3],\dots,[v_g,v_1'],[v_1',v_2'],\dots,[v_{g-1}',v_g']$. To complete the triangulation add as many edges as needed joining $v_0$ to the other vertices.} 
\label{fig:triangulation}
\end{figure}
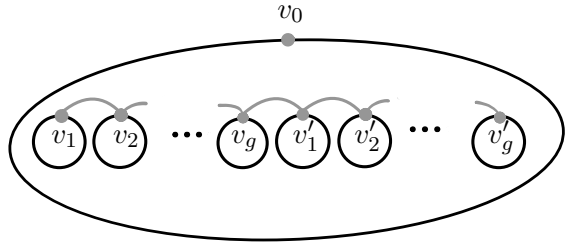

Being able to work with hyperbolic fillings is going to be key in the next section, where we bound the number of closed geodesics $\gamma$ in $\Sigma$ with length $\ell_\Sigma(\gamma)\le L$ and which admit at least two homotopically distinct fillings. 

\begin{defi*}
Two fillings $\beta_1:S_1\to\Sigma$ and $\beta_2:S_2\to\Sigma$ are {\em homotopic} if there is a homeomorphism $F:S_1\to S_2$ such that $\beta_1$ is homotopic to $\beta_2\circ F$.
\end{defi*}

Evidently, to bound the number of pairs of non-homotopic fillings with the same boundary we will need some criterion to decide when two fillings are homotopic. To be able to state it note that if $\beta_1:S_1\to\Sigma$ and $\beta_2:S_2\to\Sigma$ are homotopic hyperbolic fillings, then there is
\begin{quote}
a closed hyperbolic surface $S=S_1\cup_{\D S_1=\D S_2} S_2$ obtained by isometrically gluing $S_1$ and $S_2$ along the boundary, in such a way that there is a pleated surface $\Theta:S\to\Sigma$ with $\Theta\vert_{S_i}=\beta_i$.
\end{quote}
For lack of a better name we refer to $\Theta:S\to\Sigma$ as the {\em pseudo-double} associated to $\beta_1$ and $\beta_2$, and to the curve $\D S_1=\D S_2\subset S$ as the {\em crease} of the pseudo-double.

Our criterion to decide if the two hyperbolic fillings $\beta_1$ and $\beta_2$ of a geodesic $\gamma$ are homotopic will be in terms of the structure of the $\epsilon_0$ thin part of the domain of the associated pseudo-double, where we choose once and forever the constant 
\begin{equation}\label{eq margulis lemma}
\epsilon_0=\frac 1{10}\cdot\min\left\{\text{Margulis constant of }\BH^2,\ \text{systole of }\Sigma\right\}.
\end{equation}
The following is our criterion.

\begin{lem}\label{lem criterion fillings to be homotopic}
Suppose that $\beta_1:S_1\to\Sigma$ and $\beta_2:S_2\to\Sigma$ are hyperbolic minimal genus fillings of genus $g$ of a geodesic $\gamma$, let $\Theta:S\to\Sigma$ be the associated pseudo-double, and denote its crease $\partial S_1 = \partial S_2\subset S$ also by $\gamma$. If 
\begin{enumerate}
\item the $\epsilon_0$-thin part of $S$ has $6g-3$ connected components $U_1,\dots,U_{6g-3}$, and
\item $\gamma$ traverses each $U_i$ exactly twice, 
\end{enumerate}
then $\beta_1$ and $\beta_2$ are homotopic.
\end{lem}
\begin{proof}
Note that the surface $S$ has genus $2g$. In particular, the assumption on the $\epsilon_0$-thin part implies that there is a pants decomposition $P$ in $S$ consisting of closed geodesics of length at most $2\epsilon_0$. Now, since $\Theta$ is evidently 1-Lipschitz and since $2\epsilon_0$ is less than the systole of $\Sigma$ we get that each one of the components of $P$ is mapped to a homotopically trivial curve. The assumption that the crease $\gamma$, a simple curve, intersects each component of the pants decomposition exactly twice implies that $\gamma$ cuts each pair of pants into two hexagons. Paint blue those in $S_1$ and yellow those in $S_2$. 

We can now construct a homotopy relative to the crease as follows. Each component of $P$ consists of a blue arc and a yellow arc whose juxtaposition is homotopically trivial. This implies that we can homotope all yellow arcs, relative to their endpoints to the corresponding blue arcs. Extend this homotopy to a homotopy fixing $\D S_2$ and defined on the whole of $S_2$. Now the boundary of each yellow hexagon is mapped to the boundary of each blue hexagon. Since $\Sigma$ has trivial $\pi_2$ we deduce that those two hexagons are homotopic. Proceeding like this with each hexagon we get a homotopy relative to the crease of the yellow parts of $S$ to the blue part. This is what we wanted to get.
\end{proof}

\section{Bounding the number of multifillings}\label{sec fillings2}

The goal of this section is to prove Theorem \ref{bounding multi-fillings}:

\begin{named}{Theorem \ref{bounding multi-fillings}}
For any $g$ there are at most $\bconst\cdot L^{6g-5}\cdot e^{\frac L2}$ genus $g$ closed geodesics $\gamma$ in $\Sigma$ with length $\ell(\gamma)\le L$ and with two non-homotopic fillings $\beta_1:S_1\to\Sigma$ and $\beta_2:S_2\to\Sigma$ of genus $\le g$.
\end{named}

The proof of Theorem \ref{bounding multi-fillings} turns out to be kind of involved. We hope that the reader will not despair with all the weird objects and rather opaque statements they will find below.

\subsection*{Wired surfaces}
Under a {\em wired surface} we will understand a compact connected simplicial complex $\Delta$ obtained as follows: start with a compact, possibly disconnected, triangulated surface $\Surf(\Delta)$ and an, evidently finite, subset $P_\Delta$ of the set of vertices of the triangulation of $\Surf(\Delta)$ such that every connected component of $\Surf(\Delta)\setminus P_\Delta$ has negative Euler-characteristic. We think of the elements in $P_\Delta$ as {\em plugs}. Now attach 1-simplices, to which we will refer as {\em wires}, by attaching both end-points to plugs, and do so in such a way that each plug arises exactly once as the end-point of a wire---we denote the set of wires by $\wire(\Delta)$. A wired surface $\Delta$ without wires, that is one with $\Delta=\Surf(\Delta)$, is said to be {\em degenerate}. Otherwise it is {\em non-degenerate}.
\medskip

How will wired surfaces arise? We will say that a pair $(\BF,\BT)$ is a {\em decoration} of a surface $S$ if 
\begin{itemize}
\item $\BF$ is a partial foliation of $S$ supported by the union of a finite collection of disjoint, essential and non-parallel, cylinders, and
\item $\BT$ is a triangulation of the complement of the interior of those cylinders.
\end{itemize}
Now, if $(\BF,\BT)$ is a decoration of $S$, then we get an associated wired surface $\Delta=S/\sim_\BF$ by collapsing each leaf of $\BF$ to a point and dividing each one of the arising bigons into two triangles by adding a vertex in the interior of the bigon. We will say that the quotient map $\pi:S\to\Delta$ is a {\em resolution} of $\Delta$ with {\em associated foliation} $\BF$. Every wired surface admits an essentially unique resolution, unique in the sense that any two differ by a PL-homeomorphism mapping one of the foliations to the other one. 
\medskip

Suppose now that $\Delta$ is a wired surface. A {\em simple curve} on $\Delta$ is a map $\eta:\BS^1\to\Delta$ such that there are a resolution $\pi:S\to\Delta$ with associated foliation $\BF$ and an essential simple curve $\eta':\BS^1\to S$ which is transversal to $\BF$ and with $\eta=\pi\circ\eta'$. 

Note that transversality to the associated foliation implies that if $\eta$ is a simple curve of a wired surface $\Delta$ then $\eta\cap\pi^{-1}(\Delta\setminus\Surf(\Delta))$ consists of a collection of segments, each one of them mapped homeomorphically to a wire. If $I$ is such a wire then we denote by $n_I(\eta)$ the {\em weight} of $\eta$ in $I$, that is the number connected components of $\eta\cap\pi^{-1}(\Delta\setminus\Surf(\Delta))$ which are mapped homeomorphically to $I$, or in other words, the number of times that $\eta$ crosses $I$. We refer to the vector
\begin{equation}\label{eq vector of weights wired}
\vec n_\Delta(\eta)=(n_I(\eta))_{I\in\wire(\Delta)}
\end{equation}
as the {\em weight vector for $\eta$ in $\Delta$}. The intersection of the image of $\eta$ with the surface part $\Surf(\Delta)$ is a simple arc system with endpoints in the set $P_\Delta$. Note that up to homotopying $\eta$ to another simple curve in $\Delta$ we might assume that all the components of the arc system $\eta\cap\Surf(\Delta)$ are essential. This is equivalent to asking that for each wire $I$ we have $n_I(\eta)\le n_I(\eta')$ for any other simple curve $\eta'$ in $\Delta$ homotopic to $\eta$. We will suppose from now on, without further mention, that all simple curves in $\Delta$ satisfy these minimality requirements.

So far, wired surfaces are just topological objects. Let us change this. Under a {\em hyperbolic wired surface} we understand a wired surface $\Delta$ whose surface part $\Surf(\Delta)$ is endowed with a piece-wise hyperbolic metric, that is, one with respect to which the simplexes in the triangulation of $\Surf(\Delta)$ are isometric to hyperbolic triangles.

Let $\Delta$ be a hyperbolic wired surface, and as always let $\Sigma$ be our fixed hyperbolic surface. We will say that a map $\Xi:\Delta\to\Sigma$ is {\em tight} if the following holds:
\begin{itemize}
\item $\Xi$ maps every wire to a geodesic segment, and
\item $\Xi$ is an isometry when restricted to each one of the simplexes in the triangulation of $\Surf(\Delta)$.
\end{itemize}
We will be interested in counting pairs $(\Xi:\Delta\to\Sigma,\gamma)$ consisting of tight maps and simple curves. Evidently, without further restrictions, there could be infinitely many such pairs. What we are going to count is pairs were the curve has bounded length. We are however going to use a pretty strange notion of length. Consider namely for some given small but positive $\epsilon$ the following quantity
\begin{equation}\label{eq horrible length bound}
\ell_\Xi^\epsilon(\gamma)=\epsilon\cdot\ell_{\Surf(\Delta)}(\gamma\cap\Surf(\Delta))+\sum_{I\in\wire}n_I(\gamma)\cdot\max\left\{\ell_\Sigma(\Xi(I))-\frac 1\epsilon,0\right\}
\end{equation}
It is evident that this notion of length is exactly taylored to what we will need later on, but let us try to parse what \eqref{eq horrible length bound} actually means. What is the role of $\epsilon$? Well, if we think of the length as a measure of the cost of a journey, then the first $\epsilon$ just makes traveling along the surface part pretty cheap, meaning that for the same price we can cruise longer over there. Along the same lines, when traveling through the wires, we only pay when the wires are very long. 

\begin{lem}\label{lem weird lemma}
Let $\Delta$ be a non-degenerate hyperbolic wired surface with set of wires $\wire=\wire(\Delta)$. Fix a tight map $f:\Delta\to\Sigma$, a positive integer vector $\vec n=(n_I)_{I\in\wire}\in\BN_+^{\wire}$, and denote by 
$$\min=\min_{I\in\wire}n_I\ge 1\text{ and }d=\vert\{I\in \wire\text{ with }n_I=\min\}\vert$$ 
the smallest entry of $\vec n$ and the number of times that this value is taken. 

For any $\epsilon>0$ there are at most $\bconst\cdot L^{d-1}\cdot e^{\frac{L}{\min}}$ homotopy classes of pairs $(\Xi:\Delta\to\Sigma,\gamma)$ where $\Xi$ is a tight map with $\Xi\vert_{\Surf(\Delta)}=f\vert_{\Surf(\Delta)}$ and where $\gamma$ is a simple multicurve in $\Delta$ with $n_I(\gamma)\ge n_I$ for every wire $I$ and with $\ell_\Xi^\epsilon(\gamma)\le L$.
\end{lem}

Note that the fact that the obtained bound, or rather its rate of growth, does not depend on $\epsilon$ implies that actually the only way to get many homotopy classes is to play with the wires. In fact, since the given bound only depends on $d$ and $\min$, the only wires that matter are those which the curve crosses as little as possible. 

Another comment before launching the proof. Namely, what happens if the wired surface $\Delta$ in Lemma \ref{lem weird lemma} is degenerate? Well, if there are no wires, then $\Delta$ is nothing other than a surface with a (piece-wise hyperbolic) metric. In such a surface there are at most $\bconst\cdot L^{3\cdot\vert\chi(\Delta)\vert}$ simple multi-curves of length at most $L$---see for example \cite{book}. This means that for a degenerate wired surface one actually gets a polynomial bound instead of an exponential one.

We are now ready to launch the proof of Lemma \ref{lem weird lemma}:

\begin{proof}
Note that, since the map $\Xi$ is fixed on $\Surf(\Delta)$, we get that the homotopy type of the map, or even the map itself, is determined by what happens to the wires. In particular, as in the proof of Lemma \ref{lem bounding the rest} we get that if we give ourselves a positive vector $\vec\lambda=(\lambda_I)_{I\in\wire}\in\BR_+^{\wire}$, then there are at most $\bconst\cdot e^{\Vert\vec\lambda\Vert}$ homotopy classes of tight maps $\Xi:\Delta\to\Sigma$ with $\Xi\vert_{\Surf(\Delta)}=f$ and such that 
\begin{equation}\label{eq man I am sick of this}
\lambda_I\le \ell_\Sigma(\Xi(I))\le\lambda_I+1\text{ for all }I\in\wire.
\end{equation}
As always, we have set $\Vert\vec\lambda\Vert=\lambda_1+\dots+\lambda_r$. 

Now, we are not counting homotopy classes of maps, but rather of pairs $(\Xi:\Delta\to\Sigma,\gamma)$ where the multicurve $\gamma$ satisfies $\ell_\Xi^\epsilon(\gamma)\le L$. Note that, if $\Xi$ satisfies \eqref{eq man I am sick of this} then our given length bound $\ell_\Xi^\epsilon(\gamma)\le L$ implies that
\begin{align*}
\ell_{\Surf(\Delta)}(\gamma\cap\Surf(\Delta))
&= \frac 1\epsilon\left(\ell_\Xi^\epsilon(\gamma)-\sum_{I\in\wire}n_I(\gamma)\cdot\max\left\{\ell_\Sigma(\Xi(I))-\frac 1\epsilon,0\right\}\right)\\
&\le \frac 1\epsilon\left(L-\sum_{I\in\wire}n_I\cdot\max\left\{\lambda_I-\frac 1\epsilon,0\right\}\right)\\
&\le \frac 1\epsilon\left(L-\langle\vec n,\vec\lambda\rangle+\frac 1\epsilon\cdot\Vert\vec n\Vert\right)
\end{align*}
Now, since for any given length there are only polynomially many simple arc systems of bounded length we deduce that for each $\Xi$ satisfying \eqref{eq man I am sick of this} there are at most $\bconst\cdot(L-\langle\vec n,\vec\lambda\rangle)^{\bconst}$ homotopy classes of simple multicurves $\gamma$ in $\Delta$ with $\ell_\Xi^\epsilon(\gamma)\le L$ and satisfying $n_{I}(\gamma)\ge n_{I}$ for each wire $I$. Putting all of this together we get the following:

\begin{fact}
There are at most $\bconst\cdot (L-\langle\vec n,\vec\lambda\rangle)^{\bconst}\cdot e^{\Vert\lambda\Vert}$ homotopy classes of pairs $(\Xi:\Delta\to\Sigma,\gamma)$ where $\Xi$ is a tight map with $\Xi\vert_{\Surf(\Delta)}=f$, satisfying \eqref{eq man I am sick of this}, and where $\gamma$ is a simple curve in $\Delta$ with $n_I(\gamma)\ge n_I$ for every wire $I$ and with $\ell_\Xi^\epsilon(\gamma)\le L$.\qed
\end{fact}

Now we get that the quantity we want to bound, that is the number of homotopy classes of pairs $(\Xi:\Delta\to\Sigma,\gamma)$ where $\Xi$ is a tight map with $\Xi\vert_{\Surf(\Delta)}=f$ and where $\gamma$ is a simple curve in $\Delta$ with $n_I(\gamma)\ge n_I$ for every wire $I$ and with $\ell_\Xi^\epsilon(\gamma)\le L$ is bounded from above by
$$\sum_{\lambda\in\BN^{\wire},\ \Vert\lambda\Vert\le L}\bconst\cdot (L-\langle\vec n,\vec\lambda\rangle)^{\bconst}\cdot e^{\Vert\lambda\Vert}$$
This quantity is then bounded from above by the value of the integral
$$\bconst\int_{\{\vec x\in\BR_+^{\wire},\ \langle\vec n,\vec x\rangle\}\le L}(L-\langle \vec n,\vec x\rangle)^{\bconst}\cdot e^{\Vert\vec x\Vert}dx$$
and now it is a calculus problem that we leave to the reader to check that this integral is bounded by $\bconst\cdot L^{d-1}\cdot e^{\frac{L}{n_{\min}}}$ where $n_{\min}=\min_In_I\ge 1$ and where  $d=\vert\{I\text{ wire with }n_I=n_{\min}\}\vert$. We are done.
\end{proof}

At this point we know how to bound the number of homotopy classes of tight maps of wired surfaces. It is time to explain why we care about being able to do so.

\subsection*{Pseudo-doubles}
Earlier, just before the statement of Lemma \ref{lem criterion fillings to be homotopic} we introduced the pseudo-double associated to two fillings. Let us extend that terminology a bit: Under a {\em pseudo-double} we understand a pair $(\Theta:S\to\Sigma,\gamma)$ where 
\begin{itemize}
\item $\Theta:S\to\Sigma$ is a pleated surface with $S$ closed,
\item $\gamma\subset S$, the {\em crease}, is a simple curve cutting $S$ into two connected components, and
\item $\Theta$ maps $\gamma$ to a geodesic in $\Sigma$ and its restriction $\Theta\vert_{S\setminus\gamma}$ to the complement of $\gamma$ is geometrically incompressible.
\end{itemize}
Two pseudo-doubles $(\Theta:S\to\Sigma,\gamma)$ and $(\Theta':S'\to\Sigma,\gamma')$ are {\em homotopic} if there is a homeomorphism $f:S\to S'$ with $\gamma'$ homotopic to $f(\gamma)$ and with $\Theta$ homotopic to $\Theta'\circ f$. 
\medskip

Note that this terminology is consistent with the use of the word pseudo-double in the previous section.
\medskip

Recall now that we fixed earlier some $\epsilon_0$ satisfying \eqref{eq margulis lemma} and note that if $(\Theta:S\to\Sigma,\gamma)$ is a pseudo-double then the crease $\gamma$, being separating, crosses every component $U$ of the thin part $S^{\le\epsilon_0}$ an even number of times $\iota(\gamma,U)\in 2\BN$. In fact, since by the choice of $\epsilon_0$ we get that $\Theta$ maps the core of every component of the thin part to a homotopically trivial curve, and since the restriction of $\Theta$ to $S\setminus\gamma$ is geometrically incompressible we get that actually
\begin{equation}\label{eq cross atleast twice}
\iota(\gamma,U)\ge 2\text{ for all components }U\text{ of the thin part of }S,
\end{equation}
by what we mean that $\gamma$ traverses each such $U$ at least twice. 

Our next goal is to bound, for growing $L$, the number of homotopy classes of pseudo-doubles $(\Theta:S\to\Sigma,\gamma)$ where $S$ has given topological type, where there are precisely $d$ components $U$ of the thin part with $\iota(\gamma,U)=2$, and where $\ell_S(\gamma)\le L$:

\begin{prop}\label{prop this is a pain in the butt}
Let $\epsilon_0$ be as in \eqref{eq margulis lemma} and suppose that $S_0$ is a closed orientable surface. 
\begin{enumerate}
\item For every $d\ge 1$ there are at most $\bconst\cdot L^{d-1}\cdot e^{\frac L2}$ homotopy classes of pseudo-doubles $(\Theta:S\to\Sigma,\gamma)$ where $S$ is homeomorphic to $S_0$, where the thin part $S^{\le\epsilon_0}$ of $S$ has exactly $d$ components $U$ with $\iota(\gamma,U)=2$, and where $\gamma$ has length $\ell_S(\gamma)\le L$. 
\item There are at most $\bconst\cdot e^{\frac L3}$ homotopy classes of pseudo-doubles $(\Theta:S\to\Sigma,\gamma)$ where $S$ is homeomorphic to $S_0$, where there is no component $U$ of the thin part $S^{\le\epsilon_0}$ with $\iota(\gamma,U)=2$, and where $\gamma$ has length $\ell_S(\gamma)\le L$. 
\end{enumerate}
\end{prop}

Recall that we declared a {\em decoration} $(\BF,\BT)$ of a surface $S$ to be a pair consisting of partial foliation $\BF$ supported by the union of disjoint essential and non-parallel cylinders, and of a triangulation $\BT$ of the complement of the interior of those cylinders. To see where these decorations come from, assume that $S$ is a hyperbolic surface.
\begin{itemize}
\item The $\epsilon_0$-thick part of $S$ is metrically bounded in the sense that it has a triangulation $\BT$ whose vertices are $\frac 1{10}\epsilon_0$-separated, and whose edges have length at most $\frac 13\epsilon_0$ and are geodesic unless contained in the boundary $\D(S^{\ge \epsilon_0})$ of the $\epsilon_0$-thick part---in that case they have just constant curvature. 
\item The components of the $\epsilon_0$-thin part  $S^{\le\epsilon_0}$ are not metrically bounded but still have a very simple structure: they are cylinders foliated by constant curvature circles, namely the curves at constant distance from the geodesic at the core of the cylinder. 
\end{itemize}
Putting these things together, that is the triangulation $\BT$ of the thick part and the foliation $\BF$ of the thin part, we get what we will refer as a {\em thin-thick decoration $(\BF,\BT)$ of $S$}---note that the triangulation $\BT$ is not unique, and this is why we use an undetermined article. More importantly, note also that the number of components of the thin part is bounded just in terms of the topology of $S$, and that the number of vertices in the triangulation $\BT$ is bounded by some number depending on the chosen $\epsilon_0$ and again the topological type of $S$. Since we have fixed $\epsilon_0$, it follows that every compact surface $S_0$ admits finitely many decorations $(\BF_1,\BT_1),\dots,(\BF_r,\BT_r)$ such that if $S$ is any hyperbolic surface homeomorphic to $S_0$ then there is a homeomorphism $\sigma:S_0\to S$ and some $i$ such that $(\sigma(\BF_i),\sigma(\BT_i))$ is a $\epsilon_0$-thin-thick decoration of $S$. We state this fact for later reference:

\begin{lem}\label{lem finite decorations}
For every closed surface $S_0$ there are finitely many decorations $(\BF_1,\BT_1),\dots,(\BF_r,\BT_r)$ such that for any hyperbolic surface $S$ homeomorphic to $S_0$ there are $i\in\{1,\dots,r\}$ and a homeomorphism $\sigma:S_0\to S$ such that $\sigma(\BF_i,\BT_i)$ is a $\epsilon_0$-thin-thick decoration of $S$.\qed
\end{lem}

After these comments, we can finally launch the proof of Proposition \ref{prop this is a pain in the butt}:

\begin{proof}[Proof of Proposition \ref{prop this is a pain in the butt}]
Starting with the proof of (1), note that from Lemma \ref{lem finite decorations} we get that it suffices to prove for each fixed decoration $(\BF,\BT)$ of $S_0$ that
\begin{itemize}
\item[{\bf (*)}]  there are at most $\bconst\cdot L^{d-1}\cdot e^{\frac L2}$ homotopy classes of pseudo-doubles $(\Theta:S\to\Sigma,\gamma)$ where there is a homeomorphism $\sigma:S_0\to S$ such that $(\sigma(\BF),\sigma(\BT))$ is a decoration of the $\epsilon_0$-thin-thick decomposition of $S$, where the thin part $S^{\le\epsilon_0}$ of $S$ has exactly $d$ components $U$ with $\iota(\gamma,U)=2$, and where $\gamma$ has length $\ell_S(\gamma)\le L$.
\end{itemize}
Assuming from now on that we have an $\epsilon_0$-thin-thick decoration $(\BF, \BT)$, let $\Delta=S_0/\sim_\BF$ be the wired surface obtained from $S$ by collapsing each leaf of $\BF$ to a point and let $\pi:S_0\to\Delta$ be the corresponding quotient map. 

The reason why we consider $\Delta$ is that, as we already mentioned earlier, we have that by the choice of $\epsilon_0$ all the leaves of the foliation $\sigma(\BF)$ are mapped by $\Theta$ to homotopically trivial curves. This implies that there is a map
$$\Xi:\Delta\to\Sigma$$ 
mapping the wires of $\Delta$ to geodesic segments and with $\Theta\circ\sigma$ homotopic to $\Xi\circ\pi$ by a homotopy whose tracks are bounded by $\bconst$. In particular, the edges in $\BT$ are mapped to paths homotopic to geodesic paths of at most length $\bconst$. Pulling tight relative to the vertices, we can assume that $\Xi$ maps 2-dimensional simplices in the triangulation of $\Delta$ to hyperbolic triangles. Note that the bound on the lengths of the images of the edges of $\BT$ imply that the restriction of $\Xi$ to $\Surf(\Delta)$ is $\bconst$-Lipschitz.


This uniform Lipschitz bound implies that $\Xi\vert_{\Surf(\Delta_0)}$ belongs to finitely many homotopy classes and that the tracks of the homotopy to any chosen representative of the correct homotopy class are bounded by $\bconst$. Choosing the representatives to be tight maps with respect to some hyperbolic structure on $\Delta_0$ we get:

\begin{fact}
There are finitely many hyperbolic structures $\Delta_1,\dots,\Delta_r$ on $\Delta$ and finitely many tight maps $f_1,\dots,f_r:\Delta_i\to\Sigma$, such that for any $\Theta:S\to\Sigma$ and $\sigma:S_0\to S$ as in (*) there are $i\in\{1,\dots,r\}$ and a tight map $\Xi:\Delta_i\to\Sigma$ with $\Xi\vert_{\Surf(\Delta_i)}=f_i\vert_{\Surf(\Delta_i)}$ and such that $\Theta\circ\sigma:S_0\to\Sigma$ is homotopic to $\Xi\circ\pi:S_0\to\Sigma$ by a homotopy whose tracks have at most length $\bconst$.\qed
\end{fact}

Continuing with the same notation, note that the bound on the tracks of the homotopy between $\Theta\circ\sigma$ and $\Xi\circ\pi$ means that when we compare the length of the geodesic $\gamma$ in $S$ with that of the curve $\eta=(\pi\circ\sigma^{-1})(\gamma)$ in the hyperbolic wired surface $\Delta_i$ then there is at most an increase by an additive amount every time $\gamma$ crosses a simplex of the wired surface. It means that lengths 
\begin{itemize}
\item increase at most by an additive amount every time we cross a component of the thin part, and
\item increase by a multiplicative amount while we are in the thick part. 
\end{itemize}
Said in other words: there is some $R$ with
\begin{align*}
\ell_\Sigma(\Xi(\eta\cap\Surf(\Delta)))&\le R\cdot \ell_S(\gamma\cap S^{\ge\epsilon_0})\\
\ell_\Sigma(\Xi(\eta\setminus\Surf(\Delta)))&\le \sum_{\kappa\in\pi_0(\gamma\cap S^{\le\epsilon_0})}(\ell_S(\kappa)+R)
\end{align*}
Recalling that the wires $I$ of $\Delta_0$ correspond to the thin parts of $S^{\le\epsilon_0}$ and that the weight $n_I(\eta)$ is nothing other than number of times that $\eta$ crosses the wire $I$, we get from the last two inequalities that
\begin{equation}\label{eq my god this has no end}
\frac 1R\cdot\ell_\Xi(\eta\cap\Surf(\Delta))+\sum_{I\in\wire(\Delta_0)}n_I(\eta)\cdot\max\left\{\ell_\Xi(I)-R,0\right\}\le\ell_S(\gamma)
\end{equation}
where the ``max" arises because a length is always non-negative. Anyways, note that with the notation introduced in \eqref{eq horrible length bound} we can rewrite \eqref{eq my god this has no end} as
$$\ell_\Xi^{\frac 1R}(\gamma)\le\ell_S(\gamma)$$
Note also that from \eqref{eq cross atleast twice} we get that $n_I(\eta)\ge 2$ for all $I\in\wire(\Delta_0)$. This means that using the notation of Lemma \ref{lem weird lemma} we can restate the assumptions in Proposition \ref{prop this is a pain in the butt} (1) as follows: $\min=2$ and this value is achieved $d\ge 1$ times. Lemma \ref{lem weird lemma} implies thus that there are at most $\bconst\cdot L^{d-1}\cdot e^{\frac L2}$ homotopy classes of pairs $(\Xi,\eta)$ arising from pseudo-doubles $(\Theta:S\to\Sigma,\gamma)$ where $\Theta$ is as in the Fact and where $\ell_S(\gamma)\le L$. This implies a fortiori that we have at most that many choices for the homotopy class of $\Xi$. Since the homotopy class of $\Xi$ determines that of $\Theta$, we are done with the proof of (1). The proof of (2) is pretty much identical, the only difference is that now $\min\ge 4$. Plugging this in the argument above we get the bound $\bconst\cdot L^{k-1}\cdot e^{\frac L4}$ for some $k$. This is evidently a stronger bound that $\bconst\cdot e^{\frac L3}$, and we are done. 
\end{proof}

\subsection*{The fruits of our labor}
After all this work we are now ready to prove Theorem \ref{bounding multi-fillings}.

\begin{proof}[Proof of Theorem \ref{bounding multi-fillings}]
Suppose that $\gamma$ is a closed geodesic with length $\ell_\Sigma(\gamma)\le L$ and such that there are two non-homotopic minimal genus fillings $\beta_1:S_1\to\Sigma$ and $\beta_2:S_2\to\Sigma$ with $\beta_i(\D S_i)=\gamma$. From Proposition \ref{prop existence hyperbolic filling} we get that, without loss of generality we might assume that both $\beta_1:S_1\to\Sigma$ and $\beta_2:S_2\to\Sigma$ are hyperbolic fillings. 

Let then $S=S_i\cup_{\D S_1=\D S_2}S_2$ be the hyperbolic surface obtained by gluing both surfaces $S_1$ and $S_2$ along the boundary in such a way that there is a pleated surface $\Theta:S\to\Sigma$ with $\Theta\vert_{S_i}=\beta_i$. Note once again that the map $\Theta$ maps the crease $\hat\gamma=\D S_1=\D S_2$ geodesically to $\gamma$. Moreover, Lemma \ref{lem minimal filling is incompressible} implies that the restriction of $\Theta$ to $S\setminus\hat\gamma$ is geometrically incompressible. Taken together, all of this means that the pair $(\Theta:S\to\Sigma,\hat\gamma)$ is a pseudo-double. 

Now, Lemma \ref{lem criterion fillings to be homotopic} implies, together with the assumption that $\beta_1$ and $\beta_2$ are not homotopic, that the $\epsilon_0$-thin part of $S$ has at most $6g-4$ connected components which are traversed twice by the crease $\hat\gamma$. We thus get from Proposition \ref{prop this is a pain in the butt} that there are at most $\bconst\cdot L^{6g-5}\cdot e^{\frac L2}$ choices for the homotopy class of $(\Theta:S\to\Sigma,\hat\gamma)$. Since the geodesic $\gamma$ is determined by the homotopy class of $\Theta(\hat\gamma)$, we have proved that there are $\bconst\cdot L^{6g-5}\cdot e^{\frac L2}$ choices for $\gamma$, as we had claimed.
\end{proof}

\section{Proof of the main theorem}\label{sec main}

In this section we prove Theorem \ref{thm counting curves} from the introduction. 

\begin{named}{Theorem \ref{thm counting curves}}
Let $\Sigma$ be a closed, connected, and oriented hyperbolic surface and for $g\ge 1$ and $L>0$ let $\BFB_g(L)$ be as in \eqref{eq set we want to count}. We have 
$$\vert\BFB_g(L)\vert\sim \frac {2}{12^g\cdot g!\cdot(3g-2)!\cdot\vol(T^1\Sigma)^{2g-1}}\cdot L^{6g-4}\cdot e^{\frac L2}$$
as $L\to\infty$.
\end{named}

Before we can even explain the idea of the proof of Theorem \ref{thm counting curves} we need to recall what fat graphs are, and a few of their properties:

\subsection*{Fat graphs}
A {\em fat graph} $X$ is a graph endowed with a cyclic ordering of the set $\hal_v$ for each $v$---fat graphs are also sometimes called {\em ribbon graphs}. Every fat graph is endowed with a canonically built thickening $\neigh(X)$, the {\em thickening of $X$}. For the sake of concreteness let us discuss this in the particular case that $X$ is trivalent. Well, we start by taking an oriented filled-in hexagon $G_v$ for every vertex $v$, see Figure \ref{fig:fat}. If we label by $a,b,c$ the three elements in $\hal_v$, given in the correct cyclic order, then we label the boundary components of $G_v$ by $a,ab,b,bc,c,ca$, also given in the correct cyclic order. Now, for every edge $e\in\edg(X)$ let $v_1,v_2\in\ver(X)$ be the two (possibly identical) vertices at which the two half-edges $\vec e_1\in\hal_{v_1}$ and $\vec e_2\in\hal_{v_2}$ corresponding to $e$ are based, and identify in an orientation reversing way the $\vec e_1$-edge of $\D G_{v_1}$ with the $\vec e_2$-edge of $\D G_{v_2}$. Proceeding like this for all edges we end up with the {\em thickening} $\neigh(X)$ of $X$. 

\begin{figure}[h]
\leavevmode \SetLabels
\L(.23*.5) $\circlearrowleft$\\%
\L(.67*.5) $\circlearrowleft$\\%
\endSetLabels
\begin{center}
\AffixLabels{\centerline{\includegraphics[width=0.7\textwidth]{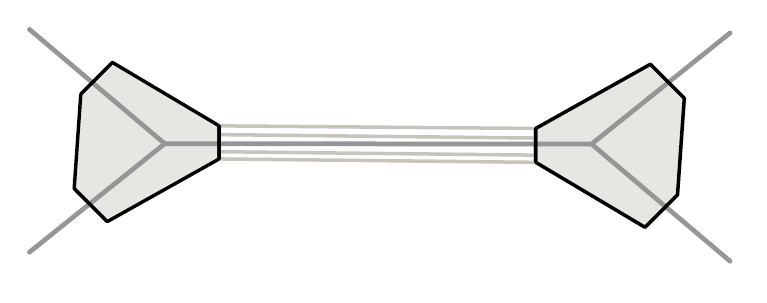}}\hspace{1cm}}
\vspace{-24pt}
\end{center}
\caption{Constructing the thickening of $X$. Pictured are two oriented hexagons corresponding to two vertices connected by an edge. The lighter lines indicate the gluing of those two hexagons.} 
\label{fig:fat}
\end{figure}

\begin{defi*}
A trivalent fat graph $X$ has {\em genus $g$} if $\neigh(X)$ is homeomorphic to a surface of genus $g$ with one boundary component. 
\end{defi*}

One of the quantities that appear in the right side of Theorem \ref{thm counting curves} is the number of genus $g$ fat graphs, or rather, that number weighted by the number of automorphims of each such fat graph, where a map $F:X\to X'$ between two fat graphs is a {\em fat graph homeomorphism} if first it is a homeomorphism between two the underlying graphs, and then if it sends one fat structure to the other one. Anyways, what is really lucky for us is that Bacher and Vdovina \cite{Bacher-Vdovina} have computed that
$$\sum\frac 1{\vert\Aut(X)\vert}=\frac 2{12^g}\cdot\frac{(6g-5)!}{g!\cdot(3g-3)!}$$
where the sum takes place over all homeomorphism classes of genus $g$ fat graphs. 

\begin{bem}
Bacher and Vdovina's result is phrased in terms of 1-vertex triangulations of the closed surface of genus $g$ up to orientation preserving homeomorphism. Let us explain briefly how one goes from such triangulations to genus $g$ fat graphs and back. The dual graph of a triangulation of a surface is a trivalent fat graph---its thickening is the surface minus the vertices of the triangulation. It follows that if the triangulation has a single vertex and the surface is closed of genus $g$ then the fat graph has genus $g$. Conversely, the thickening of a fat graph is equipped with a natural arc system (one arc dual to each edge). When collapsing each boundary component of the thickening to a point one gets a closed surface together with a triangulation with as many vertices as connected components of the boundary. It follows that if $X$ is a genus $g$ fat graph then one gets a 1-vertex triangulation of a genus $g$ surface.
\end{bem}

Let $X$ now be a fat graph and, for the sake of concreteness assume that it has genus $g$. By construction we have a canonical embedding of $X$ into $\neigh(X)$ whose image is a spine. In particular there is a retraction
$$\spine:\neigh(X)\to X$$
such that the pre-image of every vertex is a tripod and the pre-image of every point in the interior of and edge in $X$ is a segment. The image of $\D\neigh(X)$ under $\spine$ runs twice over every edge of $X$. We will refer to this parametrized curve in $X$ as $\D X$. 

\begin{bem}
Note that reversing the orientation of $X$, that is reversing the cyclic order at each vertex, has the effect of reversing the orientation of $\D X$.
\end{bem}

\subsection*{The map}
We will reduce the proof of Theorem \ref{thm counting curves} to the fact that we know how to count critical realizations of graphs, that is to Theorem \ref{thm counting minimising graphs}. Let us explain the basic idea. For given $g$ consider the set
$$\BFX_{g}=\left\{(X,\phi)\middle\vert\begin{array}{c}X\text{ is a fat graph of genus }g\text{ and}\\
\phi:X\to\Sigma\text{ is a criticial realization}\\ \text{of the underlying graph}\end{array}\middle\}\right/_{\text{equiv}}$$
of equivalence classes of realizations, where two realizations $\phi:X\to\Sigma$ and $\phi':X'\to\Sigma$ are {\em equivalent} if there is a fat graph homeomorphism $\sigma:X\to X'$ such that $\phi=\phi'\circ\sigma$. 

\begin{quote}
{\bf We stress something very important:} the critical realization $\phi$ for $(X,\phi)\in\BFX_{g}$ is explicitly not assumed to respect the fat structure. On the other hand, the homeomorphism $\sigma$ in the definition of equivalence definitively has to preserve the fat structure.
\end{quote}

Note that if $(X,\phi)$ is a equivalent to $(X',\phi')$ then the curves $\phi(\D X)$ and $\phi'(\D X')$ are freely homotopic to each other. In particular we have a well-defined map
\begin{equation}\label{eq nielsen map}
\Lambda:\BFX_{g}\to\BFC,\ \ (X,\phi)\mapsto\text{ geodesic homotopic to }\phi(\D X)
\end{equation}
where $\BFC$ is, as it was earlier, the collection of all oriented geodesics in $\Sigma$. 

The basic idea of the proof of Theorem \ref{thm counting curves} is that the map \eqref{eq nielsen map} has the following informally stated properties:
\begin{enumerate}
\item $\Lambda$ is basically injective with image basically contained in $\BFB_g$, 
\item $\Lambda$ is basically surjective onto $\BFB_g$, and 
\item generically, the geodesic $\Lambda(X,\phi)$ has length almost exactly equal to $2\cdot\ell(\phi)-C$ for some explicit constant $C$.
\end{enumerate}
Let us start by clarifying the final point. Suppose that $(X,\phi)\in\BFX_g$ is such that $\phi$ is $\ell_0$-long for some large $\ell_0$. Well, since the curve $\D X$ runs exactly twice over each edge $X$ we get that it its image $\phi(\D X)$ consists of $2\edg(X)$ geodesic segments of length at least $\ell_0$ and with $3\cdot\ver(X)=-6\cdot\chi(X)$ corners where it makes angle equal to $\frac{2\pi}3$. We get now from a standard hyperbolic geometry computation (or, if you so wish, from a limiting argument) that when $\ell_0$ is large then, up to an small error depending on $\ell_0$, when we pull tight $\phi(\D X)$ to get $\Lambda(X,\phi)$ we save $\log\frac 43$ at each one of those corners. Taking into account that $\chi(X)=1-2g$ when $X$ has genus $g$, this is what we have proved:

\begin{lem}\label{lem surface has right length}
For every $\delta>0$ there is $\ell_\delta$ with 
$$\left\vert\ell_\Sigma(\Lambda(X,\phi))-\left(2\ell_\Sigma(\phi)-6\cdot(2g-1)\cdot \log\frac 43\right)\right\vert\le\delta$$
for every $(X,\phi)\in\BFX_{g}$ such that $\phi$ is $\ell_\delta$-long.
\qed
\end{lem}

\begin{bem}
Where does $\log\frac 43$ come from? Well, if $\Delta\subset\BH^2$ is an ideal triangle with vertices $\theta_1,\theta_2$ and $\theta_3$ and center $p$, and if $p'$ is the projection of $p$ to the side say $(\theta_1,\theta_2)$ then $\log\frac 2{\sqrt 3}$ is the difference between the values at $p$ and $p'$ of the Buseman function centered at $\theta_1$, and every time we pass by a vertex we basically save twice that amount (see Figure \ref{fig:ideal triangle}).
\end{bem}

\begin{figure}[h]
\leavevmode \SetLabels
\L(.44*.62) $p$\\%
\L(.33*.70) $p'$\\%
\L(.33*.05) $\theta_1$\\%
\L(.59*.05) $\theta_2$\\%
\endSetLabels
\begin{center}
\AffixLabels{\centerline{\includegraphics[width=0.3\textwidth]{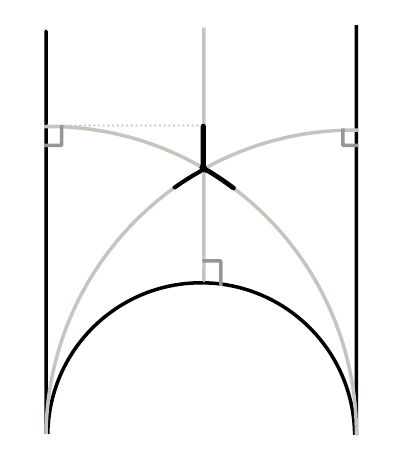}}\hspace{1cm}}
\vspace{-24pt}
\end{center}
\caption{The ideal triangle $\Delta$ with vertices $\theta_1, \theta_2, \theta_3=\infty$ and center $p$. Each one of the legs of the bold printed tripod has length $\log\frac 2{\sqrt 3}$.} 
\label{fig:ideal triangle}
\end{figure}
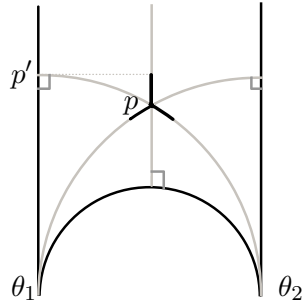

Our next goal is to prove that the map $\Lambda$ is basically bijective onto $\BFB_g$, but we should first formalize what we mean by ``basically". Well, suppose that we have a set $Z$ consisting of either realizations, or curves, or of anything else consisting of elements $\alpha\in Z$ whose length $\ell_\Sigma(\alpha)$ can be measured, and set $Z(L)=\{\alpha\in Z\text{ with }\ell_\Sigma(\alpha)\le L\}$. We will say that a subset 
\begin{equation}\label{eq defi negligible}
W\subset Z\text{ is {\em negligible} in }Z\text{ if }\limsup_{L\to\infty}\frac{\vert W(L)\vert}{\vert Z(L)\vert}=0
\end{equation}
If the ambient set $Z$ is understood from the context, then we might just say that $W$ is {\em negligible.} The complement of a negligible set is said to be {\em generic} and the elements of a generic set are themselves {\em generic}. For example we get from Lemma \ref{lem most critical realizations are long} and Lemma \ref{lem surface has right length} that for all $\delta$ we have that
$$\left\vert\ell_\Sigma(\Lambda(X,\phi))-\left(2\ell_\Sigma(\phi)-2\cdot(6g-3)\cdot \log\frac 43\right)\right\vert\le\delta$$
for $(X,\phi)\in\BFX_{g}$ generic.

\subsection*{Basic bijectivity of $\Lambda$}
Above we used the word ``basically" as meaning that something was true up to negligible sets. Let us start by proving that the map $\Lambda$ is basically injective and that its image is contained in the set $\BFB_g$ of closed geodesics of genus $g$.

\begin{lem}\label{lem pi injective}
There is a generic subset $W\subset\BFX_{g}$ such that the restriction of $\Lambda$ to $W$ is injective and that its image is contained in $\BFB_g$.
\end{lem}
\begin{proof}
Let $\ell_0$ and $C$ be such that for every $(X,\phi)\in\BFX_g$ so that $\phi$ is $\ell_0$-long we have
$$2\cdot\ell_\Sigma(\phi)-C\le\ell_\Sigma(\Lambda(X,\phi))\le 2\ell_\Sigma(\phi)$$
and let $\BFX_{g,\ell_0}$ be the set of those pairs. We get from Lemma \ref{lem most critical realizations are long} that $\BFX_{g,\ell_0}$ is generic in $\BFX_g$. It follows hence from Theorem \ref{thm counting minimising graphs} that 
\begin{equation}
\label{eq sabaton}\vert\{(X,\phi)\in\BFX_{g,\ell_0}\text{ with }\ell_\Sigma(\Lambda(X,\phi))\le L\}\vert\ge\bconst\cdot L^{6g-4}\cdot e^{\frac L2}
\end{equation}
Note now that each element $(X,\phi)$ of $\BFX_g$, and thus of $\BFX_{g,\ell_0}$, determines not only the curve $\Lambda(X,\phi)$ but also a homotopy class of fillings for this curve, namely $\phi\circ\spine:\neigh(X)\to\Sigma$. Let now $Z\subset\BFX_{g,\ell_0}$ be the set of pairs $(X,\phi)$ so that the filling $\phi\circ\spine:\neigh(X)\to\Sigma$ is not unique in the sense that the curve $\Lambda(X,\phi)$ admits another non-homotopic genus $g$ filling and let
$$W=\BFX_{g,\ell_0}\setminus Z$$
be its complement. From Theorem \ref{bounding multi-fillings} we get that $Z$ consists of at most $\bconst\cdot L^{6g-5}\cdot e^{\frac L2}$ many elements and hence that $W$ is generic.

\begin{claim*}
If $\ell_0$ is over some threshold we have  that $\Lambda(W)\subset\BFB_g$.
\end{claim*}
\begin{proof}
First we have by construction that the curve $\Lambda(X,\phi)$ admits the genus $g$ filling $\phi\circ\spine_X:\neigh(X)\to\Sigma$. Suppose that it admits a smaller genus filling. Then, adding handles and mapping them to points we get that $\Lambda(X,\phi)$ admits a non-$\pi_1$-injective genus $g$ filling. Since on the other hand the filling $\phi\circ\spine_X:\neigh(X)\to\Sigma$ is $\pi_1$-injective as long as $\ell_0$ is over some threshold, we get that $\Lambda(X,\phi)$ admits two non-homotopic genus $g$ fillings, contradicting the assumption that $(X,\phi)\in W$.
\end{proof}

It remains to prove that the restriction of $\Lambda$ to the generic set $W$ is injective.

Well, suppose that we have $(X,\phi),(X',\phi')\in W$ with $\Lambda(X,\phi)=\Lambda(X',\phi')$. Since $(X,\phi)$ and $(X',\phi')$ belong to $W$ we know that the two fillings 
$$\phi\circ\spine_X:\neigh(X)\to\Sigma\text{ and }\phi'\circ\spine_{X'}:\neigh(X')\to\Sigma$$ 
are homotopic. Recall that this means that there is a homeomorphism
$$\sigma:\neigh(X)\to\neigh(X')$$
with $\phi'\circ\spine_{X'}\circ\sigma$ homotopic to $\phi\circ\spine_X$. Since $X$ and $X'$ are spines of $\neigh(X)$ and $\neigh(X')$ we deduce that there is a homotopy equivalence $\bar\sigma:X\to X'$ such that $\phi'\circ\bar\sigma$ is homotopic to $\phi$. Now, since the lengths of both $\phi(X)$ and $\phi'(X')$ are, up to a constant, basically half the length of $\Lambda(X,\phi)=\Lambda(X',\phi')$ we see that $\phi$ and $\phi'$ satisfy the conditions in Proposition \ref{lemma new}. It thus follows that there is a homeomorphism $F:X'\to X$ mapping edges at constant velocity, with $F\circ\bar\sigma$ homotopic to the identity, and with $\phi\circ F$ homotopic to $\phi'$. Now, both 
$$\phi\circ F,\phi':X\to\Sigma$$
are critical realizations and both are homotpic to each other. We get then from (1) in Lemma \ref{lem component critical realization} that $\phi'=\phi\circ F$. To conclude, note that since $F$ is a homotopy inverse of $\bar\sigma$ and since $\bar\sigma$ is induced by the homeomorphism $\sigma:\neigh(X')\to\neigh(X)$ we get that $F$ is a fat graph homeomorphism. This proves that $(X,\phi)$ and $(X',\phi')$ are equivalent, and hence that the restriction of $\Lambda$ to $W$ is injective. We are done with Lemma \ref{lem pi injective}.
\end{proof}

\begin{bem}
Note that \eqref{eq sabaton} and Lemma \ref{lem pi injective} imply that 
\begin{equation}\label{eq more sabaton}
\vert\BFB_g(L)\vert>\bconst\cdot L^{6g-4}\cdot e^{\frac L2}
\end{equation}
for all $L$ large enough.
\end{bem}

Our next goal is to show that the image of $\Lambda$ contains a large subset of $\BFB_g$.

\begin{lem}\label{lem pi surjective}
There is a generic subset of $\BFB_g$ which is contained in the image of $\Lambda$.
\end{lem}
\begin{proof}
Let $Z\subset\BFB_g$ be the set of those geodesics which admits a hyperbolic genus $g$ filling $\beta:S\to\Sigma$ such that the $\epsilon_0$-thin part of the double $DS$ of $S$ has at most $6g-4$ connected components $U$ through with $\iota(U,\D S) = 2$. It follows from Proposition \ref{prop this is a pain in the butt} that $Z$ has at most $\bconst\cdot L^{6g-5}\cdot e^{\frac L2}$ elements with length $\le L$. It follows from \eqref{eq more sabaton} that $Z$ is negligible and hence that its complement $W=\BFB_g\setminus Z$ is generic. We claim that $W$ is contained in the image of $\Lambda$, at least if $\epsilon_0$ is chosen small enough.

Well, each $\gamma\in W$ admits a hyperbolic filling $\beta:S\to\Sigma$ such that the $\epsilon_0$-thin part of the double $DS$ has at least $6g-3$ connected components $U$ with $\iota(\D S,U)=2$. Since $DS$ has genus $2g$ we have that  $6g-3$ is actually the maximal number of connected components that its thin part can have. It follows that the double $DS$ of $S$ admits a pants decomposition consisting of very short curves and that moreover $\D S$ cuts each one of them exactly twice. It follows that $S$ has $6g-3$ orthogeodesics cutting $S$ into a union of $4g-2$ right-angle hexagons. These hexagons have 3 alternating sides which are extremely short. It follows that each one of the hexagons contains a pretty large compact set which is almost isometric to a large neighborhood of the center of an ideal triangle in $\BH^2$. Declare the center of the hexagon to be the image of the center of the ideal triangle by this almost isometric map. Now, we can represent the dual graph of the decomposition of $S$ by our short orthogeodesic segments as a geodesic subgraph $X$ of $S$ with vertices in the centers of the hexagons, see Figure \ref{fig:hexagons}. 

The graph $X$ is a spine of $S$ and hence it inherits from $S$ a fat graph structure with $S$ as its neighborhood. Let now $\psi:X\to\Sigma$ be the realization obtained from the restriction of $\beta$ to $X$ by pulling the edges tight. 

\begin{figure}[h]
\begin{center}
\AffixLabels{\centerline{\includegraphics[width=0.8\textwidth]{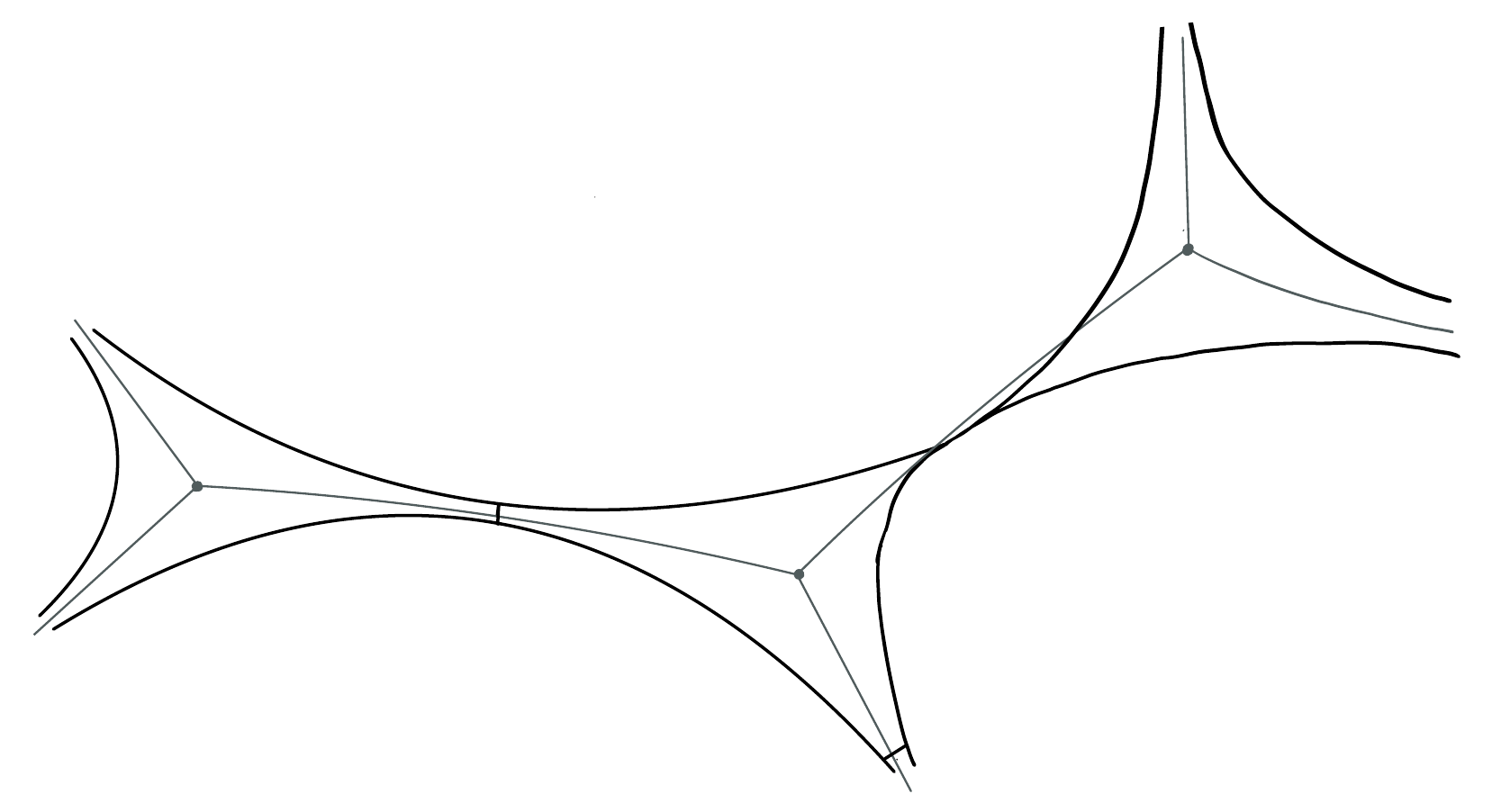}}\hspace{1cm}}
\vspace{-24pt}
\end{center}
\caption{Getting an almost critical realization out of a filling whose domain can be cut into hexagons by very short orthogeodesics.} 
\label{fig:hexagons}
\end{figure}

\begin{claim*}
For all $\delta$ there is $\epsilon$ such that if $\epsilon_0<\epsilon$ then the realization $\psi:X\to\Sigma$ is $\delta$-critical.
\end{claim*}
\begin{proof}
Suppose that the claim fails to be true. This means that for some $\delta$ there are a sequence of counter examples with $\epsilon_0\to 0$. Since $\epsilon_0\to 0$ we get that the images of the hexagons converge to a 1-Lipschitz map of a geodesic triangle into $\Sigma$ which however maps the boundary geodesics to geodesics. Such a map is an isometric embedding of the triangle in question. Now, in a geodesic triangle, the geodesic rays starting in the center and pointing into the cusps make angle $\frac{2\pi}3$. This means that the angles in the approximates converge to $\frac{2\pi}3$ contradicting our assumption that some of them were at least $\delta$ off $\frac{2\pi}3$. We have proved the claim.
\end{proof}

It follows thus from the claim and Corollary \ref{lem critical near kind of critical} that, as long as $\epsilon_0$ is under some threshold, the realization $\psi:X\to\Sigma$ is homotopic to a critical realization $\phi:X\to\Sigma$. This implies $(X,\phi)$ belongs to the domain of $\Lambda$ and that $\gamma=\Lambda(X,\phi)$. We have proved that the generic set $W\subset\BFB_g$ is contained in the image of $\Lambda$.
\end{proof}

We are now ready to wrap all of this up. 

\subsection*{Proof of Theorem \ref{thm counting curves}}

Combining Lemma \ref{lem surface has right length}, Lemma \ref{lem pi injective} and Lemma \ref{lem pi surjective} we get that for all $\delta>0$ there is a generic subset $W\subset\BFX_g$ which is mapped injectively under $\Lambda$ to a generic subset of $\BFB_g$ in such a way that
$$\left\vert\ell_\Sigma(\Lambda(X,\phi))-2\ell_\Sigma(\phi)+2\kappa\right\vert\le\delta$$
where 
\begin{equation}\label{eq kappa}
\kappa=-3\chi(X)\cdot\log\frac 43=\log\left(\left(\frac 43\right)^{6g-3}\right)
\end{equation}
It follows that for all $\delta>0$ our set $\BFB_g(L)$ has, for $L\to\infty$, at least as many elements as $\BFX_g(\frac L2+\kappa-\delta)$ and at most as many as $\BFX_g(\frac L2+\kappa+\delta)$. In symbols this just means that
\begin{equation}\label{eq glasses}
\left\vert\BFX_g\middle(\frac L2+\kappa-\delta\middle)\right\vert\preceq\vert\BFB_g(L)\vert\preceq\left\vert\BFX_g\middle(\frac L2+\kappa+\delta\middle)\right\vert
\end{equation}
for large $L$. It thus remains to estimate the cardinality of $\BFX_g(L)$.

\begin{lem}\label{lem vdovina}
We have 
$$\vert\BFX_g(L)\vert\sim \frac {1}{12^g}\cdot\left(\frac 32\right)^{6g-3}\cdot\frac{1}{g!\cdot(3g-2)!\cdot\vol(T^1\Sigma)^{2g-1}}\cdot L^{6g-4}\cdot e^{L}$$
as $L\to\infty$.
\end{lem}
\begin{proof}
From Theorem \ref{thm counting minimising graphs} we get that every trivalent graph $X$ has
$$\vert\BFG^X(L)\vert\sim 
\left(\frac 23\right)^{3\chi(X)}\cdot\frac{\vol(T^1\Sigma)^{\chi(X)}}{(-3\chi(X)-1)!}\cdot L^{-3\chi(X)-1}\cdot e^{L}$$
critical realizations of length at most $L$ in $\Sigma$. This implies that whenever $X$ is fat graph of genus $g$ then asymptotically there are 
$$\frac 1{\vert\Aut(X)\vert}\left(\frac 23\right)^{3\chi(X)}\cdot\frac{\vol(T^1\Sigma)^{\chi(X)}}{(-3\chi(X)-1)!}\cdot L^{-3\chi(X)-1}\cdot e^{L}$$
many elements in $\BFX_g(L)$ represented by $(X,\phi)$ for some critical $\phi:X\to\Sigma$ of length $\ell(\phi)\le L$. Adding over all possible types of genus $g$ fat graphs we get that
$$\vert\BFX_g(L)\vert\sim\sum_X\frac 1{\vert\Aut(X)\vert}\left(\frac 23\right)^{3\chi(X)}\cdot\frac{\vol(T^1\Sigma)^{\chi(X)}}{(-3\chi(X)-1)!}\cdot L^{-3\chi(X)-1}\cdot e^{L}$$
From the Bacher-Vdovina \cite{Bacher-Vdovina} result mentioned earlier and taking into consideration that $\chi(X)=1-2g$ we get
$$\vert\BFX_g(L)\vert\sim\frac 2{12^g}\cdot\frac{(6g-5)!}{g!\cdot(3g-3)!}\cdot\left(\frac 32\right)^{6g-3}\cdot\frac{\vol(T^1\Sigma)^{1-2g}}{(6g-4)!}\cdot L^{6g-4}\cdot e^{L}$$
The claim follows now from elementary algebra.
\end{proof}

Now, from Lemma \ref{lem vdovina} we get that 
$$\left\vert\BFX_g\left(\frac L2\right)\right\vert\sim \frac {1}{12^g}\cdot\left(\frac 32\right)^{6g-3}\cdot\frac{1}{g!\cdot(3g-2)!\cdot\vol(T^1\Sigma)^{2g-1}}\cdot\frac{L^{6g-4}}{2^{6g-4}}\cdot e^{\frac L2}.$$
Taking into account that 
$$\left\vert\BFX_g\left(\frac L2+\kappa\right)\right\vert\sim\left\vert\BFX_g\left(\frac L2\right)\right\vert\cdot e^\kappa=\left\vert\BFX_g\left(\frac L2\right)\right\vert\cdot \left(\frac 43\right)^{6g-3}$$
we get that
$$\left\vert\BFX_g\left(\frac L2+\kappa\right)\right\vert
\sim\frac {2}{12^g\cdot g!\cdot(3g-2)!\cdot\vol(T^1\Sigma)^{2g-1}}\cdot L^{6g-4}\cdot e^{\frac L2}$$
It follows thus from \eqref{eq glasses} that 
$$\vert\BFB_g(L)\vert\sim \frac {2}{12^g\cdot g!\cdot(3g-2)!\cdot\vol(T^1\Sigma)^{2g-1}}\cdot L^{6g-4}\cdot e^{\frac L2}$$
as we wanted to show. This concludes the proof of Theorem \ref{thm counting curves}.\qed

\section{Curves bounding immersed surfaces}\label{sec immersed}
We now turn our attention to Theorem \ref{thm counting curves immersion}:

\begin{named}{Theorem \ref{thm counting curves immersion}}
Let $\Sigma$ be a closed, connected, and oriented hyperbolic surface and for $g\ge 1$ and $L>0$ let $\BFB_g(L)$ be as in \eqref{eq set we want to count}. We have 
$$\vert\{\gamma\in\BFB_g(L)\text{ bounds immersed surface of genus }g\}\vert\sim \frac 1{2^{4g-2}}\vert\BFB_g(L)\vert$$
as $L\to\infty$.
\end{named}

Denote by 
$$\BFB_g^{\imm}(L)=\{\gamma\in\BFB_g(L)\text{ bounds immersed surface of genus }g\}$$
the set we want to count. From the proof of Theorem \ref{thm counting curves} we get that 
$$\vert\BFB_g^{\imm}(L)\vert\sim\left\vert\BFX_g^{\imm}\left(\frac L2+\kappa\right)\right\vert$$
where $\kappa$ is as in \eqref{eq kappa} and where 
$$\BFX_g^{\imm}(L)\stackrel{\text{def}}=\left\{(X,\phi)\in\BFX_g(L)\,\middle\vert\begin{array}{c}\text{the realization }\phi:X\to\Sigma\text{ extends}\\ \text{to an immersion of the thickening}\\ \neigh(X)\text{ of }X\end{array}\right\}.$$
Equivalently, $(X,\phi)\in\BFX_g$ belongs to $\BFX_g^{\imm}$ if the cyclic ordering at each vertex of $X$ agrees with the one pulled back from $\Sigma$ via $\phi$, that is the one coming from the orientation of $\Sigma$. We can refer to such realizations of a fat graph as {\em fat realizations}. With this language we have that
$$\BFX_g^{\imm}(L)=\left\{(X,\phi)\in\BFX_g(L)\,\middle\vert\,\phi\text{ is a fat realization of the fat graph }X\right\}.$$
For a given trivalent fat graph $X$ let
$$\BFG^{X}_{\imm}(L)=\{\phi\in\BFG^X(L)\,\vert\,\phi\text{ is a fat realization of }X\}$$
be the set of fat critical realizations of $X$ of total length at most $L$. Note that
$$\vert\BFX_g^{\imm}(L)\vert=\sum_X\frac 1{\vert\Aut(X)\vert}\vert\BFG^X_{\imm}(L)\vert.$$
What is still missing to be able to run the proof as in that of Theorem \ref{thm counting curves} is a version of Theorem \ref{thm counting minimising graphs} for $\BFG^X_{\imm}$. Well, here it is:

\begin{sat}\label{thm counting minimising graphs fat}
Let $\Sigma$ be a closed, connected, and oriented hyperbolic surface. For every connected trivalent fat graph $X$ we have
$$\vert\BFG^X_{\imm}(L)\vert\sim 
2^{2\chi(X)}\cdot\left(\frac 23\right)^{3\chi(X)}\cdot\frac{\vol(T^1\Sigma)^{\chi(X)}}{(-3\chi(X)-1)!}\cdot L^{-3\chi(X)-1}\cdot e^{L}$$
as $L\to\infty$.
\end{sat}

Assuming Theorem \ref{thm counting minimising graphs fat} for the moment we get from $\chi(X)=1-2g$ and from the Bacher-Vdovina theorem that 
$$\vert\BFX^{\imm}_g(L)\vert\sim\frac 1{2^{4g-2}}\cdot\frac 2{12^g}\cdot\frac{(6g-5)!}{g!\cdot(3g-3)!}\cdot\left(\frac 32\right)^{6g-3}\cdot\frac{\vol(T^1\Sigma)^{1-2g}}{(6g-4)!}\cdot L^{6g-4}\cdot e^{L}$$
from where we get, as in the proof of Theorem \ref{thm counting curves}, that
$$\vert\BFB_g^{\imm}(L)\vert\sim \frac 1{2^{4g-2}}\frac {2}{12^g\cdot g!\cdot(3g-2)!\cdot\vol(T^1\Sigma)^{2g-1}}\cdot L^{6g-4}\cdot e^{\frac L2}$$
The claim of Theorem \ref{thm counting curves immersion} follows now from this statement combined with Theorem \ref{thm counting curves}.
\medskip

All that is left to do is to prove Theorem \ref{thm counting minimising graphs fat}. Since it is basically identical to the proof of Theorem \ref{thm counting minimising graphs} we just point out the differences. The key is to obtain a fat graph version of Proposition \ref{prop critical in box}. In a nutshell, the idea of the proof of this proposition was that 
\begin{enumerate}
\item we could compute the volume $\vol(\CG_{\epsilon-\crit}^X(\vec L,h))$ of the set of $\epsilon$-critical realizations of $X$ whose edge lengths were in a box, and
\item we knew that every connected component contributes the same amount, and how much.
\end{enumerate}
Let thus $\CG_{\epsilon-\crit-\imm}^X(\vec L,h)\subset\CG_{\epsilon-\crit}^X(\vec L,h)$ be those $\epsilon$-critical realizations in our box which preserve the fat structure. Recalling now that by Proposition \ref{prop sum up section 5} the connected components of $\CG_{\epsilon-\crit}^X(\vec L,h)$ have small diameter, we deduce that the induced fat structure is constant over each such connected component. It follows that $\CG_{\epsilon-\crit-\imm}^X(\vec L,h)$ is a union of connected components of $\CG_{\epsilon-\crit}^X(\vec L,h)$. In particular, to be able to obtain a fat graph version of Proposition \ref{prop critical in box} we just need to be able to compute $\vol(\CG_{\epsilon-\crit-\imm}^X(\vec L,h))$. 

Now, in the proof of Proposition \ref{prop critical in box}, the key ingredient of the computation of the volume of $\CG_{\epsilon-\crit}(\vec L,h)$ was Corollary \ref{kor fat graph}---we remind the reader that the statement of the said corollary was that for any $\vec x\in\Sigma^{\ver X}$ we have
$$\vert\BFG_{\vec x,\epsilon-\crit}^X(\vec L,h)\vert \sim \epsilon^{4\vert\chi(X)\vert}\cdot\left(\frac{2}{3}\right)^{2\chi(X)}\cdot\pi^{\chi(X)}\cdot\frac{(e^h-1)^{-3\chi(X)}\cdot e^{\Vert\vec L\Vert}}{\vol(\Sigma)^{-3\chi(X)}}$$
as $\min_{e\in\edg(X)}L_e\to\infty$, where $\BFG^X_{\vec x,\epsilon-\crit}(\vec L,h)$ is the set of $\epsilon$-critical realizations $\phi:X\to \Sigma$ mapping the vertex $v$ to the point $x_v=\phi(v)$. As we see, if we want to just copy line-by-line the computation of the volume of $\CG_{\epsilon-\crit}(\vec L,h)$ to get the volume of $\CG_{\epsilon-\crit-\imm}(\vec L,h)$ what we need to know is the number of elements in the set $\BFG^X_{\vec x,\epsilon-\crit-\imm}(\vec L,h)$ of $\epsilon$-critical realizations $\phi:X\to \Sigma$ mapping the vertex $v$ to the point $x_v=\phi(v)$ and preserving the fat structure.

Now, to obtain the number of elements in $\BFG_{\vec x,\epsilon-\crit}^X(\vec L,h)$ what we did was to invoke Theorem \ref{sat delsarte graph} and compute the volume of the set 
$$U_{\vec x,\epsilon-\crit}^X\subset\prod_{v\in\ver(X)}\left(\bigoplus_{\bar e\in\hal_v(X)} T^1_{x_v}\Sigma\right)$$
of those tuples $(v_{\vec e})_{\vec e\in\hal(X)}$ with $\angle(v_{\vec e_1},v_{\vec e_2})\in[\frac{2\pi}3-\epsilon,\frac{2\pi}3+\epsilon]$ for all distinct $\vec e_1,\vec e_2\in\hal(X)$ incident to the same vertex. Accordingly, to compute the number of elements in $\BFG_{\vec x,\epsilon-\crit\imm}^X(\vec L,h)$ we need to compute the volume of the set 
$$U_{\vec x,\epsilon-\crit-\imm}^X\subset U_{\vec x,\epsilon-\crit}^X$$
consisting of tuples such that for any vertex $v\in \ver(X)$ the cyclic order of the half-edges incident to $v$ agrees with the one of the corresponding unit tangent vectors. Following the computation of $\vol(U_{\vec x,\epsilon=\crit})$ we get that
$$\vol(U_{\vec x,\epsilon-\crit-\imm}^X)=\frac 1{2^{\vert\ver(X)\vert}}\vol(U_{\vec x,\epsilon-\crit}^X)=2^{2\chi(X)}\cdot\vol(U_{\vec x,\epsilon-\crit}^X)$$
As we just discussed, this implies that 
$$\vert\BFG_{\vec x,\epsilon-\crit-\imm}^X(\vec L,h)\vert \sim 2^{2\chi(X)}\cdot \vert\BFG_{\vec x,\epsilon-\crit}^X(\vec L,h)\vert$$
and hence that
$$\vol(\CG_{\epsilon-\crit-\imm}^X(\vec L,h))\sim 2^{2\chi(X)}\cdot\vol(\CG_{\epsilon-\crit}^X(\vec L,h))$$
and thus that
$$\vert\BFG^X_{\imm}(L)\vert\sim 2^{2\chi(X)}\cdot\vert\BFG^X(L)\vert.$$
Theorem \ref{thm counting minimising graphs fat} follows now from Theorem \ref{thm counting minimising graphs}. \qed

\section{Comments}\label{sec comments}

In this section we discuss where and how much we use the assumption that $\Sigma$ is a closed orientable surface.
\medskip

\subsection*{First, do the results here apply if we replace $\Sigma$ by a compact 2-dimensional orbifold $\CO=\Gamma\bs\BH^2$?} 

The answer is yes for Theorem \ref{thm counting minimising graphs}, with exactly the same proof. One should just do everything equivariantly. For example, a realization in $\CO$ of a graph $X$ should be a map $\tilde\phi:\tilde X\to\BH^2$ from the universal cover of $X$ to $\BH^2$ which is equivariant under a homomorphism $\tilde\phi_*:\pi_1(X)\to\Gamma$, and where two such pairs $(\tilde\phi,\tilde\phi_*)$ and $(\tilde\psi,\tilde\psi_*)$ are identified if they differ by an element of $\Gamma$. Once we rephrase the situation in those terms, everything extends in the obvious way. For example, the space $\CG^X$ of realizations of a graph in $\CO$ is now an orbifold: the map $\CG^X\to\CO^{\ver X}$ sending each realization to the images of the vertices is a covering in the category of orbifolds. 

On the other hand, to prove Theorem \ref{thm counting curves} one needs to be a little bit careful because in Section \ref{sec fillings} and Section \ref{sec fillings2} we used repeatedly that every sufficiently short curve in $\Sigma$ is homotopically trivial, and this is no longer true if we are working in $\CO$.

\subsection*{What about allowing $\Sigma$ to have cusps?}
Here we again have problems with the discussion in Section \ref{sec fillings} and Section \ref{sec fillings2}, but this time it is much worse. In some sense, the results of Section \ref{sec fillings} and Section \ref{sec fillings2} are just generalizations of Lemma \ref{lem bounding the rest}, and this lemma fails in the presence of cusps: indeed, suppose that $\Sigma$ is a once punctured torus and $X$ is a graph with two vertices $x$ and $x'$ (if you wish you can make $X$ trivalent while essentially keeping the same reasoning as we will present, but doing so might obscure things slightly) and with $6$ edges $f_1,f_2,e_1,e_2,e_3$ and $h$ such that
\begin{itemize}
\item the edges $f_1,f_2$ are incident on both ends to $x$,
\item the edges $e_1,e_2,e_3$ are incidents on both ends to $x'$, and
\item the edge $h$ runs from $x$ to $x'$.
\end{itemize}
Fix now a horospherical neighborhood of the cusps in $\Sigma$, fix a point ${\bf x}_0$ and let ${\bf x}_t$ be the point at distance $t$ of ${\bf x}_0$ along the ray pointing directly into the cusp. Now, if we are given a vector $\vec L=(F_1,F_2,E_1,E_3,E_3,H)$ with positive real coefficients consider realizations $\phi:X\to\Sigma$ with $\phi(x)={\bf x}_0$, with $\phi(x')={\bf x}_H$, and with $\phi(h)$ equal to the segment of length $H$ joining ${\bf x}_0$ and ${\bf x}_H$. Now, we have $\bconst e^{F_1+F_2}$ choices for the images of $f_1$ and $f_2$ subject to the restriction that $\phi(f_i)$ is for $i=1,2$ a geodesic segment of length at most $F_i$. Note also that the horospherical simple loop based at ${\bf x}_H$ has length $\bconst\cdot e^{-H}$ and that geodesic loops in the cusp, based at ${\bf x}_H$ and with length $\ell$ are homotopic to horospherical segments of length at most $\bconst\cdot e^{\frac\ell 2}$. This implies that, if we want to map $e_i$ to a loop in the cusp and of length at most $E_i$, then we have at least $\bconst\cdot e^{\frac 12E_i}\cdot e^H$ choices. Altogether we have at least
$$\bconst\cdot e^{F_1+F_2+\frac 12E_1+\frac 12 E_2+\frac 12 E_3+3H}$$
choices of (homotopy classes of) realizations of $X$ into $\Sigma$ with $\ell(\phi(f_1))\le F_1$, $\ell(\phi(f_2))\le F_2$, $\ell(\phi(e_1))\le E_1$, $\ell(\phi(e_2))\le E_2$ $\ell(\phi(e_3))\le E_3$, $\ell(\phi(H))\le H$. In particular if we set 
$$\vec L=(F_1,F_2,E_1,E_2,E_3,H)=(n,n,\frac 12n,\frac 12n,\frac 12 n,\frac 52n)$$
we have at least $\bconst\cdot e^{\frac{61}4n}$ such realizations, and this is a much larger number than $\bconst e^{6n}=\bconst e^{\Vert L\Vert}$. This proves that the analogue of Lemma \ref{lem bounding the rest} fails if $\Sigma$ is not compact.

In summary, if $\Sigma$ has finite volume but is not compact, then we do not even know whether Theorem \ref{thm counting minimising graphs} holds. Although we suspect that the answer is yes. 

\subsection*{Can $\Sigma$ be non-orientable?} If $\Sigma$ is a compact hyperbolic surface which however is non-orientable then we know that both Theorem \ref{thm counting minimising graphs} and Theorem \ref{thm counting curves} hold true: we never used orientability during their proofs. We did however in the proof of Theorem \ref{thm counting curves immersion}: unless $\Sigma$ is oriented it makes little sense to speak about the induced fat graph structure. We do not know what happens with Theorem \ref{thm counting curves immersion} when the ambient surface is not oriented.

\subsection*{And what about higher dimensions?}
If we replace $\Sigma$ by a closed hyperbolic manifold of other dimension than $2$, then Theorem \ref{thm counting minimising graphs} and Theorem \ref{thm counting curves} should still hold, and with proofs which, if not identical, keep the same spirit. It is however less clear whether there should be an interesting analogue of Theorem \ref{thm counting curves immersion}: at least if $\dim\ge 5$, where every map of a surface can be deformed to an embedding.

\begin{appendix}

\section{The geometric prime number theorem}\label{sec huber}
The argument we used to prove Theorem \ref{thm counting minimising graphs} can be used to recover Huber's geometric primer number theorem, and this is what we do here. Besides giving a simple proof of this theorem, it might help the reader understand the logic of the proof of Theorem \ref{thm counting minimising graphs}.

\begin{named}{The Geometric Prime Number Theorem}[Huber]
Let $\Sigma$ be a closed, connected, orientable, hyperbolic surface and let $\BFC(L)$ be the set of closed non-trivial oriented geodesics in $\Sigma$ of length at most $L$. We have
$$\vert\BFC(L)\vert\sim\frac{e^L}L$$ 
as $L\to\infty$.
\end{named}

\begin{bem}
In Section \ref{sec comments} we discussed some of the difficulties one would face when extending the main results to the case that $\Sigma$ is a finite volume surface or an orbifold. None of these problems really arise when proving the geometric prime number theorem, and indeed the proof we present here works with only minimal changes if $\Sigma$ is replaced by an arbitrary finite are orbifold $\Gamma\bs\BH^2$. We decided to just deal with the compact case to avoid hiding the structure of the argument.
\end{bem}

The proof of Huber's theorem would be cleaner and nicer if all closed geodesics were primitive. Luckily, this is almost true. Indeed it follows for example from the work of Coornaert and Knieper \cite{Coornaert-Knieper} that there is some $C>0$ with
\begin{equation}\label{eq coarse bound for geodesics}
\frac 1C\cdot\frac{e^L}L\le\vert\BFC(L)\vert\le C\cdot e^L
\end{equation}
for all $L>0$. We stress that the Coornaert-Knieper argument is pretty coarse: what they actually prove is a statement about groups acting isometrically, discretely and cocompactly on Gromov-hyperbolic spaces. 

The point for us is that \eqref{eq coarse bound for geodesics} implies that most geodesics of length at most $L$ are primitive. Indeed, every non-primitive geodesic of length at most $L$ is a multiple of a primitive geodesic of length at most $\frac 12 L$. On the other hand, if $s_0$ is the systole of $\Sigma$ then at most $\frac{L}{s_0}$ geodesics of length at most $L$ arise as multiples of any given geodesic. These two observations, together with the right side of \eqref{eq coarse bound for geodesics}, imply that there are at most $\frac L{s_0}\BFC(\frac 12L)\le \frac C{s_0}\cdot L\cdot e^{L/2}$ non-primitive geodesics of length at most $L$. Taking the left side of \eqref{eq coarse bound for geodesics} into consideration we deduce that, as we had claimed above, most geodesics are primitive.

\begin{lem}\label{lem most primitive}
Fix $h$ and let $\BFC(L,h)$ and $\BFP(L,h)$ be, respectively, the sets of all geodesics and of all primitive geodesics of length in $[L,L+h]$. Then we have
$$\vert\BFC(L,h)\vert\sim\vert\BFP(L,h)\vert$$
as $L\to\infty$.\qed
\end{lem}

After this preparatory comment let us start the real business. Let $\CL$ be the space of all geodesic loops in $\Sigma$ and consider the map $\Pi:\CL\to\Sigma$ mapping each loop to its base point. As was the case for more general graphs, the map $\Pi$ is a covering, meaning that when we pull-back the hyperbolic metric using $\Pi$ we can think of $\CL$ as being a hyperbolic surface. Note also that the set of connected components of $\CL$ agrees with the set of free homotopy classes of loops. It follows that for every closed geodesic $\gamma$ we have a connected component $\CL^\gamma$.

Now, given $\epsilon>0$ small let $\CL_\epsilon$ be the set of all geodesic loops with angle defect at most $\epsilon$, that is, the set of geodesic loops whose initial and terminal velocity vectors meet with unoriented angle in $[0,\epsilon]$. For $L>0$ let $\CL_\epsilon(L,L+h)$ be the elements in $\CL_\epsilon$ with length between $L$ and $L+h$. Accordingly, set $\CL^\gamma_\epsilon(L,L+h)=\CL^\gamma\cap\CL_\epsilon(L,L+h)$. 

Let us establish some basic properties of $\CL_\epsilon(L,L+h)$ and $\CL_\epsilon^\gamma(L,L+h)$:

\begin{lem}\label{lem ap2}
Fix $h>0$ and $\epsilon>0$. We have that
$$\vol(\CL_\epsilon(L,L+h))\sim \epsilon\cdot(e^{L+h}-e^L)\text{ as }L\to\infty.$$
Moreover, there is a function $C(\epsilon)$ with $\lim_{\epsilon\to 0}C(\epsilon)=1$ such that for every sufficiently long closed geodesic $\gamma$ we have that
\begin{enumerate}
\item $\CL_\epsilon^\gamma(L,L+h)=\emptyset$ unless $\ell(\gamma)\in[L-\epsilon,L+h]$,
\item $\vol(\CL_\epsilon^\gamma(L,L+h))\le C(\epsilon)\cdot\epsilon\cdot L$, and
\item $C(\epsilon)^{-1}\le\frac 1{\epsilon\cdot L}\vol(\CL_\gamma^\epsilon(L,L+h))\le C(\epsilon)$ if $\gamma$ is primitive and $\ell(\gamma)\in[L,L+h-\epsilon]$.
\end{enumerate}
\end{lem}
\begin{proof}
As in the proof of Proposition \ref{prop critical in box} we exploit the cover $\Pi:\CL\to\Sigma$ to compute the volume of $\CL_\epsilon(L,L+h)$:
$$\vol(\CL_\epsilon(L,L+h))=\int_\Sigma\vert \Pi^{-1}(x)\cap\CL_\epsilon(L,L+h)\vert dx$$
Now, $\Pi^{-1}(x)\cap\CL_\epsilon(L,L+h)$ is nothing other than the number of geodesic arcs going from $x$ to $x$ with length within $[L,L+h]$ and such that the initial and terminal velocities make at most angle $\epsilon$. Once we fix $x$, the set of admissive pairs of initial and terminal velocities is a subset $T^1_x\Sigma\times T^1_x\Sigma$ with volume $4\pi\epsilon$. We thus get from Theorem \ref{sat delsarte graph} that 
$$\vert\Pi^{-1}(x)\cap\CL_\epsilon(L,L+h)\vert\sim\epsilon\cdot\frac{e^{L+h}-e^L}{\vol(\Sigma)}$$
Since we are assuming that $\Sigma$ is closed, this is uniform in $x$, meaning that we get
$$\vol(\CL_\epsilon(L,L+h))\sim\int_\Sigma\epsilon\cdot\frac{e^{L+h}-e^L}{\vol(\Sigma)}=\epsilon\cdot(e^{L+h}-e^L)$$
We have proved the first claim.

\begin{figure}[h]
\leavevmode \SetLabels
\L(.47*.6) $d$\\%
\L(.39*.83) $\gamma_x$\\%
\L(.45*.93) $x$\\%
\endSetLabels
\begin{center}
\AffixLabels{\centerline{\includegraphics[width=0.3\textwidth]{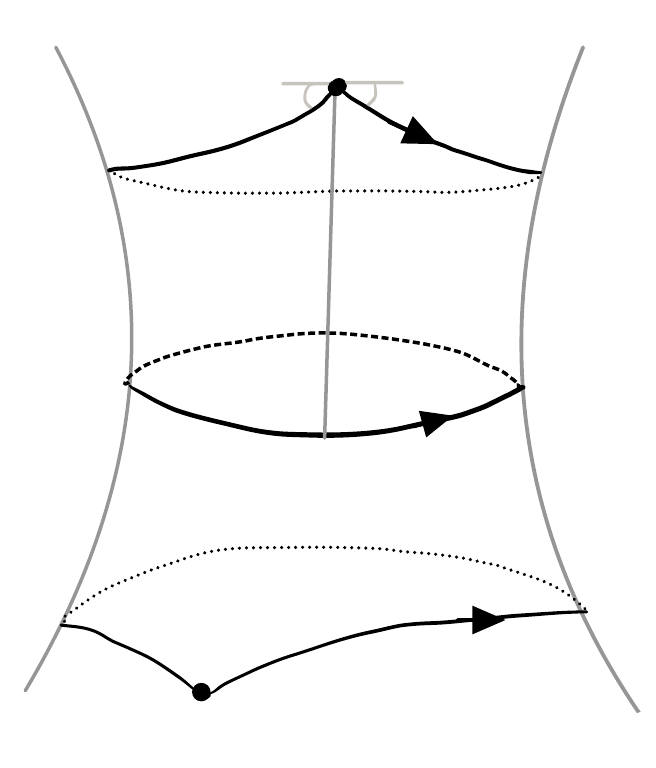}}\hspace{1cm}}
\vspace{-24pt}
\end{center}
\caption{The cylinder $\langle\gamma\rangle\bs\BH^2$. The marked angles each measures one half of the angle defect of $\gamma_x$.} 
\label{fig:cylinder}
\end{figure}


To prove the rest, let us give a concrete description of $\CL^\gamma$. Well, writing $\gamma=\eta^k$ for some $\eta$ primitive and $k\ge 1$ consider the cover 
$$\pi_\eta:\langle\eta\rangle\bs\BH^2\to\Sigma$$
of $\Sigma$ corresponding to $\eta$. Now, for each $x\in\langle\eta\rangle\bs\BH^2$ there is a unique hyperbolic loop $\gamma_x$ based at $x$ and freely homotopic to running $k$ times over $\eta$. The basic observation is that the map
$$\langle\eta\rangle\bs\BH^2\to\CL^\gamma,\ x\mapsto\pi_\eta\circ\gamma_x$$
is an isometry between $\langle\eta\rangle\bs\BH^2$ and the connected component $\CL^\gamma$. 

Now, if one denotes by $d$ the distance in $\langle\eta\rangle\bs\BH^2$ between $x$ and the central geodesic (see Figure \ref{fig:cylinder}) one gets from formula 2.3.1(vi) in the final page in \cite{Buser} that 
\begin{align*} 
\frac 12\cdot(\text{angle defect of }\gamma_x)
&\sim\tan\left(\frac 12\cdot(\text{angle defect of }\gamma_x)\right)\\
&=\sinh(d)\cdot\tanh\left(\frac{\ell(\gamma)}2\right)\sim d
\end{align*}
where the asymptotics hold when $\ell(\gamma)$ is large and the angle defect is small. Now, claims (1), (2) and (3) follow from elementary considerations.
\end{proof}

Armed with these two lemmas we are ready to conclude the proof of the geometric prime number theorem. Note that it suffices to prove that for $h$ positive and fixed we have 
\begin{equation}\label{eq bob}
\vert\BFC(L,h)\vert\sim\frac{e^{L+h}-e^L}L
\end{equation}
as $L\to\infty$. Here $\BFC(L,h)$ is as in Lemma \ref{lem most primitive}.

Anyways, with notation as in both lemmas we have for $\epsilon$ positive and small and for $L$ large that
\begin{align*}
\vert\BFP(L,h)\vert
&\stackrel{\text{\ref{lem ap2}(3)}}\preceq\frac{C(\epsilon)}{\epsilon\cdot L}\vol\left(\cup_{\gamma\in\BFP(L,h)}\CL_\epsilon^\gamma(L,L+h+\epsilon)\right)\\
&\le\frac{C(\epsilon)}{\epsilon\cdot L}\vol(\CL_\epsilon(L,L+h+\epsilon))\\
&\stackrel{\text{\ref{lem ap2}}}\sim C(\epsilon)\cdot\frac{e^{L+h+\epsilon}-e^L}L
\end{align*}
On the other hand we also have that
\begin{align*}
\vert\BFC(L,h)\vert
&\stackrel{\text{\ref{lem ap2}(2)}}\ge\frac 1{C(\epsilon)\cdot\epsilon\cdot L}\vol\left(\cup_{\gamma\in\BFC(L,h)}\CL_\epsilon(L,h)\right)\\
&=\frac 1{C(\epsilon)\cdot\epsilon\cdot L}\vol(\CL_\epsilon(L,h))\\
&\stackrel{\text{\ref{lem ap2}}}\sim \frac 1{C(\epsilon)}\frac{e^{L+h}-e^L}L
\end{align*}
Since this holds for all $\epsilon$, and since $C(\epsilon)$ tends to $1$ when $\epsilon\to 0$ we have that
$$\vert\BFP(L,h)\vert\preceq\frac{e^{L+h}-e^L}L\preceq\vert\BFC(L,h)\vert$$
Now, Lemma \ref{lem most primitive} implies \eqref{eq bob}, and Bob's your uncle.

\end{appendix}

\end{document}